\documentclass[oneside,article]{memoir}

\usepackage{amsmath}         
\usepackage{amsthm,thmtools} 
\usepackage{enumitem}        

\usepackage{tikz}
\usetikzlibrary{matrix,arrows,decorations.pathreplacing,angles,quotes}
  
\usepackage{tikz-cd}

\usepackage[nobysame,alphabetic,initials]{amsrefs} 
\usepackage[hidelinks]{hyperref}                   
\usepackage[capitalize,nosort]{cleveref}           

\usepackage{amssymb}
\usepackage[mathscr]{euscript}
\usepackage{relsize}
\usepackage{stmaryrd}
\usepackage{stackengine}

\usepackage{indentfirst} 
\usepackage{float} 
\usepackage{multicol}
\usepackage{tabularx}
\usepackage{longtable}
\allowdisplaybreaks

\usepackage[nomargin,inline,marginclue,draft]{fixme} 
  \fxsetup{targetlayout=color}

\isopage[12]

\setlrmargins{*}{*}{1}
\checkandfixthelayout

\counterwithout{section}{chapter}

  \renewcommand{\equiv}{\simeq}

  \newcommand{\op}{{\mathord\mathrm{op}}}

  \newcommand{\ob}{\operatorname{ob}}

  \newcommand{\slice}{\mathbin\downarrow}

  \makeatletter
  \def\horn#1{\expandafter\horn@i#1,,\@nil}
  \def\horn@i#1,#2,#3\@nil{\Lambda^{#2}[#1]}
  \makeatother

  \renewcommand{\tilde}{\widetilde}

  \renewcommand{\hat}{\widehat}

  \newcommand{\cat}[1]{\mathscr{#1}}

  \newcommand{\from}{\colon}

  \declaretheorem[style=definition,within=section]{definition}
  \declaretheorem[style=definition,numberlike=definition]{notation}
  \declaretheorem[style=definition,numberlike=definition]{example}
  \declaretheorem[style=definition,numberlike=definition]{remark}
  \declaretheorem[style=definition,numberlike=definition]{construction}

  \declaretheorem[style=plain,numberlike=definition]{corollary}
  \declaretheorem[style=plain,numberlike=definition]{lemma}
  \declaretheorem[style=plain,numberlike=definition]{proposition}
  \declaretheorem[style=plain,numberlike=definition]{theorem}

  \declaretheorem[style=plain,numbered=no,name=Theorem]{theorem*}

  \Crefname{corollary}{Corollary}{Corollaries}
  \Crefname{definition}{Definition}{Definitions}
  \Crefname{notation}{Notation}{Notations}
  \Crefname{lemma}{Lemma}{Lemmas}
  \Crefname{proposition}{Proposition}{Propositions}
  \Crefname{remark}{Remark}{Remarks}
  \Crefname{theorem}{Theorem}{Theorems}
  \Crefname{construction}{Construction}{Constructions}

  \newlist{axioms}{enumerate}{1}
  \Crefname{axiomsi}{}{}

  \newenvironment{tikzeq*}
  {
    \begingroup
    \begin{equation*}
    \begin{tikzpicture}[baseline=(current bounding box.center)]
  }
  {
    \end{tikzpicture}
    \end{equation*}
    \endgroup
    \ignorespacesafterend
  }

  \tikzset
  {
    diagram/.style=
    {
      matrix of math nodes,
      column sep={4.3em,between origins},
      row sep={4em,between origins},
      text height=1.5ex,
      text depth=.25ex
    },
    over/.style={preaction={draw=white,-,line width=6pt}},
    every to/.style={font=\footnotesize},
    inj/.style={right hook->},
    surj/.style={-{Latex[open]}},
    cof/.style={>->},
    fib/.style={->>},
  }

  \DeclareFontFamily{U}{mathx}{\hyphenchar\font45}

  \DeclareFontShape{U}{mathx}{m}{n}{
    <5> <6> <7> <8> <9> <10>
    <10.95> <12> <14.4> <17.28> <20.74> <24.88>
    mathx10}{}

  \DeclareSymbolFont{mathx}{U}{mathx}{m}{n}

  \DeclareFontFamily{U}{mathb}{\hyphenchar\font45}

  \DeclareFontShape{U}{mathb}{m}{n}{
    <5> <6> <7> <8> <9> <10>
    <10.95> <12> <14.4> <17.28> <20.74> <24.88>
    mathb10}{}

  \DeclareSymbolFont{mathb}{U}{mathb}{m}{n}

  \DeclareMathAccent{\widebar}{0}{mathx}{"73}

  \DeclareMathSymbol{\Rsh}{\mathrel}{mathb}{"E9}

  \DeclareFontFamily{U}{MnSymbolA}{}

  \DeclareFontShape{U}{MnSymbolA}{m}{n}{
    <-6> MnSymbolA5
    <6-7> MnSymbolA6
    <7-8> MnSymbolA7
    <8-9> MnSymbolA8
    <9-10> MnSymbolA9
    <10-12> MnSymbolA10
    <12-> MnSymbolA12}{}

  \DeclareSymbolFont{MnSyA}{U}{MnSymbolA}{m}{n}

  \DeclareMathSymbol{\twoheaddownarrow}{\mathrel}{MnSyA}{27}

  \newcommand{\MSC}[1]{%
    \let\thempfn\relax
    \footnotetext[0]{2020 Mathematics Subject Classification: #1.}
  }

\newcommand{\Cat}{\mathsf{Cat}}
 
\newcommand{\Gpd}{\mathsf{Gpd}}
\newcommand{\Grp}{\mathsf{Grp}}
\newcommand{\Graph}{\mathsf{Graph}}
\newcommand{\Set}{\mathsf{Set}}
\newcommand{\Cov}{\mathsf{Cov}}

\newcommand{\V}[1]{{#1}_{V}} 
\newcommand{\E}[1]{{#1}_{E}} 

\newcommand{\gtimes}{\otimes} 
\newcommand{\gexp}[2]{\ensuremath{#1}^{\gtimes #2}} 
\newcommand{\ghom}[2]{\operatorname{hom}^{\gtimes}(#1, #2)} 

\newcommand{\gbox}{\sqsubset}

\author{Krzysztof Kapulkin \and Udit Mavinkurve}
\title{The fundamental group in discrete homotopy theory}
\date{\today}

\begin{document}

  \maketitle

  \begin{abstract}
    We develop a robust foundation for studying the fundamental group(oid) in A-homotopy theory, including: equivalent definitions and basic properties, the theory of covering graphs, and the discrete version of the Seifert--van Kampen theorem.
     \MSC{05C38 (primary), 55Q99, 05C63 (secondary)}
  \end{abstract}

  \setlist[enumerate]{label=(\arabic*)}

\section*{Introduction}

Discrete homotopy theory is an emerging area of mathematics concerned with using homotopy-theoretic techniques to study combinatorial properties of (simple) graphs.
It does so by introducing a new combinatorial notion of homotopy between graph maps and subsequently reinterpreting the usual homotopy-theoretic invariants, such as homotopy or homology groups, through the lenses of this new notion.

Several approaches to, or models of, discrete homotopy theory exist in the literature, with the most prominent ones being A-homotopy theory and $\times$-homotopy theory \cite{dochtermann:hom-complex-and-homotopy}.
Here we work with the former.
A-homotopy theory was established in the foundational work of Barcelo and collaborators \cite{barcelo-kramer-laubenbacher-weaver,babson-barcelo-longueville-laubenbacher}.
Many of its ideas, however, can be traced back to the work of Atkin \cite{atkin:i,atkin:ii} (after whom the field is named), later made mathematically rigorous by Kramer and Laubenbacher \cite{kramer-laubenbacher}, on the homotopy theory of simplicial complexes for network analysis, and of Maurer \cite{maurer:matroid-basis-i,maurer:matroid-basis-ii} on basis (exchange) graphs of matroids.
Thus from its conception, A-homotopy theory was known to have a broad range of applications, both within and outside mathematics, and we refer the reader to \cite{barcelo-laubenbacher} for an excellent survey thereof, including further connections to hyperplane arrangements and time series analysis.
The $\times$-homotopy theory referenced above is on the other hand closely related to graph colorings.
As $n$-colorings of a graph $G$ can be understood as graph homomorphisms $G \to K_n$, where $K_n$ is a complete graph on $n$ vertices, studying the space of such graph homomorphisms, known as the \emph{Hom complex}, allows one to import cohomological techniques to disprove existence of certain maps \cite{lovasz,babson-kozlov,dochtermann:hom-complex-and-homotopy,dochtermann:homotopy-groups}.

More recently, A-homotopy theory has seen significant development, driven partly by new applications in Topological Data Analysis \cite{memoli-zhou} and partly by the introduction of new techniques from abstract homotopy theory \cite{carranza-kapulkin:cubical-graphs}.
The two go hand in hand: new applications require effective computational tools which can be developed using these new techniques, e.g., the long exact sequence of a fibration is constructed and the Hurewicz theorem is proven in \cite{carranza-kapulkin:cubical-graphs}.

In the present paper we take the natural next step towards building a robust foundation of computations within A-homotopy theory by focusing on the homotopy $1$-type of a graph.
More precisely, we study the fundamental group and, more generally, the fundamental groupoid of a graph, developing a number of important computational tools.
We now survey the contributions of this paper and their relation to the existing work in the field.

After collecting the background information on A-homotopy theory in \cref{sec:background}, we compare different definitions of homotopy groups in \cref{sec:fin-fomrulation}.
Historically, e.g., in \cite{babson-barcelo-longueville-laubenbacher,barcelo-laubenbacher}, homotopy groups of a graph are defined by taking maps out of an infinite grid graph and requiring that these maps are constant outside of a fixed finite region.
Alternatively, one could define homotopy groups using maps out of finite grids, without stipulating any additional conditions.
Each of these approaches has its own set of advantages and disadvantages, although we found that in our work, the `finite grid' approach is generally preferable.

What is not immediately clear in this approach is what identifications need to be made when defining maps on grids of different sizes.
To describe them, we adapt the notion of reparametrization by a shrinking map, originally used in \cite{grigor'yan-lin-muranov-yau:homotopy}, to the setting of undirected graphs; see \cref{def:shrinking} below.
The equivalence of the two approaches is then established in \cref{equivalence-finite-infinite}.

We then in \cref{sec:fund-gpd} define the fundamental groupoid $\Pi_1 X$ of a graph $X$, which is a multi-object generalization of the fundamental group.
We prove in particular that $\Pi_1$ is a monoidal functor (\cref{thm:pi1_monoidal}) and use it to deduce that that the category of graphs is enriched over groupoids (\cref{graph-gpd-enriched}).
In the combinatorial context, the fundamental groupoid was previously studied by Chih and Scull \cite{chih-scull:fund-gpd} for Dochtermann's $\times$-homotopy theory \cite{dochtermann:hom-complex-and-homotopy} in the category of non-reflexive graphs.
The two notions are however incompatible, as they are based on different graph products.
Thus, although minor similarities between the proof techniques can be noted, our results neither formally imply those of \cite{chih-scull:fund-gpd} nor vice versa.
In general however the proof techniques of our work are more categorical, whereas those of \cite{chih-scull:fund-gpd} are more direct.

Following that, in \cref{sec:covering}, we develop the theory of covering graphs, mirroring the familiar theory of covering spaces.
In A-homotopy theory, covering graphs were previously studied by Barcelo, Greene, Jarrah, and Welker \cite{barcelo-greene-jarrah-welker:vanishing}, and also Hardeman \cite{hardeman:lifting}, although the definition used there more closely resembles that of a local isomorphism, which we strengthen by imposing an additional lifting condition.
Using this refined definition, we are able to prove the existence of the universal cover for all pointed graphs (\cref{univ-cov-exists}) and establish a classification of covering graphs via functors on the fundamental group(oid) of the graph (\cref{thm:covers-as-functors}).
The notions of local isomorphism and covering graph agree for graphs without 3- and 4-cycles, and hence our results generalize those of, e.g.,  \cite[Lem.~4.3.(3)]{barcelo-greene-jarrah-welker:vanishing} and \cite[Thm.~3.0.10]{hardeman:lifting}.
On the other hand, the extension to graphs with 3- and 4-cycles is an important one, since graphs without 3- or 4-cycles must have trivial higher homotopy groups (see ~\cite[Thm.~1.1]{lutz}).
In $\times$-homotopy theory, covering graphs were studied by Chih and Scull \cite{chih-scull:covers}, although, just as in the case of the fundamental groupoid, the two notions are quite dissimilar, and the proof techniques are very different, with ours being more categorical, and those of \cite{chih-scull:covers} being more combinatorial.

Finally, in \cref{sec:kampen}, we prove the discrete Seifert--van Kampen theorem on the preservation of pushouts by the fundamental group functor (\cref{thm:van-kampen-groups}), which we deduce from a stronger version for the fundamental groupoid (\cref{thm:van-kampen-refined}).
A version of this theorem was previously proven in the context of A-homotopy theory in \cite{barcelo-kramer-laubenbacher-weaver}, however under an additional hypothesis on the pushout square to be preserved.
We strengthen the theorem presented therein by relaxing the hypothesis to one that is often found to be satisfied in examples of interest.
In particular, in \cref{ex:S2-simply-connected}, we give a graph whose fundamental group could be computed using our version of the Seifert--van Kampen theorem, but not the one found in \cite{barcelo-kramer-laubenbacher-weaver}.

As an application of our framework, in \cref{sec:application}, we present a purely combinatorial construction of a graph with a prescribed fundamental group, given the group's presentation.

\textbf{Acknowledgement.} This material is based upon work supported by the National Science Foundation under Grant No.~DMS-1928930 while K.K.~participated in a program hosted by the Mathematical Sciences Research Institute in Berkeley, California, during the Spring 2023 semester.

\section{Preliminaries} \label{sec:background}

In this opening section, we review the necessary background on A-homotopy theory.

\subsection{The category of graphs}

We begin by recalling the definition of a graph and a graph map.

\begin{definition}
      A \emph{graph} $X = \left( \V{X},\E{X} \right)$ consists of a set $\V{X}$ of vertices together with a set $\E{X}$ of edges, where an edge $e \in \E{G}$ is a cardinality $2$ subset of $\V{X}$.
\end{definition}

That is, the graphs in this paper are simple and undirected.
If $e = \left\{x,x'\right\} \in \E{X}$ is an edge in the graph $X$, then we say that the vertices $x$ and $x'$ are \emph{adjacent} and write $x \sim x'$.

\begin{definition}
      A \emph{graph map} $f \from X \to Y$ is a function $f \from \V{X} \to \V{Y}$ on the underlying vertex sets satisfying the following property: if $x \sim x'$ is an edge in $X$, then either $f\left(x\right) = f\left(x'\right)$ or $f\left(x\right) \sim f\left(x'\right)$ is an edge in $X$.
\end{definition}
    
That is, given an edge in the domain graph $X$, $f$ either contracts it to a single vertex in the codomain graph $Y$ or maps to it to an edge in the codomain graph $Y$.
This might be surprising at first, but it correctly captures the idea of working with reflexive graphs.
Indeed, by postulating that our graphs are reflexive, our definition of a graph map recovers the familiar definition of a graph homomorphism.
Note however that the above definition differs from that of a graph homomorphism (a function between sets of vertices taking edges to edges) if the graphs are not reflexive.
This is indeed one of the key differences between A-homotopy theory and $\times$-homotopy theory.

We write $\Graph$ for the the category of graphs and graph maps.

\begin{example} \leavevmode
\begin{enumerate}
    \item The $n$-interval $I_{n}$, for $n \in \mathbb{N}$.
    \begin{itemize}
        \item vertices: integers $0, 1, \ldots, n$
        \item edges: $i \sim i+1$ for $i = 0, \ldots, n-1$.
    \end{itemize}
    
    
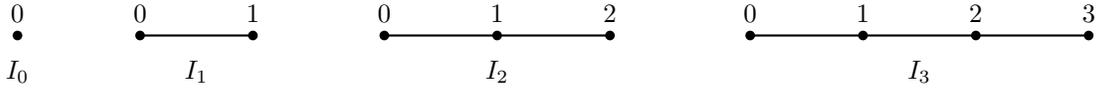
\begin{figure}[H]
    \centering
    \begin{minipage}{0.09\textwidth}
    \centering
    \begin{tikzpicture}[colorstyle/.style={circle, draw=black, fill=black, thick, inner sep=0pt, minimum size=1 mm, outer sep=0pt},scale=1.5]
        \node (0) at (0,0) [colorstyle, label = above: {$0$}] {};
        \node at (0,-0.5) [anchor = south]{$I_{0}$};
    \end{tikzpicture}
    \end{minipage}
    \begin{minipage}{0.19\textwidth}
    \centering
    \begin{tikzpicture}[colorstyle/.style={circle, draw=black, fill=black, thick, inner sep=0pt, minimum size=1 mm, outer sep=0pt},scale=1.5]
        \node (0) at (0,0) [colorstyle, label = above: {$0$}] {};
        \node (1) at (1,0) [colorstyle, label = above: {$1$}] {};
        \draw [thick] (0) -- (1);
        \node at (0.5,-0.5) [anchor = south]{$I_{1}$};
    \end{tikzpicture}
    \end{minipage}
    \begin{minipage}{0.29\textwidth}
    \centering
    \begin{tikzpicture}[colorstyle/.style={circle, draw=black, fill=black, thick, inner sep=0pt, minimum size=1 mm, outer sep=0pt},scale=1.5]
        \node (0) at (0,0) [colorstyle, label = above: {$0$}] {};
        \node (1) at (1,0) [colorstyle, label = above: {$1$}] {};
        \node (2) at (2,0) [colorstyle, label = above: {$2$}] {};
        \draw [thick] (0) -- (2);
        \node at (1,-0.5) [anchor = south]{$I_{2}$};
    \end{tikzpicture}
    \end{minipage}
    \begin{minipage}{0.39\textwidth}
    \centering
    \begin{tikzpicture}[colorstyle/.style={circle, draw=black, fill=black, thick, inner sep=0pt, minimum size=1 mm, outer sep=0pt},scale=1.5]
        \node (0) at (0,0) [colorstyle, label = above: {$0$}] {};
        \node (1) at (1,0) [colorstyle, label = above: {$1$}] {};
        \node (2) at (2,0) [colorstyle, label = above: {$2$}] {};
        \node (3) at (3,0) [colorstyle, label = above: {$3$}] {};
        \draw [thick] (0) -- (3);
        \node at (1.5,-0.5) [anchor = south]{$I_{3}$};
    \end{tikzpicture}
    \end{minipage}
    \end{figure}
    
    \item The infinite interval $I_{\infty}$.
    \begin{itemize}
        \item vertices: all integers $i \in \mathbb{Z}$
        \item edges: $i \sim i+1$ for all $i \in \mathbb{Z}$
    \end{itemize}
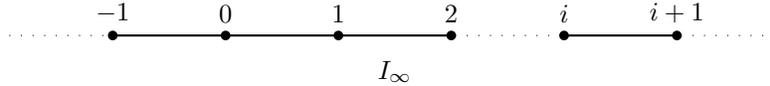
    \begin{figure}[H]
    \centering
    \begin{minipage}{0.9\textwidth}
    \centering
    \begin{tikzpicture}[colorstyle/.style={circle, draw=black, fill=black, thick, inner sep=0pt, minimum size=1 mm, outer sep=0pt},scale=1.5]
       \node (-etc) at (0,0) {};
       \node (-1) at (1,0) [colorstyle, label = above: {$-1$}] {};
       \node (0) at (2,0) [colorstyle, label = above: {$0$}] {};
       \node (1) at (3,0) [colorstyle, label = above: {$1$}] {};
       \node (2) at (4,0) [colorstyle, label = above: {$2$}] {};
       \node (i) at (5,0) [colorstyle, label = above: {$i$}] {};
       \node (i+1) at (6,0) [colorstyle, label = above: {$i+1$}] {};
       \node (etc) at (7,0) {};
       \draw [loosely dotted] (-etc) -- (-1);
       \draw [thick] (-1) -- (0);
       \draw [thick] (0) -- (1);
       \draw [thick] (1) -- (2);
       \draw [loosely dotted] (2) -- (i);
       \draw [thick] (i) -- (i+1);
       \draw [loosely dotted] (i+1) -- (etc);
       \node at (3.5,-0.5) [anchor = south]{$I_{\infty}$};
   \end{tikzpicture}
   \end{minipage}
   \end{figure}
    \item The $n$-cycle $C_{n}$ for $n \geq 3$.
    \begin{itemize}
        \item[$\bullet$] vertices: integers modulo n, that is $0, \ldots, n-1 \in \mathbb{Z}/n\mathbb{Z}$
        \item[$\bullet$] edges: $i \sim i+1$ for all $i \in \mathbb{Z}/n\mathbb{Z}$
    \end{itemize}
    \begin{figure}[H]
    \centering
    \begin{minipage}{0.3\textwidth}
    \centering
    \begin{tikzpicture}[colorstyle/.style={circle, draw=black, fill=black, thick, inner sep=0pt, minimum size=1 mm, outer sep=0pt},scale=1.5]
        \node (0) at (0.5,0.866) [colorstyle, label = above: {$0$}] {};
        \node (1) at (1,0) [colorstyle, label = below right: {$1$}] {};
        \node (2) at (0,0) [colorstyle, label = below left: {$2$}] {};
        \draw [thick] (0) -- (1);
        \draw [thick] (1) -- (2);
        \draw [thick] (2) -- (0);
        \node at (0.5,-0.75) [anchor = south]{$C_{3}$};
    \end{tikzpicture}
    \end{minipage}
    \begin{minipage}{0.3\textwidth}
    \centering
    \begin{tikzpicture}[colorstyle/.style={circle, draw=black, fill=black, thick, inner sep=0pt, minimum size=1 mm, outer sep=0pt},scale=1.5]
        \node (0) at (0,1) [colorstyle, label = above left: {$0$}] {};
        \node (1) at (1,1) [colorstyle, label = above right: {$1$}] {};
        \node (2) at (1,0) [colorstyle, label = below right: {$2$}] {};
        \node (3) at (0,0) [colorstyle, label = below left: {$3$}] {};
        \draw [thick] (0) -- (1);
        \draw [thick] (1) -- (2);
        \draw [thick] (2) -- (3);
        \draw [thick] (3) -- (0);
        \node at (0.5,-0.75) [anchor = south]{$C_{4}$};
    \end{tikzpicture}
    \end{minipage}
    \begin{minipage}{0.3\textwidth}
    \centering
    \begin{tikzpicture}[colorstyle/.style={circle, draw=black, fill=black, thick, inner sep=0pt, minimum size=1 mm, outer sep=0pt},scale=1.5]
        \node (0) at (0.809,1.539) [colorstyle, label = above : {$1$}] {};
        \node (1) at (1.618,0.951) [colorstyle, label = right: {$2$}] {};
        \node (2) at (1.309,0) [colorstyle, label = below right: {$3$}] {};
        \node (3) at (0.309,0) [colorstyle, label = below left: {$4$}] {};
        \node (4) at (0,0.951) [colorstyle, label = left: {$0$}] {};
        \draw [thick] (0) -- (1);
        \draw [thick] (1) -- (2);
        \draw [thick] (2) -- (3);
        \draw [thick] (3) -- (4);
        \draw [thick] (4) -- (0);
        \node at (0.809,-0.75) [anchor = south]{$C_{5}$};
    \end{tikzpicture}
    \end{minipage}
    \end{figure}
\end{enumerate}
\end{example}

\begin{proposition}
    The functor $\V{(-)} \from \Graph \to \Set$ that maps a graph to its underlying vertex set admits both adjoints.
    
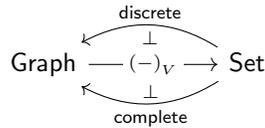
\[
    \begin{tikzcd}
        \Graph \arrow[rr,"\V{(-)}" description] && \arrow[ll, bend right,"\bot","\mathsf{discrete}"'] \arrow[ll, bend left,"\bot"',"\mathsf{complete}"] \Set
    \end{tikzcd}
    \]
    The left adjoint to $\V{(-)}$ maps a set $X$ to the discrete graph on $X$, given by the pair $\left(X,\emptyset\right)$, while the right adjoint to $\V{(-)}$ maps a set $X$ to the complete graph on $X$, given by the pair $\left(X,\left\{ x \sim x' \ \middle\vert \ x, x' \in X, x \neq x'\right\}\right)$.
    In particular, $\V{(-)}$ preserves limits and colimits.
    \qed
\end{proposition}

\begin{proposition}\label{bicomplete}
    The category $\Graph$ is (co)complete.
    \qed
\end{proposition}

\subsection{Monoidal structure on the category of graphs} 

To define the notion of homotopy, we need to review the definition of the product of graphs.
Although $\Graph$ has all small (categorical) products, and in fact all small (co)limits, in the context of A-homotopy theory, we work with a different monoidal product, known as the Cartesian or box product.

\begin{definition}
    The \emph{box product} $X \gtimes Y$ of two graphs $X$ and $Y$ is defined as follows:
    \begin{align*}
        \V{(X \gtimes Y)} &= \V{X} \times \V{Y} \\
        \E{(X \gtimes Y)} &= \left\{ \left(x,y\right) \sim \left(x',y'\right) \ \middle\vert \
        \begin{aligned}
            \text{either } & x \sim x' \text{ and }    y = y' \\
            \text{or } & x = x' \text{ and } y \sim y'
        \end{aligned} \right\} \\
        &\cong \left(\E{X} \times \V{Y}\right) \sqcup \left(\V{X} \times \E{Y}\right)
    \end{align*}
\end{definition}

The box product differs from the categorical product in $\Graph$.
For instance, the box product $I_{1} \gtimes I_{1}$ is given by the graph on the left and has 4 edges, whereas the categorical product $I_{1} \times I_{1}$ is given by the graph on the right and has 6 edges.
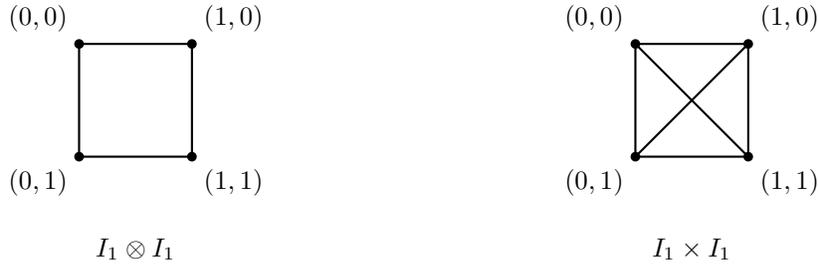
\begin{figure}[H]
\centering
\begin{minipage}{0.45\textwidth}
\centering
\begin{tikzpicture}[colorstyle/.style={circle, draw=black, fill=black, thick, inner sep=0pt, minimum size=1 mm, outer sep=0pt},scale=1.5]
    \node (1) at (0,1) [colorstyle, label = above left: {$\left(0,0\right)$}] {};
    \node (2) at (1,1) [colorstyle, label = above right: {$\left(1,0\right)$}] {};
    \node (3) at (0,0) [colorstyle, label = below left: {$\left(0,1\right)$}] {};
    \node (4) at (1,0) [colorstyle, label = below right: {$\left(1,1\right)$}] {};
    \draw [thick] (1) -- (2);
    \draw [thick] (1) -- (3);
    \draw [thick] (3) -- (4);
    \draw [thick] (2) -- (4);
    \node at (0.5,-1) [anchor = south]{$I_{1} \gtimes I_{1}$};
\end{tikzpicture}
\end{minipage}
\begin{minipage}{0.45\textwidth}
\centering
\begin{tikzpicture}[colorstyle/.style={circle, draw=black, fill=black, thick, inner sep=0pt, minimum size=1 mm, outer sep=0pt},scale=1.5]
    \node (1) at (0,1) [colorstyle, label = above left: {$\left(0,0\right)$}] {};
    \node (2) at (1,1) [colorstyle, label = above right: {$\left(1,0\right)$}] {};
    \node (3) at (0,0) [colorstyle, label = below left: {$\left(0,1\right)$}] {};
    \node (4) at (1,0) [colorstyle, label = below right: {$\left(1,1\right)$}] {};
    \draw [thick] (1) -- (2);
    \draw [thick] (1) -- (3);
    \draw [thick] (3) -- (4);
    \draw [thick] (2) -- (4);
    \draw [thick] (2) -- (3);
    \draw [thick] (1) -- (4);
    \node at (0.5,-1) [anchor = south]{$I_{1} \times I_{1}$};
\end{tikzpicture}
\end{minipage}
\end{figure}

Given two graphs $X$ and $Y$, we have the maps $\pi_{X} \from X \gtimes Y \to X$ and $\pi_{Y} \from X \gtimes Y \to Y$ given by the following composites:
\[
\pi_{X} = \left(
\begin{tikzcd}
    X \gtimes Y \arrow[r,"\text{id}_{X} \gtimes !"]& X \gtimes I_{0} \arrow[r,"\cong"] & X
\end{tikzcd}
\right)
\]
\[
\pi_{Y} = \left(
\begin{tikzcd}
    X \gtimes Y \arrow[r,"! \gtimes \text{id}_{Y}"]& I_{0} \gtimes Y \arrow[r,"\cong"] & Y
\end{tikzcd}
\right)
\]
Thus, we have a map from the box product to the categorical product:
\[
    \left(\pi_{X}, \pi_{Y}\right) \from X \gtimes Y \to X \times Y
\]
We can check that this map is always bijective on vertices and injective on edges.

\begin{definition}
    The \emph{internal hom} $\ghom{X}{Y}$ of two graphs $X$ and $Y$ is defined as follows:
    \begin{align*}
        \V{{\ghom{X}{Y}}} &= \Graph\left(X,Y\right) \\
        \E{\ghom{X}{Y}} &= \left\{ f \sim g \ \middle\vert \ f,g \in \Graph\left(X,Y\right), f \neq g \text{ and } \forall x \in X, \text{ }
        \begin{aligned}
            \text{either } & f\left(x\right) \sim g\left(x\right) \\
            \text{or } & f\left(x\right) = g\left(x\right)
        \end{aligned}\right\}
    \end{align*}    
\end{definition}

One then easily verifies:

\begin{proposition} \label{prop:graph-closed-monoidal}
   The category $\left(\Graph,\gtimes,I_{0},\ghom{-}{-} \right)$ is a closed symmetric monoidal category.
   In other words, given graphs $X, Y, Z$, we have a bijection
    \[
        \Graph\left(X \gtimes Y, Z\right) \cong \Graph\left(X,\ghom{Y}{Z}\right)
    \]
    that is natural in $X$, $Y$, and $Z$.
    \qed 
\end{proposition}

\subsection{Paths in a graph}

We now define combinatorial paths and recall some of their basic properties.

\begin{definition}
    Let $X$ be a graph and $x,x' \in \V{X}$ be vertices.
    \begin{enumerate}
        \item A \emph{path} $\gamma \from x \rightsquigarrow x'$ in $X$ of \emph{length} $n \in \mathbb{N}$ is a map $\gamma \from I_{n} \to X$ such that $\gamma\left(0\right) = x$ and $\gamma\left(n\right) = x'$.
        \item A \emph{path} $\gamma \from x \rightsquigarrow x'$ in $X$ of \emph{length} $\infty$ is a map $\gamma \from I_{\infty} \to X$ for which there exist integers $N_{-}, N_{+} \in \mathbb{Z}$ such that $\gamma\left(i\right) = x$ for all $i \leq N_{-}$ and $\gamma\left(i\right) = x'$ for all $i \geq N_{+}$.
    \end{enumerate}
\end{definition}

\begin{definition}
    Two vertices $x$ and $x'$ in a graph $X$ are \emph{path-connected} if there exists a path from $x$ to $x'$.
    A graph $X$ is \emph{path-connected} if every pair of vertices in $X$ is path-connected.
\end{definition}

\begin{definition}
    Given any vertex $x$ in a graph $X$, the \emph{constant path} $c_{x} \from x \rightsquigarrow x$ of length $n \in \mathbb{N} \cup \left\{\infty\right\}$ is given by the composite
    \[
    c_{x} = \left(\begin{tikzcd}
        I_{n} \arrow[r,"!"] & I_{0} \arrow[r,"x"] & X
    \end{tikzcd}\right)
    \]
\end{definition}

\begin{definition} \label{def:inverse-path}
    Let $X$ be a graph and $x,x' \in \V{X}$ be vertices.
    \begin{enumerate}
        \item Given a path $\gamma \from x \rightsquigarrow x'$ in $X$ of finite length $n$, its \emph{inverse path} $\overline{\gamma} \from x' \rightsquigarrow x$, defined by a map $\overline{\gamma} \from I_{n} \to X$, is given by $\overline{\gamma}\left(i\right) = \gamma\left(n-i\right)$.
        \item Given a path $\gamma \from x \rightsquigarrow x'$ in $X$ of length $\infty$, its \emph{inverse path} $\overline{\gamma} \from x' \rightsquigarrow x$, defined by a map $\overline{\gamma} \from I_{\infty} \to X$, is given by $\overline{\gamma}\left(i\right) = \gamma\left(-i\right)$.
    \end{enumerate}
\end{definition}

\begin{definition}
    Let $X$ be a graph and $x,x',x'' \in \V{X}$ be vertices.
    \begin{enumerate}
        \item Given two paths $\gamma \from x \rightsquigarrow x'$ and $\sigma \from x' \rightsquigarrow x''$ in $X$ of finite lengths $m$ and $n$ respectively, their \emph{concatenation} $\gamma \ast \sigma \from x \rightsquigarrow x''$, defined by a map $\gamma \ast \sigma \from I_{m+n} \to X$, is given by
        \[
            \left(\gamma \ast \sigma\right)\left(i\right) =
            \begin{cases}
                \gamma\left(i\right) & \text{if } i \leq m \\
                \sigma\left(i-m\right) & \text{if } i \geq m
            \end{cases}
        \]
        \item Given two paths $\gamma \from x \rightsquigarrow x'$ and $\sigma \from x' \rightsquigarrow x''$ in $X$, both of length $\infty$, their \emph{concatenation} $\gamma \ast \sigma \from x \rightsquigarrow x''$, defined by a map $\gamma \ast \sigma \from I_{\infty} \to X$, is given by
        \[
            \left(\gamma \ast \sigma\right)\left(i\right) =
            \begin{cases}
                \gamma\left(i\right) & \text{if } i \leq M_{+} \\
                \sigma\left(i-M_{+}+N_{-}\right) & \text{if } i \geq M_{+}
            \end{cases}
        \]
        where $M_{+}$ is the smallest integer such that $\gamma\left(i\right) = x'$ for all $i \geq M_{+}$ and $N_{-}$ is the largest integer such that $\sigma\left(i\right) = x'$ for all $i \leq N_{-}$.
    \end{enumerate}
\end{definition}

\begin{definition} \label{def:edge-of-path}
    Let $n \in \mathbb{N}$ and for $i = 1, \ldots, n$, let $e_{i} \from I_{1} \to I_{n}$ be the graph map given by
    \[
        e_{i}\left(j\right) =
        \begin{cases}
            i-1 & \text{if } j = 0 \\
            i & \text{if } j = 1
        \end{cases}
    \]
    Given a path $\gamma \from x \rightsquigarrow x'$ in $X$ of finite length $n$, its \emph{$i$-th edge} is given by the path $e_{i}^{*}\left(\gamma\right) = \gamma \circ e_{i} \from \gamma\left(i-1\right) \rightsquigarrow \gamma\left(i\right)$ of length $1$, where $i = 1, \ldots, n$.
\end{definition}

\begin{proposition}
    Path-connectedness of vertices in a graph $X$ is the equivalence relation on $\V{X}$ generated by adjacency of vertices.
    \qed
\end{proposition}

Given a graph map $f \from X \to Y$ and a path $\gamma \from x \rightsquigarrow x'$ in $X$, the composite $f \circ \gamma$ is a path in $Y$ from $f\left(x\right)$ to $f\left(x'\right)$.
Thus, if two vertices in $X$ are path-connected, then their respective images under $f$ are path-connected in $Y$.

\begin{definition}
    The \emph{set of path-components} $\pi_{0}X$ of a graph $X$ is the set of equivalence classes of vertices of $X$ under path-connectedness.
    A graph map $f \from X \to Y$ induces a function $\pi_{0}f \from \pi_{0}X \to \pi_{0}Y$.
    The assignment $f \mapsto \pi_{0}f$ is clearly functorial, and we have a well-defined functor $\pi_{0} \from \Graph \to \Set$.
\end{definition}

\begin{proposition}\label{prop:pi0-discrete-graph-adjunction}
    The functor $\pi_{0} \from \Graph \to \Set$ is left adjoint to the functor $\Set \to \Graph$ that maps a set $X$ to the discrete graph on $X$.
    \[
    \pushQED{\qed}
    \begin{tikzcd}
        \Set \arrow[rr,"\mathsf{discrete}" description] && \arrow[ll, bend right,"\bot","\pi_{0}"'] \arrow[ll, bend left,"\bot"',"\V{(-)}"] \Graph
    \end{tikzcd}
    \qedhere
    \popQED
    \]    
\end{proposition}

Given two graphs $X$ and $Y$, we can apply $\pi_{0}$ to the maps $\pi_{X} \from X \gtimes Y \to X$ and $\pi_{Y} \from X \gtimes Y \to Y$ and then take their product in $\Set$ to obtain the following set-function:
\[
    \pi_{0}\left(X \gtimes Y\right) \to \pi_{0}X \times \pi_{0}Y
\]

Recall that a functor $F \from (\cat{C},\otimes_{\cat{C}},1_{\cat{C}}) \to (\cat{D},\otimes_{\cat{D}},1_{\cat{D}})$ between two monoidal categories, that is equipped with natural isomorphisms $F(1_{\cat{C}}) \to 1_{\cat{D}}$ and $F(X \otimes_{\cat{C}} Y) \to F(X) \otimes_{\cat{D}} F(Y)$ for every $X, Y \in \cat{C}$ subject to coherence axioms corresponding to associativity and unitality, is strong monoidal.

\begin{proposition}
    The functor $\pi_{0} \from \Graph \to \Set$ is strong monoidal.
    In particular, the set-function $\pi_{0}\left(X \gtimes Y\right) \to \pi_{0}X \times \pi_{0}Y$ is a bijection that is natural in the graphs $X$ and $Y$.
\end{proposition}
\begin{proof}
     Clearly the functor $\pi_{0} \from \Graph \to \Set$ preserves units, since $\pi_{0}I_{0} \cong \left\{\ast\right\}$.
     In order to check that the set-function $\pi_{0}\left(X \gtimes Y\right) \to \pi_{0}X \times \pi_{0}Y$ is a bijection, we need to check that given two vertices $x, x'$ that are path-connected in $X$ and two vertices $y,y'$ that are path-connected in $Y$, the pairs $\left(x,y\right)$ and $\left(x',y'\right)$ are path-connected in the box-product $X \gtimes Y$.
     Let $\gamma \from x \rightsquigarrow x'$ be a path in $X$ and let $\sigma \from y \rightsquigarrow y'$ be a path in $Y$.
     Then, the concatenation $\left(\gamma \otimes y\right) \ast \left(x' \otimes \sigma \right)$ is a path from $\left(x,y\right)$ to $\left(x',y'\right)$ in $X \gtimes Y$.
\end{proof}

\subsection{Homotopy theory of graphs}

We are now ready to define homotopies and homotopy equivalences.
Before doing so, we make a small observation that helps motivate these notions.

Observe that two distinct maps $f, g \from X \to Y$ are adjacent in the graph $\ghom{X}{Y}$ if and only if there exists a graph map $H \from I_{1} \to \ghom{X}{Y}$ such that $H\left(0\right) = f$ and $H\left(1\right) = g$, or equivalently, a graph map $H \from X \gtimes I_{1} \to Y$ such that $H\left(-,0\right) = f$ and $H\left(-,1\right) = g$.

\begin{definition}
    Let $X$ and $Y$ be two graphs and $f, g \from X \to Y$ be two graph maps.
    \begin{enumerate}
        \item An \emph{$A$-homotopy} $H \from f \Rightarrow g$ of length $n \in \mathbb{N}$ is a map $H \from X \gtimes I_{n} \to Y$ such that $H\left(-,0\right) = f$ and $H\left(-,n\right) = g$.
        \item An \emph{$A$-homotopy} $H \from f \Rightarrow g$ of length $\infty$ is a map $H \from X \gtimes I_{\infty} \to Y$ for which there exist integers $N_{-}, N_{+} \in \mathbb{Z}$ such that $H\left(-,i\right) = f$ for all $i \leq N_{-}$ and $H\left(-,i\right) = g$ for all $i \geq N_{+}$.
    \end{enumerate}
    When such an $A$-homotopy exists, we say that $f$ and $g$ are $A$-homotopic and write $f \sim_{A} g$.
    Equivalently, an $A$-homotopy $H \from f \Rightarrow g$ between two graph maps $f, g \from X \to Y$ is a path $H \from f \rightsquigarrow g$ in the graph $\ghom{X}{Y}$.
    It follows that $\sim_{A}$ is an equivalence relation on the set $\Graph\left(X,Y\right)$.
\end{definition}

\begin{definition}
    A graph map $f \from X \to Y$ is an \emph{$A$-homotopy equivalence} if there exists some graph map $g \from X \to Y$ such that $g \circ f \sim_{A} \text{id}_{X}$ and $f \circ g \sim_{A} \text{id}_{Y}$.
    When such a homotopy equivalence exists, we say that the graphs $X$ and $Y$ are \emph{$A$-homotopy equivalent} and write $X \simeq_{A} Y$.
\end{definition}

Note that $\simeq_{A}$ defines an equivalence relation on the class of graphs.

\begin{definition}
    A graph $X$ is \emph{$A$-contractible} if the unique map $X \to I_{0}$ is an $A$-homotopy equivalence.
    In other words, if there exists some vertex $x_{0} \from I_{0} \to X$ such that $\text{id}_{G} \sim_{A} c_{x_{0}}$, where $c_{x_{0}} \from X \to X$ is the composite:
    \[
    c_{x_{0}} = \left(\begin{tikzcd}
        X \arrow[r] & I_{0} \arrow[r,"x_{0}"] & X
    \end{tikzcd}\right)
    \]
\end{definition}

\begin{example}
    The interval graph $I_{n}$ is contractible for $n \in \mathbb{N}$. However the infinite interval $I_{\infty}$ is not contractible.
\end{example}

\begin{example}[\cite{barcelo-kramer-laubenbacher-weaver}]
    The cycle graph $C_{n}$ is contractible for $n = 3, 4$.
    The following figures depict $A$-homotopies from the identity map on $C_{n}$ to the constant map $c_{0}$, for $n = 3, 4$.
    \begin{figure}[H]
\centering
\begin{minipage}{0.45\textwidth}
\centering
\begin{tikzpicture}[colorstyle/.style={circle, draw=black, fill=black, thick, inner sep=0pt, minimum size=1 mm, outer sep=0pt},scale=1.5]
    \node (1) at (0,1) [colorstyle, label = left: {0}] {};
    \node (2) at (1,1) [colorstyle, label = right: {2}] {};
    \node (3) at (0,0) [colorstyle, label = left: {0}] {};
    \node (4) at (1,0) [colorstyle, label = right: {0}] {};
    \node (5) at (0.75,1.433) [colorstyle, label = above: {1}] {};
    \node (6) at (0.75,0.433) [colorstyle, label = above left: {0}] {};
    \draw [thick] (1) -- (2);
    \draw [thick] (1) -- (3);
    \draw [thick] (3) -- (4);
    \draw [thick] (2) -- (4);
    \draw [thick] (1) -- (5); 
    \draw [thick] (5) -- (2);
    \draw [thick] (3) -- (6); 
    \draw [thick] (6) -- (4);
    \draw [thick] (5) -- (6);
    \node at (0.5,-1) [anchor = south]{$C_{3} \gtimes I_{1} \to C_{3}$};
\end{tikzpicture}
\end{minipage}
\begin{minipage}{0.45\textwidth}
\centering
\begin{tikzpicture}[colorstyle/.style={circle, draw=black, fill=black, thick, inner sep=0pt, minimum size=1 mm, outer sep=0pt},scale=1.5]
    \node (1) at (0,2) [colorstyle, label = left: {0}] {};
    \node (2) at (1,2) [colorstyle, label = below right: {1}] {};
    \node (3) at (0.75,2.433) [colorstyle, label = above left: {3}] {};
    \node (4) at (1.75,2.433) [colorstyle, label = right: {2}] {};
    \node (5) at (0,1) [colorstyle, label = left: {0}] {};
    \node (6) at (1,1) [colorstyle, label = below right: {1}] {};
    \node (7) at (0.75,1.433) [colorstyle, label = above left: {0}] {};
    \node (8) at (1.75,1.433) [colorstyle, label = right: {1}] {};
    \node (9) at (0,0) [colorstyle, label = left: {0}] {};
    \node (10) at (1,0) [colorstyle, label = below right: {0}] {};
    \node (11) at (0.75,0.433) [colorstyle, label = above left: {0}] {};
    \node (12) at (1.75,0.433) [colorstyle, label = right: {0}] {};
    \draw [thick] (1) -- (2);
    \draw [thick] (1) -- (3);
    \draw [thick] (3) -- (4);
    \draw [thick] (2) -- (4);
    
    \draw [thick] (5) -- (6);
    \draw [thick] (5) -- (7);
    \draw [thick] (7) -- (8);
    \draw [thick] (6) -- (8);
    
    \draw [thick] (9) -- (10);
    \draw [thick] (9) -- (11);
    \draw [thick] (11) -- (12);
    \draw [thick] (10) -- (12);

    \draw [thick] (1) -- (9);
    \draw [thick] (2) -- (10);
    \draw [thick] (3) -- (11);
    \draw [thick] (4) -- (12);
    
    \node at (0.875,-1) [anchor = south]{$C_{4} \gtimes I_{2} \to C_{4}$};
\end{tikzpicture}
\end{minipage}
\end{figure}
\end{example}

\subsection{Homotopy groups of graphs}

We conclude our background section by reviewing the definition of discrete homotopy groups.

\begin{definition}
    Let $X$ be a graph and $x_{0},x_{1} \in \V{X}$ be two vertices.
    For $m \in \mathbb{N} \cup \left\{\infty\right\}$, the $m$-path graph $P_{m}X\left(x_{0},x_{1}\right)$ is the full subgraph of $\ghom{I_{m}}{X}$ with vertices given by paths of length $m$ in $X$ from $x_{0}$ to $x_{1}$.
\end{definition}

Observe that two distinct paths $\gamma, \sigma \from x_{0} \rightsquigarrow x_{1}$ are adjacent in the graph $P_{m}X\left(x_{0},x_{1}\right)$ if and only if there exists a graph map $H \from I_{1} \to P_{m}X\left(x_{0},x_{1}\right)$ such that $H\left(0\right) = \gamma$ and $H\left(1\right) = \sigma$, or equivalently, if and only if there exists a graph map $H \from I_{m} \gtimes I_{1} \to X$ such that $H\left(-,0\right) = \gamma$ and $H\left(-,1\right) = \sigma$.

\begin{definition}
    Let $X$ be a graph, $x_{0}, x_{1} \in X$ be two vertices, and $\gamma, \sigma \from x_{0} \rightsquigarrow x_{1}$ be two paths in $X$ of length $m$ for some $m \in \mathbb{N} \cup \left\{\infty\right\}$. 
    \begin{enumerate}
        \item A \emph{path-homotopy} $H \from \gamma \Rightarrow \sigma$ of length $n \in \mathbb{N}$ is a map $H \from I_{m} \gtimes I_{n} \to X$ such that $H\left(-,0\right) = \gamma$ and $H\left(-,n\right) = \sigma$, and such that each $H\left(-,i\right)$ is a path from $x_{0}$ to $x_{1}$ in $X$ of length $m$.
        \item A \emph{path-homotopy} $H \from \gamma \Rightarrow \sigma$ of length $\infty$ is a map $H \from I_{m} \gtimes I_{\infty} \to X$ for which there exist integers $N_{-}, N_{+} \in \mathbb{Z}$ such that $H\left(-,i\right) = \gamma$ for all $i \leq N_{-}$ and $H\left(-,i\right) = \sigma$ for all $i \geq N_{+}$, and such that each $H\left(-,i\right)$ is a  path from $x_{0}$ to $x_{1}$ in $X$ of length $m$.
    \end{enumerate} 
    When such a path-homotopy exists, we say that $\gamma$ and $\sigma$ are path-homotopic and write $\gamma \sim_{m} \sigma$.
\end{definition}

Equivalently, a path-homotopy $H \from \gamma \Rightarrow \sigma$ between two paths $\gamma, \sigma \from x_{0} \rightsquigarrow x_{1}$ in $X$, both of length $m$, is a path $H \from \gamma \rightsquigarrow \sigma$ in the graph $P_{m}X\left(x_{0},x_{1}\right)$.
It follows that $\sim_{m}$ is an equivalence relation on the set $\V{P_{m}X\left(x_{0},x_{1}\right)}$.

\begin{definition}
    For a pointed graph $\left(X,x_{0}\right)$ and $m \in \mathbb{N} \cup \left\{\infty\right\}$, the $m$-loop graph $\Omega_{m}\left(X,x_{0}\right)$ is given by $P_{m}X\left(x_{0},x_{0}\right)$.
    It has a distinguished vertex given by the constant path at $x_{0}$ of length $m$.
    This gives an endofunctor $\Omega_{m} \from \Graph_{*} \to \Graph_{*}$.
\end{definition}

Iterating the endofunctor $\Omega_{m} \from \Graph_{*} \to \Graph_{*}$, we get:
\[
    \Omega_{m}^{d} \from \Graph_{*} \to \Graph_{*}, \qquad \text{for } d \in \mathbb{N}
\]

\begin{proposition}
    For any pointed graph $\left(X,x_{0}\right)$, $m \in \mathbb{N} \cup \left\{\infty\right\}$ and $d \in \mathbb{N}$, the graph $\Omega_{m}^{d}\left(X,x_{0}\right)$ is isomorphic to the full subgraph of $\ghom{\gexp{I_{m}}{d}}{X}$ with vertices given by the maps $f \from \gexp{I_{m}}{d} \to X$ for which each component
    \[
        f_{i} \from I_{m} \to \ghom{\gexp{I_{m}}{d-1}}{G}, \qquad \text{for } i = 1, \ldots, d
    \]
    given by
    \[
        t \mapsto \left(\left(t_{1},\ldots,\hat{t_{i}},\ldots,t_{d}\right) \mapsto f\left(t_{1},\ldots,t,\ldots,t_{d}\right)\right)
    \]
    is a path of length $m$ from $c_{x_{0}}$ to $c_{x_{0}}$, where $c_{x_{0}}$ is the constant map at $x_{0}$.
    \qed
\end{proposition}

\begin{definition}
    For $d \geq 1$, and for each $i = 1, \ldots, d$, we define a binary operation $\cdot_{i}$ on the vertex set of the graph $\Omega_{\infty}^{d}\left(X,x_{0}\right)$ as follows: given $f,g \in \Omega_{\infty}^{d}\left(X,x_{0}\right)$, let $f \cdot_{i} g \in \Omega_{\infty}^{d}\left(X,x_{0}\right)$ be the element of $\ghom{\gexp{I_{\infty}}{d}}{X}$ corresponding to the concatenation $f_{i} \ast g_{i}$, where $f_{i}, g_{i} \from I_{\infty} \to \ghom{\gexp{I_{\infty}}{d-1}}{X}$ are the paths $c_{x_{0}} \rightsquigarrow c_{x_{0}}$ of length $\infty$ given by
    \[
        t \mapsto \left(\left(t_{1},\ldots,\hat{t_{i}},\ldots,t_{d}\right) \mapsto f\left(t_{1},\ldots,t,\ldots,t_{d}\right)\right)
    \]
    and
    \[
        t \mapsto \left(\left(t_{1},\ldots,\hat{t_{i}},\ldots,t_{d}\right) \mapsto g\left(t_{1},\ldots,t,\ldots,t_{d}\right)\right)
    \]
    respectively.
\end{definition}

Note that the binary operation $\cdot_{i}$ does \emph{not} define a graph map $\Omega_{\infty}^{d}\left(X,x_{0}\right) \gtimes \Omega_{\infty}^{d}\left(X,x_{0}\right) \to \Omega_{\infty}^{d}\left(X,x_{0}\right)$.

\begin{definition} \label{def:homotopy-groups}
    For $d \in \mathbb{N}$, we define the \emph{$d$-th $A$-homotopy group} $A_{d}\left(X,x_{0}\right)$ of a pointed graph $\left(X,x_{0}\right)$ to be the set $\pi_{0}\Omega_{\infty}^{d}\left(X,x_{0}\right)$ of path-components of the graph $\Omega_{\infty}^{d}\left(X,x_{0}\right)$.
    For $d \geq 1$, and for each $i = 1, \ldots, d$, the binary operation $\cdot_{i}$ induces a group operation on homotopy groups $\cdot_{i} \from A_{d}\left(X,x_{0}\right) \times A_{d}\left(X,x_{0}\right) \to A_{d}\left(X,x_{0}\right)$.
    For $d \geq 2$ and $1 \leq i < j \leq d$, the group operations $\cdot_{i}$ and $\cdot_{j}$ satisfy the interchange law: given any $a, b, c, d \in \Omega_{\infty}^{d}\left(X,x_{0}\right)$, we have a path in $\Omega_{\infty}^{d}\left(X,x_{0}\right)$ from $\left(a \cdot_{i} b\right) \cdot_{j} \left(c \cdot_{i} d\right)$ to $\left(a \cdot_{j} c\right) \cdot_{i} \left(b \cdot_{j} d\right)$.
    It follows by the Eckmann-Hilton argument that these group operations coincide and are abelian.
\end{definition}

In $\times$-homotopy theory, one can show that the question of homotopy between graph maps $f, g \colon X \to Y$ can be studied using the Hom-complex $\mathrm{Hom}(X, Y)$, which is a topological space encapsulating the homotopy type of the exponential object $Y^X$.
A similar construction exists in A-homotopy theory, whereby one can understand the homotopy type of a graph $X$ via a topological space $|\mathrm{N}_1 X|$, the geometric realization of its $1$-nerve \cite{carranza-kapulkin:cubical-graphs}.
Indeed, the classical homotopy groups of that space coincide with the $A$-homotopy groups of the graph.

\section{A finite formulation of the homotopy groups} \label{sec:fin-fomrulation}

In the preceding section, we defined homotopy groups using infinite paths (\cref{def:homotopy-groups}), but we could have used paths of finite length instead.
While this statement is clear on the intuitive level, its proof is quite technical.
In this section, we develop the necessary tools and prove that the two definitions (via `finite' vs `infinite' grids) of homotopy groups agree.
Note that the length of a path in a graph is a property of the map, and not its image.
To emphasize this point, we introduce the notion of reparametrization of a path.

\begin{definition} \label{def:shrinking}
    Let $m, n \in \mathbb{N} \cup \left\{\infty\right\}$.
    A \emph{shrinking} map is a surjective and order-preserving graph map $s \from I_{m} \to I_{n}$.
    Given a path $\gamma  \from x \rightsquigarrow x'$ of length $n$ in a graph $X$ and a shrinking map $s \from I_{m} \to I_{n}$, the composite $\gamma \circ s$ is a path in $X$ of length $m$ from $x$ to $x'$, traversing the same sequence of edges as $\gamma$.
    We say that the paths $\gamma \circ s$ and $\gamma$ are \emph{reparametrizations} of each other.
\end{definition}

\begin{lemma} \label{lem:reparametrization-finite-infinite}
\leavevmode
    \begin{enumerate}
        \item Every path of finite length admits a reparametrization of length $\infty$.
        \item Every path of length $\infty$ admits a reparametrization of finite length.
    \end{enumerate}
\end{lemma}
\begin{proof}
    Given a path $\gamma \from x \rightsquigarrow x'$ in $X$ of finite length $n$, we can choose the shrinking map $s \from I_{\infty} \to I_{n}$ given by
    \[
        s\left(i\right) =
        \begin{cases}
            0 & \text{if } i \leq 0 \\
            i & \text{if } 0 \leq i \leq n \\
            n & \text{if } n \leq i
        \end{cases}
    \]
    Then, the path $\gamma \circ s$ is a reparametrization of $\gamma$ having length $\infty$.

    Given a path $\gamma \from x \rightsquigarrow x'$ in $X$ of length $\infty$, there exist integers $N_{-}, N_{+} \in \mathbb{Z}$ such that $\gamma\left(i\right) = x$ for all $i \leq N_{-}$ and $\gamma\left(i\right) = x'$ for all $i \geq N_{+}$.
    Then, letting $n = N_{+} - N_{-}$, we can choose a shrinking map $s \from I_{\infty} \to I_{n}$ and a map $\tilde{\gamma} \from I_{n} \to X$ given by:
    \[
        s\left(i\right) =
        \begin{cases}
            0 & \text{if } i \leq N_{-} \\
            i - N_{-} & \text{if } N_{-} \leq i \leq N_{+} \\
            n & \text{if } N_{+} \leq i
        \end{cases}
    \]
    and
    \[
        \tilde{\gamma}\left(i\right) = \gamma\left(i+N_{-}\right).
    \]
    Note that $\gamma = \tilde{\gamma} \circ s$.
    Hence, $\tilde{\gamma}$ is a reparametrization of $\gamma$ having finite length $n$.
\end{proof}

Given any subset $\mathcal{I} \subseteq \mathbb{N} \cup \left\{\infty\right\}$, we have an equivalence relation $\sim_{\mathcal{I}}$ on the disjoint union $\coprod_{n \in \mathcal{I}}{P_{n}X\left(x,x'\right)}$ that is generated by reparametrizations.
That is, given a path $\gamma \from x \rightsquigarrow x'$ of length $n$ and a shrinking map $s \from I_{m} \to I_{n}$, for $m, n \in \mathcal{I}$, we identify $\gamma \sim_{\mathcal{I}} \gamma \circ s$ and take the smallest equivalence relation containing these pairs.

\begin{definition}\label{def:path-graphs}
  For a graph $X$ and two vertices $x$ and $x'$, the \emph{$\mathcal{I}$-indexed path graph} $P_{\mathcal{I}}X\left(x,x'\right)$ is the quotient of $\coprod_{n \in \mathcal{I}}{P_{n}X\left(x,x'\right)}$ under the equivalence relation defined above.
    \[
        P_{\mathcal{I}}X\left(x,x'\right) = \left(\coprod_{n \in \mathcal{I}}{P_{n}X\left(x,x'\right)}\right)/\sim_{\mathcal{I}}
    \]  
\end{definition}

Suppose we have an inclusion $\mathcal{I} \subseteq \mathcal{J}$, where $\mathcal{I}, \mathcal{J}$ are both subsets of $\mathbb{N} \cup \left\{\infty\right\}$.
Then, given any two paths $\gamma, \sigma \from x \rightsquigarrow x'$ that are identified in $P_{\mathcal{I}}X\left(x,x'\right)$, they are also identified in $P_{\mathcal{J}}X\left(x,x'\right)$.
Thus, the inclusion $\coprod_{n \in \mathcal{I}}{P_{n}X\left(x,x'\right)} \hookrightarrow \coprod_{n \in \mathcal{J}}{P_{n}X\left(x,x'\right)}$ induces a map
\[
\begin{tikzcd}
    P_{\mathcal{I}}X\left(x,x'\right) \arrow[r] & P_{\mathcal{J}}X\left(x,x'\right)
\end{tikzcd}
\]

\begin{proposition}
    The following maps are isomorphisms:
    \[
        P_{\left\{\infty\right\}}X\left(x,x'\right) \longrightarrow P_{\mathbb{N} \cup \left\{\infty\right\}}X\left(x,x'\right)
    \qquad \text{and} \qquad
        P_{\mathbb{N}}X\left(x,x'\right) \longrightarrow P_{\mathbb{N} \cup \left\{\infty\right\}}X\left(x,x'\right)
    \]
\end{proposition}
\begin{proof}
    We first prove surjectivity and then injectivity of both maps.
    We have the following commutative squares:
    \[
    \begin{tikzcd}
        P_{\infty}X\left(x,x'\right) \arrow[r,hook]  \arrow[d] \arrow[dr,dotted] & \coprod_{n \in \mathbb{N} \cup \left\{\infty\right\}}{P_{n}X\left(x,x'\right)} \arrow[d] \\
        P_{\left\{\infty\right\}}X\left(x,x'\right) \arrow[r] & P_{\mathbb{N} \cup \left\{\infty\right\}}X\left(x,x'\right)
    \end{tikzcd}
    \qquad
    \begin{tikzcd}
        \coprod_{n \in \mathbb{N}}{P_{n}X\left(x,x'\right)} \arrow[r,hook] \arrow[d] \arrow[dr,dotted] & \coprod_{n \in \mathbb{N} \cup \left\{\infty\right\}}{P_{n}X\left(x,x'\right)} \arrow[d] \\
        P_{\mathbb{N}}X\left(x,x'\right) \arrow[r] & P_{\mathbb{N} \cup \left\{\infty\right\}}X\left(x,x'\right)
    \end{tikzcd}
    \]
    By \cref{lem:reparametrization-finite-infinite}, the maps represented by the dotted diagonal arrows in both commutative squares are surjective.
    It follows that the maps $P_{\left\{\infty\right\}}X\left(x,x'\right) \to P_{\mathbb{N} \cup \left\{\infty\right\}}X\left(x,x'\right)$ and $P_{\mathbb{N}}X\left(x,x'\right) \to P_{\mathbb{N} \cup \left\{\infty\right\}}X\left(x,x'\right)$ are both surjective.

    Injectivity of $P_{\left\{\infty\right\}}X\left(x,x'\right) \to P_{\mathbb{N} \cup \left\{\infty\right\}}X\left(x,x'\right)$, follows from the fact that for any shrinking map $s \from I_{m} \to I_{\infty}$ where $m \in \mathbb{N} \cup \left\{\infty\right\}$, we must have $m = \infty$.

    Given any shrinking map $s \from I_{\infty} \to I_{n}$ where $n \in \mathbb{N}$, we have a corresponding shrinking map $\tilde{s} \from I_{M} \to I_{n}$ where $M \in \mathbb{N}$ defined as follows. Let $M_{-} \in \mathbb{Z}$ be such that $s\left(M_{-}\right)0$ and $M_{+} \in \mathbb{Z}$ be such that $s\left(M_{+}\right) = n$.
    Let $M = M_{+} - M_{-}$ and let $\tilde{s} \from I_{M} \to I_{n}$ be given by $\tilde{s}\left(i\right) = s\left(i+M_{-}\right)$. Then, for any path $\gamma \from x \rightsquigarrow x'$ of length $n$, the paths $\gamma \circ s$ and $\gamma \circ \tilde{s}$ are identified in $P_{\mathbb{N} \cup \left\{\infty\right\}}X\left(x,x'\right)$. That is, we can replace any shrinking map of the form $s \from I_{\infty} \to I_{n}$ by a shrinking map of the form $\tilde{s} \from I_{M} \to I_{n}$ where $M \in \mathbb{N}$. Hence, the map $P_{\mathbb{N}}X\left(x,x'\right) \to P_{\mathbb{N} \cup \left\{\infty\right\}}X\left(x,x'\right)$ is injective.
\end{proof}

\begin{lemma} \label{lem:reparametrization-v-path-homotopy}
    Given a path $\gamma \from x \rightsquigarrow x'$ in $X$ of length $\infty$ and a shrinking map $s \from I_{\infty} \to I_{\infty}$, there exists a path-homotopy $\gamma \Rightarrow \gamma \circ s$, or equivalently a path $\gamma \rightsquigarrow \gamma \circ s$ in $P_{\infty}X\left(x,x'\right)$.
\end{lemma}
\begin{proof}
    We first consider two special cases before moving on to the general case.
    
    First, suppose the shrinking map $s \from I_{\infty} \to I_{\infty}$ is of the form
    \[
        s\left(i\right) =
        \begin{cases}
            i & \text{if } i \leq k \\
            i-1 & \text{if } i \geq k+1
        \end{cases}
    \]
    for some $k \in \mathbb{Z}$ and let $H \from I_{\infty} \otimes I_{1} \to I_{\infty}$ be given by $H\left(i,0\right) = i$ and $H\left(i,1\right) = s\left(i\right)$.
    Then, $\gamma \circ H$ is a path-homotopy from $\gamma$ to $\gamma \circ s$.

    Next, suppose the shrinking map $s \from I_{\infty} \to I_{\infty}$ is of the form $s\left(i\right) = i+k$ for some $k \in \mathbb{Z}$.
    Let $H \from I_{\infty} \otimes I_{k} \to I_{\infty}$ be given by $H\left(i,j\right) = i+j$, so that $H\left(i,0\right) = i$ and $H\left(i,k\right) = s\left(i\right)$.
    Then, $\gamma \circ H$ is a path-homotopy from $\gamma$ to $\gamma \circ s$.

    We can now consider the general case where the shrinking map $s \from I_{\infty} \to I_{\infty}$ is arbitrary.
    Since $\gamma$ is a path of length $\infty$, there exist $N_{-}, N_{+} \in \mathbb{Z}$ such that $\gamma\left(i\right) = x$ for all $i \leq N_{-}$ and $\gamma\left(i\right) = x'$ for all $i \geq N_{+}$.
    Let $M_{-}, M_{+} \in \mathbb{Z}$ be such that $s\left(M_{-}\right) = N_{-}$ and $s\left(M_{+}\right) = N_{+}$.

    Define $s' \from I_{\infty} \to I_{\infty}$ as follows:
    \[
        s'\left(i\right) =
        \begin{cases}
            i - M_{-} + N_{-} & \text{if } i \leq M_{-} \\
            s\left(i\right) & \text{if } M_{-} \leq i \leq M_{+} \\
            i - M_{+} + N_{+} & \text{if } M_{+} \leq i
        \end{cases}
    \]
    Then, $s'$ is also a shrinking map.
    Furthermore, $s'$ can be obtained as the composite of a finite sequence of shrinking maps of the forms considered in the previous two special cases.
    Thus, we have a path-homotopy $\gamma \Rightarrow \gamma \circ s'$.
    Finally, we observe that $\gamma \circ s' = \gamma \circ s$.
    \qedhere
\end{proof}

\begin{proposition}
    The set-function
    \[
    \begin{tikzcd}
        \pi_{0}P_{\infty}X\left(x,x'\right) \arrow[r] & \pi_{0}P_{\left\{\infty\right\}}X\left(x,x'\right)
    \end{tikzcd}
    \]
    induced by the quotient map $P_{\infty}X\left(x,x'\right) \to P_{\left\{\infty\right\}}X\left(x,x'\right)$ is a bijection.
\end{proposition}
\begin{proof}
    We need to check that whenever two vertices in $P_{\left\{\infty\right\}}X\left(x,x'\right)$ are path-connected, their preimages in $P_{\infty}X\left(x,x'\right)$ are also path-connected.
    It suffices to consider vertices connected by an edge in $P_{\left\{\infty\right\}}X\left(x,x'\right)$.
    
    Let $\left[\gamma\right] \sim \left[\sigma\right]$ be an edge in $P_{\left\{\infty\right\}}X\left(x,x'\right)$.
    Then, there exists an edge $\tilde{\gamma} \sim \tilde{\sigma}$ in $P_{\infty}X\left(x,x'\right)$ such that $\left[\gamma\right] = \left[\tilde{\gamma}\right]$ and $\left[\sigma\right] = \left[\tilde{\sigma}\right]$.
    By \cref{lem:reparametrization-v-path-homotopy}, we have paths $\gamma \rightsquigarrow \tilde{\gamma}$ and $\sigma \rightsquigarrow \tilde{\sigma}$ in $P_{\infty}X\left(x,x'\right)$.
    Thus, there is a path $\gamma \rightsquigarrow \sigma$ in $P_{\infty}X\left(x,x'\right)$.
\end{proof}

This allows us to use the path graph $P_{\left\{\infty\right\}}X\left(x,x'\right)$, or indeed the path graph $P_{\mathbb{N}}X\left(x,x'\right)$ which is isomorphic to it, in place of $P_{\infty}X\left(x,x'\right)$, as long as we are only concerned with the set of path-connected components of these graphs.
We expect that each of these models will be useful in different contexts.
For instance, the path graph $P_{\mathbb{N}}X\left(x,x'\right)$ will come in handy when we need to use induction arguments on path-lengths.

Next, we can define loop graphs which, just like path graphs, are indexed by a subset $\mathcal{I} \subseteq \mathbb{N} \cup \infty$.
In the context of $\times$-homotopy theory, a similar definition can be found in \cite[Def.~4.2]{dochtermann:homotopy-groups}.
The remainder of this section develops properties of the loop graphs, several of which have analogues in the context of $\times$-homotopy theory \cite[\S4]{dochtermann:homotopy-groups}.

\begin{definition}
    For a pointed graph $\left(X,x_{0}\right)$ and a non-empty subset $\mathcal{I} \subseteq \mathbb{N} \cup \left\{\infty\right\}$, the $\mathcal{I}$-indexed loop graph $\Omega_{\mathcal{I}}\left(X,x_{0}\right)$ is given by $P_{\mathcal{I}}X\left(x_{0},x_{0}\right)$.
    It has a distinguished vertex given by the equivalence class of the constant path at $x_{0}$.
    This gives an endofunctor $\Omega_{\mathcal{I}} \from \Graph_{*} \to \Graph_{*}$.
\end{definition}

Iterating the endofunctor $\Omega_{\mathcal{I}} \from \Graph_{*} \to \Graph_{*}$, we get:
\[
    \Omega_{\mathcal{I}}^{d} \from \Graph_{*} \to \Graph_{*}, \qquad \text{for } d \in \mathbb{N}.
\]
Suppose we have an inclusion $\mathcal{I} \subseteq \mathcal{J}$, where $\mathcal{I}$, $\mathcal{J}$ are both non-empty subsets of $\mathbb{N} \cup \left\{\infty\right\}$.
Then, we have natural transformations
\[
    \Omega_{\mathcal{I}}^{d} \Rightarrow \Omega_{\mathcal{J}}^{d}, \qquad \text{for } d \in \mathbb{N}.
\]

\begin{proposition}
    For every pointed graph $\left(X,x_{0}\right)$, the following maps are isomorphisms:
    \[
        \pushQED{\qed}
        \Omega_{\mathbb{N}}^{d}\left(X,x_{0}\right) \to \Omega_{\mathbb{N} \cup \left\{\infty\right\}}^{d}\left(X,x_{0}\right) \qquad \text{and} \qquad \Omega_{\left\{\infty\right\}}^{d}\left(X,x_{0}\right) \to \Omega_{\mathbb{N} \cup \left\{\infty\right\}}^{d}\left(X,x_{0}\right).
        \qedhere
        \popQED
    \]
\end{proposition}

\begin{notation}
    Given a vector $\mathbf{n} = \left(n_{1},\ldots,n_{d}\right) \in \left(\mathbb{N} \cup \left\{\infty\right\}\right)^{d}$, let $I_{\mathbf{n}}$ denote the graph $I_{n_{1}} \gtimes \cdots \gtimes I_{n_{d}}$, and $\mathbf{n} \setminus n_{i}$ denote the vector $\left(n_{1},\ldots,\hat{n_{i}},\ldots,n_{d}\right) \in \left(\mathbb{N} \cup \left\{\infty\right\}\right)^{d-1}$.
    Let $\Omega_{\mathbf{n}} \from \Graph_{*} \to \Graph_{*}$ denote the endofunctor obtained as the composite $\Omega_{n_{d}} \circ \cdots \circ \Omega_{n_{1}}$.
\end{notation}

\begin{proposition}
    For any pointed graph $\left(X,x_{0}\right)$ and vector $\mathbf{n} = \left(n_{1},\ldots,n_{d}\right) \in \left(\mathbb{N} \cup \left\{\infty\right\}\right)^{d}$, the graph $\Omega_{\mathbf{n}}\left(X,x_{0}\right)$ is isomorphic to the full subgraph of $\ghom{I_{\mathbf{n}}}{X}$ with vertices given by maps $f \from I_{\mathbf{n}} \to X$ for which each component
    \[
        f_{i} \from I_{n_{i}} \to \ghom{I_{\mathbf{n}\setminus n_{i}}}{X}, \qquad \text{for } i = 1, \ldots, d
    \]
    given by
    \[
        t \mapsto \left( \left(t_{1},\ldots,\hat{t_{i}},\ldots,t_{d}\right) \mapsto f\left(t_{1},\ldots,t,\ldots,t_{d}\right) \right)
    \]
    is a path of length $n_{i}$ from $c_{x_{0}}$ to $c_{x_{0}}$, where $c_{x_{0}}$ is the constant map at $x_{0}$. \qed
\end{proposition}

Given any subset $\mathcal{I} \subseteq \mathbb{N} \cup \left\{\infty\right\}$ and $d \in \mathbb{N}$, we have an equivalence relation $\sim_{\mathcal{I}^{d}}$ on the disjoint union $\coprod_{\mathbf{n} \in \mathcal{I}^{d}}{\Omega_{\mathbf{n}}\left(X,x_{0}\right)}$ that is generated by (higher-order) reparametrizations.
That is, given an element $f \in \Omega_{\mathbf{n}}\left(X,x_{0}\right)$ and shrinking maps $s_{i} \from I_{m_{i}} \to I_{n_{i}}$, $i = 1, \ldots, d$, for $\mathbf{m}, \mathbf{n} \in \mathcal{I}^{d}$, we identify $f \sim_{\mathcal{I}^{d}} f \circ \left(s_{1} \gtimes \cdots s_{d}\right)$ and take the smallest equivalence relation containing these pairs.

\begin{proposition}
    For any pointed graph $\left(X,x_{0}\right)$, non-empty subset $\mathcal{I} \subseteq \mathbb{N} \cup \left\{\infty\right\}$ and $d \in 
    \mathbb{N}$, we have:
    \[
        \pushQED{\qed}
        \Omega_{\mathcal{I}}^{d}\left(X,x_{0}\right) \cong \left(\coprod_{\mathbf{n} \in \mathcal{I}^{d}}{\Omega_{\mathbf{n}}\left(X,x_{0}\right)}\right) / \sim_{\mathcal{I}^{d}}.
        \qedhere
        \popQED
    \]
\end{proposition}

\begin{lemma} \label{lem:paths-in-Omega-reparametrization}
    Given $f \in \Omega_{\infty}^{d}\left(X,x_{0}\right)$ and shrinking maps $s_{1}, \ldots, s_{d} \from I_{\infty} \to I_{\infty}$, there exists a path $f \rightsquigarrow f \circ \left(s_{1} \gtimes \cdots \gtimes s_{d}\right)$ in $\Omega_{\infty}^{d}\left(X,x_{0}\right)$.
\end{lemma}
\begin{proof}
    Note that $\left(s_{1} \gtimes \cdots \gtimes s_{d}\right)$ equals the following composite:
    \[
        \left(s_{1} \gtimes \text{id} \gtimes \cdots \gtimes \text{id} \right) \circ \left(\text{id} \gtimes s_{2} \gtimes \text{id} \gtimes \cdots \gtimes \text{id} \right) \circ \cdots \circ \left(\text{id} \gtimes \cdots \gtimes \text{id} \gtimes s_{d} \right)
    \]
    Thus, it suffices to prove that there exists a path $f \rightsquigarrow f \circ \left(\text{id} \gtimes \cdots \gtimes s_{i} \gtimes \cdots \gtimes \text{id} \right)$ for any $f$ and any $i = 1, \ldots, d$.

    Let $f_{i} \from I_{\infty} \to \ghom{\gexp{I_{\infty}}{d-1}}{X}$ be given by
    \[
        t \mapsto \left(\left(t_{1}, \ldots, \hat{t_{i}},\ldots,t_{d}\right) \mapsto f\left(t_{1},\ldots,t,\ldots,t_{d}\right) \right).
    \]
    By \cref{lem:reparametrization-v-path-homotopy}, there is a path-homotopy $f_{i} \Rightarrow f_{i} \circ s_{i}$, or equivalently a path $f \rightsquigarrow f \circ \left(\text{id} \gtimes \cdots \gtimes s_{i} \gtimes \cdots \gtimes \text{id} \right)$ in $\Omega_{\infty}^{d}\left(X,x_{0}\right)$.
\end{proof}

\begin{proposition} \label{quotient-by-shrinking-is-iso}
    The set-function
    \[
    \begin{tikzcd}
        \pi_{0}\Omega_{\infty}^{d}\left(X,x_{0}\right) \arrow[r] & \pi_{0}\Omega_{\left\{\infty\right\}}^{d}\left(X,x_{0}\right)
    \end{tikzcd}
    \]
    induced by the quotient map $\Omega_{\infty}^{d}\left(X,x_{0}\right) \to \Omega_{\left\{\infty\right\}}^{d}\left(X,x_{0}\right)$ is a bijection.
\end{proposition}
\begin{proof}
    We need to check that whenever two vertices in $\Omega_{\left\{\infty\right\}}^{d}\left(X,x_{0}\right)$ are path-connected, their preimages in $\Omega_{\infty}^{d}\left(X,x_{0}\right)$ are also path-connected.
    It suffices to consider vertices connected by an edge in $\Omega_{\left\{\infty\right\}}^{d}\left(X,x_{0}\right)$.

    Let $\left[f\right] \sim \left[g\right]$ be an edge in $\Omega_{\left\{\infty\right\}}^{d}\left(X,x_{0}\right)$.
    Then, there exists an edge $F \sim G$ in $\Omega_{\infty}^{d}\left(X,x_{0}\right)$ with $\left[F\right] = \left[f\right]$ and $\left[G\right] = \left[g\right]$.
    By \cref{lem:paths-in-Omega-reparametrization}, we have paths $f \rightsquigarrow F$ and $g \rightsquigarrow G$ in $\Omega_{\infty}^{d}\left(X,x_{0}\right)$.
    Thus, there is a path $f \rightsquigarrow g$ in $\Omega_{\infty}^{d}\left(X,x_{0}\right)$.
\end{proof}

As a corollary, we obtain the main theorem of this section:

\begin{theorem} \label{equivalence-finite-infinite}
    The $d$-th $A$-homotopy group $A_{d}\left(X,x_{0}\right)$ of a pointed graph $\left(X,x_{0}\right)$ can be equivalently defined as the set of path-components of any of the following graphs:
    \begin{itemize}
        \item $\Omega_{\infty}^{d}\left(X,x_{0}\right)$ of graph maps from the infinite grid to $(X, x_0)$ stabilizing in all directions;
        \item $\Omega_{\left\{\infty\right\}}^{d}\left(X,x_{0}\right)$ which is the quotient of the graph $\Omega_{\infty}^{d}\left(X,x_{0}\right)$ above by the equivalence relation generated by reparametrization; or
        \item $\Omega_{\mathbb{N}}^{d}\left(X,x_{0}\right)$ of graph maps from finite grids to $(X,x_0)$ quotiented by the equivalence relation generated by reparametrization.
    \end{itemize}
\end{theorem}

Note that, for $d \geq 2$, the group operation on $A_{d}\left(X,x_{0}\right)$ must necessarily descend from the binary operations on the vertex set of $\Omega_{\infty}^{d}\left(X,x_{0}\right)$, since none of the other models admit the correct binary operation.
The only exception to this is the case $d = 1$, where $\Omega_{\mathbb{N}}\left(X,x_{0}\right)$ admits a product structure in $\Graph$.
In particular, \cref{quotient-by-shrinking-is-iso} cannot be strengthened to a statement about isomorphic groups for $d \geq 2$.

\section{The fundamental groupoid} \label{sec:fund-gpd}

In this section, we define the fundamental groupoid $\Pi_1 X$ of a graph $X$ (see \cref{def:fund-gpd}).
Its objects are the vertices of $X$, and its hom-sets are given by the set of path-components of the path-graphs constructed in \cref{def:path-graphs}.
Taking advantage of the `finite' grid approach of \cref{sec:fin-fomrulation}, we define the operations making $\Pi_{1}X$ into a groupoid using maps between the path-graphs that encode concatentation and inversion of paths, and verify that it gives a well-defined functor from the category of graphs to the category of groupoids.
We then show that the fundamental groupoid functor is monoidal with respect to the box product and hence, homotopy invariant.

Recall the definition of concatenation of paths having finite length.
For $m, n \in \mathbb{N}$ and vertices $x, x', x'' \in X$, we have a well-defined graph map
\[
    \ast \from P_{m}X\left(x,x'\right) \times P_{n}X\left(x',x''\right) \to P_{m+n}X\left(x,x''\right) 
\]
given by
\[
    \left(\gamma, \sigma\right) \mapsto \gamma \ast \sigma
\]

Given two shrinking maps $s \from I_{m'} \to I_{m}$ and $t \from I_{n'} \to I_{n}$, where $m,m',n, n' \in \mathbb{N}$, their \emph{concatenation} is the shrinking map $s \ast t \from I_{m'+n'} \to I_{m+n}$ given by
\[
    \left(s \ast t\right)\left(i\right) = \begin{cases}
        s\left(i\right) & \text{if } 0 \leq i \leq m' \\
        t\left(i-m'\right) + m & \text{if } m' \leq i \leq m'+n'
    \end{cases}
\]

Given paths $\gamma \from x \rightsquigarrow x'$ and $\sigma \from x' \rightsquigarrow x''$ in $X$ of lengths $m,n$ respectively, and shrinking maps $s \from I_{m'} \to I_{m}$ and $t \from I_{n'} \to I_{n}$, we have
\[
    \left(\gamma \circ s\right) \ast \left(\sigma \circ t\right) = \left(\gamma \ast \sigma \right) \circ \left(s \ast t\right).
\]
Thus, we obtain an induced graph map
\[
    \ast \from P_{\mathbb{N}}X\left(x,x'\right) \times P_{\mathbb{N}}X\left(x',x''\right) \to P_{\mathbb{N}}X\left(x,x''\right) 
\]
given by
\[
    \left(\left[\gamma\right], \left[\sigma\right]\right) \mapsto \left[\gamma \ast \sigma\right]
\]

Note that concatenation of paths of length $\infty$ does not define a graph map.

\begin{lemma}\label{path-concat-assoc-unital}
    Let $X$ be any graph.
    Then, the following diagrams of graph maps commute:
    \[
    \begin{tikzcd}
        P_{\mathbb{N}}X\left(x,x'\right) \gtimes P_{\mathbb{N}}X\left(x',x''\right) \gtimes P_{\mathbb{N}}X\left(x'',x'''\right) \arrow[r,"\ast \gtimes \mathrm{id}"] \arrow[d,"\mathrm{id} \gtimes \ast"'] & P_{\mathbb{N}}X\left(x,x''\right) \gtimes P_{\mathbb{N}}X\left(x'',x'''\right) \arrow[d,"\ast"] \\
        P_{\mathbb{N}}X\left(x,x'\right) \gtimes P_{\mathbb{N}}X\left(x',x'''\right) \arrow[r,"\ast"] & P_{\mathbb{N}}X\left(x,x'''\right)
    \end{tikzcd}
    \]
    \[
    \begin{tikzcd}
        P_{\mathbb{N}}X\left(x,x'\right) \times I_{0} \arrow[dr,"\cong"'] \arrow[r,"\mathrm{id} \times \mathrm{c}_{x'}"] & P_{\mathbb{N}}X\left(x,x'\right) \times P_{\mathbb{N}}X\left(x',x'\right) \arrow[d,"\ast"] \\
        & P_{\mathbb{N}}X\left(x,x'\right)
    \end{tikzcd}
    \qquad
    \begin{tikzcd}
        P_{\mathbb{N}}X\left(x,x\right) \times P_{\mathbb{N}}X\left(x,x'\right) \arrow[d,"\ast"] &  I_{0} \times P_{\mathbb{N}}X\left(x,x'\right) \arrow[dl,"\cong"] \arrow[l,"\mathrm{c}_{x} \times \mathrm{id}"'] \\
        P_{\mathbb{N}}X\left(x,x'\right) &
    \end{tikzcd}
    \]
    for all $x,x', x'',x''' \in X$.
    for all $x,x', x'',x''' \in X$.
    \qed
\end{lemma}

For $n \in \mathbb{N}$ and vertices $x, x' \in X$, we have a well-defined graph map
\[
    \overline{(~)} \from P_{n}X\left(x,x'\right) \to P_{n}X\left(x', x\right)
\]
given by taking the inverse path $\gamma \mapsto \overline{\gamma}$ in the sense of \cref{def:inverse-path}.

Given a shrinking map $s \from I_{n'} \to I_{n}$, where $n, n' \in \mathbb{N}$, its \emph{reverse} is the shrinking map $\overline{s} \from I_{n'} \to I_{n}$ given by $\overline{s}\left(i\right) = n - s\left(n'-i\right)$.
Putting it together, given any path $\gamma \from x \rightsquigarrow x'$ in $X$ of length $n$ and a shrinking map $s \from I_{n'} \to I_{n}$, we have
\[
    \overline{\gamma \circ s} = \overline{\gamma} \circ \overline{s}.
\]
which gives an induced graph map
\[
    \overline{(~)} \from P_{\mathbb{N}}X\left(x,x'\right) \to P_{\mathbb{N}}X\left(x',x\right)
\]
given by $\left[\gamma\right] \mapsto     \left[\overline{\gamma}\right]$.

\begin{lemma}\label{path-inverse}
    Let $X$ be any graph.
    Then, the following diagram of set-functions commutes:
    \[
    \begin{tikzcd}
        \pi_{0} P_{\mathbb{N}}X\left(x',x\right) \times
        \pi_{0} P_{\mathbb{N}}X\left(x,x'\right) \arrow[d,"\ast"'] & 
        \pi_{0} P_{\mathbb{N}}X\left(x,x'\right) \arrow[r,"{(\text{id} ,\overline{(~)})}"] \arrow[l,"{(\overline{(~)}, \text{id})}"'] \arrow[d] &
        \pi_{0} P_{\mathbb{N}}X\left(x,x'\right) \times
        \pi_{0} P_{\mathbb{N}}X\left(x',x\right) \arrow[d,"\ast"] \\
        \pi_{0} P_{\mathbb{N}}X\left(x',x'\right) &
        \left\{\ast\right\} \arrow[l,"{[c_{x'}]}"'] \arrow[r,"{[c_{x}]}"] &
        \pi_{0} P_{\mathbb{N}}X\left(x,x\right)
    \end{tikzcd}
    \]
    for all $x, x' \in X$.
    \qed
\end{lemma}

\begin{definition}
    Given any path $\gamma \from x \rightsquigarrow x'$ in a graph $X$ having finite length, let its \emph{path-homotopy class} be its equivalence class in $\pi_{0}P_{\mathbb{N}}X\left(x,x'\right)$.
\end{definition}

\begin{definition}\label{def:fund-gpd}
    The \emph{fundamental groupoid} $\Pi_{1}X$ of a graph $X$ is defined as follows:
    \[
        \Pi_{1}X = \begin{cases}
            \text{objects: } & \text{vertices } x, x', \ldots \text{ of } X \\
            \text{morphisms: } & \text{path-homotopy classes } \left[\gamma\right] \from x \to x' \text{ of paths } \gamma \from x \rightsquigarrow x' \text{ in } X \text{ of finite length} 
        \end{cases}
    \]
    In other words, we have:
    \[
        \ob{\Pi_{1}X} = \V{X}
    \]
    and for every $x, x' \in \ob{\Pi_{1}X}$, we have:
    \[
        \Pi_{1}X\left(x,x'\right) = \pi_{0}P_{\mathbb{N}}X\left(x,x'\right).
    \]

    Composition is given by $\left[\sigma\right] \circ \left[\gamma\right] = \left[\gamma \ast \sigma \right]$, which is well-defined since it is obtained by applying $\pi_{0}$ to the map
    \[
        \ast \from P_{\mathbb{N}}X\left(x,x'\right) \times P_{\mathbb{N}}X\left(x',x''\right) \to P_{\mathbb{N}}X\left(x,x''\right) 
    \]
    \[
        \left(\left[\gamma\right], \left[\sigma\right]\right) \mapsto \left[\gamma \ast \sigma\right].
    \]
    
    Inverses are given by $\left[\gamma\right]^{-1} = \left[\overline{\gamma}\right]$, which is well-defined since it is obtained by applying $\pi_{0}$ to the map
    \[
        \overline{(~)} \from P_{\mathbb{N}}X\left(x,x'\right) \to P_{\mathbb{N}}X\left(x',x\right)
    \]
    \[
        \left[\gamma\right] \mapsto \left[\overline{\gamma}\right].
    \]

    Then, $\Pi_{1}G$ is a well-defined groupoid.

    Given any graph map $f \from X \to Y$, we have a well-defined functor $\Pi_{1}f \from \Pi_{1}X \to \Pi_{1}Y$ that maps a morphism $\left[\gamma\right] \from x \to x'$ in $\Pi_{1}X$ to the morphism $\left[f \circ \gamma\right] \from f\left(x\right) \to f\left(x'\right)$ in $\Pi_{1}Y$.
    Furthermore, the assignment $f \mapsto \Pi_{1}f$ is itself functorial, and we obtain a well-defined functor
    \[
        \Pi_{1} \from \Graph \to \Gpd.
    \]
\end{definition}

\begin{proposition}\label{A1-v-Pi1-connected}
    For any pointed graph $\left(X,x_{0}\right)$, we have a group isomorphism
    \[
        A_{1}\left(X,x_{0}\right) \cong \Pi_{1}X\left(x_{0},x_{0}\right)
    \]
    If $X$ is connected, the inclusion $\{x_{0}\} \hookrightarrow X$ induces an equivalence of categories between the fundamental group $A_{1}\left(X,x_{0}\right)$ and the fundamental groupoid $\Pi_{1}X$.
    \qed
\end{proposition}

\subsection{Monoidality}

Given two graphs $X$ and $Y$, we have the following groupoids:
\[
    \Pi_{1}\left(X \gtimes Y\right) \qquad \text{and} \qquad \Pi_{1}X \times \Pi_{1}Y
\]
They have identical objects: pairs $\left(x,y\right)$ where $x$ is a vertex in $X$ and $y$ is a vertex in $Y$.
However, they might a priori differ in their morphisms.
Our next goal is to show that these two groupoids are in fact isomorphic.

We begin by observing that a morphism $\left(x,y\right) \to \left(x',y'\right)$ in $\Pi_{1}\left(X \gtimes Y\right)$ is given by the path-homotopy class $\left[\gamma\right]$ of a path $\gamma \from \left(x,y\right) \rightsquigarrow \left(x',y'\right)$ in the box product $X \gtimes Y$.
On the other hand, a morphism $\left(x,y\right) \to \left(x',y'\right)$ in $\Pi_{1}X \times \Pi_{1}Y$ is given by a pair $\left(\left[\sigma\right],\left[\tau\right]\right)$ where $\left[\sigma\right]$ is the path-homotopy class of a path $\sigma \from x \rightsquigarrow x'$ in $X$ and $\left[\tau\right]$ is the path-homotopy class of a path $\tau \from y \rightsquigarrow y'$ in $Y$.

\begin{definition}
    Given graphs $X$ and $Y$, let
    \[
        \Psi_{X,Y} \from \Pi_{1}\left(X \gtimes Y\right) \to \Pi_{1}X \times \Pi_{1}Y
    \]
    be the functor whose components are given by $\Pi_{1}$ applied to $\pi_{X} \from X \gtimes Y \to X$ and $\pi_{Y} \from X \gtimes Y \to Y$, respectively.
    It maps a morphism $\left[\gamma\right] \from \left(x,y\right) \to \left(x',y'\right)$ to the morphism $\left(\left[\pi_{X}\gamma\right],\left[\pi_{Y}\gamma\right]\right) \from \left(x,y\right) \to \left(x',y'\right)$.
\end{definition}

These functors, as $X$ and $Y$ vary, satisfy certain coherence conditions and equip $\Pi_{1} \from \Graph \to \Gpd$ with the structure of a colax monoidal functor.
We want to show that the $\Psi_{X,Y}$ are isomorphisms, which would mean that $\Pi_{1} \from \Graph \to \Gpd$ is, in fact, a strong monoidal functor.

\begin{definition}
    Given any two graphs $X$ and $Y$, let
    \[
        \Phi_{X,Y} \from \Pi_{1}X \times \Pi_{1}Y \to \Pi_{1}\left(X \gtimes Y\right)
    \]
    be the functor that maps a morphism $\left(\left[\sigma\right],\left[\tau\right]\right) \from \left(x,y\right) \to \left(x',y'\right)$ in $\Pi_{1}X \times \Pi_{1}Y$ to the morphism $\left[\left(\sigma \gtimes y\right) \ast \left(x' \gtimes \tau\right)\right] \from \left(x,y\right) \to \left(x',y'\right)$ in $\Pi_{1}\left(X \gtimes Y\right)$.
\end{definition}

It is straightforward to verify that $\Psi_{X,Y} \circ \Phi_{X,Y}$ equals the identity functor on the groupoid $\Pi_{1}X \times \Pi_{1}Y$.
It only remains to show that $\Phi_{X,Y} \circ \Psi_{X,Y}$ equals the identity functor on the groupoid $\Pi_{1}\left(X \gtimes Y\right)$.
In particular, we need to show that given any path $\gamma \from \left(x,y\right) \rightsquigarrow \left(x',y'\right)$ in the box product $X \gtimes Y$, we have:
\[
    \left[\left(\pi_{X}\gamma \gtimes y\right) \ast \left(x' \gtimes \pi_{Y}\gamma\right)\right] = \left[\gamma\right].
\]

\begin{lemma} \label{lem:paths-in-Im-times-In}
    Let $m, n \in \mathbb{N}$.
    Given any two paths $p,q \from \left(0,0\right) \rightsquigarrow \left(m,n\right)$ of length $m+n$ in the graph $I_{m} \gtimes I_{n}$, we have $\left[p\right] = \left[q\right]$ in $\pi_{0}P_{\mathbb{N}}\left(I_{m} \gtimes I_{n}\right)\left(\left(0,0\right),\left(m,n\right)\right)$.
\end{lemma}
\begin{proof}
    Recall that we have a deformation retract of $I_{m} \gtimes I_{n}$ onto the vertex $\left(0,0\right)$ that can be obtained by first contracting along the second variable and then contracting along the first variable.
    Explicitly, we have a graph map $H \from I_{m} \gtimes I_{n} \gtimes I_{m+n} \to I_{m} \gtimes I_{n}$ given by
    \[
        H\left(i,j,k\right) =
        \begin{cases}
            \left(i,\min{\left(j,n-k\right)}\right) & \text{if } 0 \leq k \leq n \\
            \left(\min{\left(i,m+n-k\right)},0\right) & \text{if } n \leq k \leq m+n
        \end{cases}
    \]
    that satisfies:
    \[
        H\left(-,-,0\right) = \text{id}_{I_{m} \gtimes I_{n}}, \qquad H\left(-,-,m+n\right) = c_{\left(0,0\right)},
    \]
    and
    \[
        H\left(0,0,k\right) = \left(0,0\right), \qquad \text{for all } k = 1, \ldots, m+n.
    \]
    Then, the composite
    \[
    \begin{tikzcd}
        I_{2m+2n} \gtimes I_{m+n} \arrow[rrr, "\left(p \ast \overline{q}\right) \gtimes \text{id}_{I_{m+n}}"] &&& I_{m} \gtimes I_{n} \gtimes I_{m+n} \arrow[r,"H"] & I_{m} \gtimes I_{n}
    \end{tikzcd}
    \]
    defines a path-homotopy $p \ast \overline{q} \Rightarrow c_{\left(0,0\right)}$, or equivalently a path $p \ast \overline{q} \rightsquigarrow c_{\left(0,0\right)}$ in $\Omega_{2m+2n}\left(I_{m} \gtimes I_{n},\left(0,0\right)\right)$.
    Hence, we have:
    \begin{align*}
        \left[p\right] &= \left[p \ast c_{\left(k,l\right)}\right] \\
        &= \left[p \ast \overline{q} \ast q \right] \\
        &= \left[c_{\left(0,0\right)} \ast q\right] \\
        &= \left[q\right]
    \end{align*}
    in $\pi_{0}P_{\mathbb{N}}\left(I_{m} \gtimes I_{n}\right)\left(\left(0,0\right),\left(m,n\right)\right)$.
\end{proof}

\begin{lemma}
    Given any path $\gamma \from \left(x,y\right) \rightsquigarrow \left(x',y'\right)$ of finite length $n$ in $X \gtimes Y$, we have:
    \[
        \left[\left(\pi_{X}\gamma \gtimes y\right) \ast \left(x' \gtimes \pi_{Y}\gamma\right)\right] = \left[\gamma\right]
    \]
    in $\pi_{0}P_{\mathbb{N}}\left(X \gtimes Y\right)\left(\left(x,y\right),\left(x',y'\right)\right)$.
\end{lemma}
\begin{proof}
    Replacing the path $\gamma$ with a reparametrization of $\gamma$ if necessary, we can suppose, without loss of generality, that $\gamma\left(i\right) \neq \gamma\left(i-1\right)$ for any $i = 1, \ldots, n$.

    Let $i_{1}, \ldots, i_{k} \in \left\{1, \ldots, n\right\}$ be all the indices for which the maps $e_{i_{1}}^{*}\left(\gamma\right), \ldots, e_{i_{k}}^{*}\left(\gamma\right) \from I_{1} \to X \gtimes Y$ are constant in the second variable (see \cref{def:edge-of-path} for the notation $e_i^*$).
    Let $j_{1}, \ldots, j_{l} \in \left\{1, \ldots, n\right\}$ be all the indices for which the maps $e_{j_{1}}^{*}\left(\gamma\right), \ldots, e_{j_{l}}^{*}\left(\gamma\right) \from I_{1} \to X \gtimes Y$ are constant in the first variable.

    We define a path $\sigma \from x \rightsquigarrow x'$ in $X$ of length $k$ as follows:
    \[
        \sigma = \left(\pi_{X}\circ e_{i_{1}}^{*}\left(\gamma\right)\right) \ast \cdots \ast \left(\pi_{X}\circ e_{i_{k}}^{*}\left(\gamma\right)\right)
    \]
    and a path $\tau \from y \rightsquigarrow y'$ in $Y$ of length $l$ as follows:
    \[
        \tau = \left(\pi_{Y}\circ e_{j_{1}}^{*}\left(\gamma\right)\right) \ast \cdots \ast \left(\pi_{Y}\circ e_{j_{l}}^{*}\left(\gamma\right)\right)
    \]

    We define a shrinking map $s \from I_{n} \to I_{k}$ as follows:
    \[
        s\left(i\right) =
        \begin{cases}
            0 & \text{if } i = 0 \\
            s\left(i-1\right) & \text{if } e_{i}^{*}\left(\gamma\right) \text{ is constant in the first variable} \\
            s\left(i-1\right) + 1 & \text{if } e_{i}^{*}\left(\gamma\right) \text{ is not constant in the first variable}
        \end{cases}
    \]
    and a shrinking map $t \from I_{n} \to I_{l}$ as follows:
    \[
        t\left(i\right) =
        \begin{cases}
            0 & \text{if } i = 0 \\
            t\left(i-1\right) & \text{if } e_{i}^{*}\left(\gamma\right) \text{ is constant in the second variable} \\
            t\left(i-1\right) + 1 & \text{if } e_{i}^{*}\left(\gamma\right) \text{ is not constant in the second variable}
        \end{cases}
    \]

    Finally, we define a path $p \from \left(0,0\right) \rightsquigarrow \left(k,l\right)$ of length $k + l = n$ in $I_{k} \gtimes I_{l}$ as follows:
    \[
        p\left(i\right) = \left(s\left(i\right),t\left(i\right)\right)
    \]

    By construction, we have the following:
    \[
        \pi_{X} \circ \gamma = \sigma \circ s, \qquad \pi_{Y} \circ \gamma = \tau \circ t, \qquad \text{and} \qquad \gamma = \left(\sigma \gtimes \tau\right) \circ p.
    \]
    Then, we have:
    \begin{align*}
        \left(\left(\pi_{X} \circ \gamma\right) \gtimes y\right) \ast \left(x' \gtimes \left(\pi_{Y} \circ \gamma\right)\right) &= \left(\left(\pi_{X} \circ \gamma\right) \gtimes \left(\pi_{Y} \circ \gamma\right) \right) \circ p_{\bot,n,n} \\
        &= \left(\sigma \circ s\right) \gtimes \left(\tau \circ t\right) \circ p_{\bot,n,n} \\
        &= \left(\sigma \gtimes \tau\right) \circ p_{\bot,k,l} \circ \left(s \ast t\right)
    \end{align*}
    where $p_{\bot,n,n}$ represents the path $\left(\text{id}_{I_{n}} \gtimes 0\right) \ast \left(n \gtimes \text{id}_{I_{n}}\right)$ in $I_{n} \gtimes I_{n}$ and $p_{\bot,k,l}$ represents the path $\left(\text{id}_{I_{k}} \gtimes 0\right) \ast \left(k \gtimes \text{id}_{I_{l}}\right)$ in $I_{k} \gtimes I_{l}$.

    By \cref{lem:paths-in-Im-times-In}, we have $\left[p_{\bot,k,l}\right] = \left[p\right]$ in $\pi_{0}P_{\mathbb{N}}\left(I_{k} \gtimes I_{l}\right)\left(\left(0,0\right),\left(k,l\right)\right)$.
    Consequently, we have:
    \begin{align*}
        \left[\left(\left(\pi_{X} \circ \gamma\right) \gtimes y\right) \ast \left(x' \gtimes \left(\pi_{Y} \circ \gamma\right)\right)\right] &= \left[\left(\sigma \gtimes \tau\right) \circ p_{\bot,k,l} \circ \left(s \ast t\right)\right] \\
        &= \left[\left(\sigma \gtimes \tau\right) \circ p \circ \left(s \ast t\right)\right] \\
        &= \left[\gamma \circ \left(s \ast t\right)\right] \\
        &= \left[\gamma\right].
        \qedhere
    \end{align*}
\end{proof}

\begin{theorem} \label{thm:pi1_monoidal}
    The functor $\Pi_{1} \from \Graph \to \Gpd$ is strong monoidal.
    In particular, given any two graphs $X$ and $Y$, the functors
    \[
        \Psi_{X,Y} \from \Pi_{1}\left(X \gtimes Y\right) \to \Pi_{1}X \times \Pi_{1}Y
    \]
    and
    \[
        \Phi_{X,Y} \from \Pi_{1}X \times \Pi_{1}Y \to \Pi_{1}\left(X \gtimes Y\right)
    \]
    are inverse to each other.
    \qed
\end{theorem}

We can now deduce the enrichment of the category of graphs over the category of groupoids.

\begin{theorem} \label{graph-gpd-enriched}
    The category $\Graph$ of graphs admits an enrichment over $\Gpd$ with hom-groupoids given by $\Pi_{1}\left(\ghom{X}{Y}\right)$.
    Given this enrichment, $\Pi_{1} \from \Graph \to \Gpd$ is a $\Gpd$-enriched functor.
\end{theorem}
\begin{proof}
    This follows from a straightforward application of the change of base lemma for enriched categories \cite[Lem.~3.4.3]{riehl:cat-htpy} to the $\Graph$, which is enriched over itself by \cref{prop:graph-closed-monoidal} as a closed monoidal category, and the functor $\Pi_{1} \from \Graph \to \Gpd$, which is strong monoidal by \cref{thm:pi1_monoidal}.
\end{proof}

A category enriched over $\Gpd$ is precisely a $\left(2,1\right)$-category, with 2-cells given by morphisms in the hom-groupoids.
In the case of $\Graph$, where the hom-groupoids are given by $\Pi_{1}\left(\ghom{X}{Y}\right)$, the 2-cells are given by path-homotopy classes of paths in the graph $\ghom{X}{Y}$ or, equivalently, path-homotopy classes of homotopies between graph maps.

A similar statement can be made about the categorical product in place of the box product, using the fact that the two are weakly equivalent.
From the pont of view of A-homotopy theory, the box product is more natural to work with, which is why we have not pursued this direction in the present paper.

\begin{corollary}\label{thm:htpy-invar}
    $\Graph$ admits the structure of a $\left(2,1\right)$-category, with $0$-cells given by graphs, $1$-cells given by graph maps, and $2$-cells given by path-homotopy classes of homotopies between graph maps.
    Furthermore, $\Pi_{1} \from \Graph \to \Gpd$ is a $\left(2,1\right)$-functor and hence, maps homotopies between graph maps to natural isomorphisms between the induced functors.
    \qed
\end{corollary}

\section{Covering graphs} \label{sec:covering}

In this section, we develop the theory of covering graphs, a discrete analogue of the theory of covering spaces from algebraic topology.

We begin by introducing local isomorphisms and covering maps, giving several equivalent definitions of each of these notions (\cref{lem:equiv-defs-local-iso,lem:equiv-defs-covering}), and studying their basic properties, e.g., the path and homotopy lifting properties (\cref{lem:path-lifting-property,lem:homotopy-lifting-property}, respectively).
Note that the notion that we refer to as the `local isomorphism' below was previously studied in combinatorics, e.g., \cite{leighton,sunada} and in A-homotopy theory specifically, e.g., \cite{barcelo-greene-jarrah-welker:vanishing,hardeman:lifting}, under the name `covering map' which we reserve for a stronger notion.

We then turn our attention to the category of coverings of a fixed graph and show that it is equivalent to the category of $\Set$-valued functors on the fundamental groupoid of the graph (\cref{thm:covers-as-functors}).
From that, we deduce the familiar Galois correspondence between connected coverings of a pointed connected graph and the poset of subgroups of its fundamental group (\cref{cor:Galois-correspondence}).

We conclude this section with the construction of the universal cover of a pointed graph (\cref{univ-cov-exists}), which we take to be the initial object in the category of pointed coverings.

We note here that the key results of this section, including \cref{thm:covers-as-functors,univ-cov-exists,cor:Galois-correspondence}, require the full strength of our definition of a covering graph.
When working with local isomorphisms instead, one needs to impose an additional assumption that the base graph has no 3- and 4-cycles, which by the result of Lutz \cite[Thm.~1.1]{lutz} limits the theory to graphs with no non-trivial homotopy groups above dimension 1.

\subsection{Local isomorphisms and covering maps}

To define local isomorphisms and covering graphs, we need some preliminary definitions.
The first of these is a well-known categorical notion of lifting properties.

\begin{definition}
  Let $\cat{C}$ be any category.
  A morphism $p \from C \to D$ in $\cat{C}$ has the \emph{right-lifting property (RLP)} against a morphism $i \from A \to B$ in $\cat{C}$ if for every commutative square of the form
  \[
  \begin{tikzcd}
      A \arrow[r,"u"] \arrow[d,"i"'] & C \arrow[d,"p"] \\
      B \arrow[r,"v"] & D 
  \end{tikzcd}
  \]
  there exists a morphism $h \from B \to C$, called a \emph{lift}, such that the following diagram commutes:
  \[
  \begin{tikzcd}
      A \arrow[r,"u"] \arrow[d,"i"'] & C \arrow[d,"p"] \\
      B \arrow[r,"v"] \arrow[ur,"h" description, dotted] & D 
  \end{tikzcd}
  \]
  Furthermore, if the lift $h \from B \to C$ is unique, then we say that $p$ has the \emph{unique RLP} against $i$.
\end{definition}

Next, we define certain graph maps that we will be interested in taking lifts against.
These may be thought of as graph-theoretic analogues of the \emph{open-box inclusions} that appear in the theory of cubical Kan complexes.

\begin{definition}
    For any $m, n \in \mathbb{N}$, let $\gbox_{m,n} \from I_{2m+n} \to I_{m} \gtimes I_{n}$ be the graph map given by:
    \[
        \gbox_{m,n}\left(i\right) = 
        \begin{cases}
            \left(m-i,0\right) & \text{if } 0 \leq i \leq m \\
            \left(0,i-m\right) & \text{if } m \leq i \leq m+n \\
            \left(i-m-n,n\right) & \text{if } m+n \leq i \leq 2m+n
        \end{cases}
    \]
\end{definition}

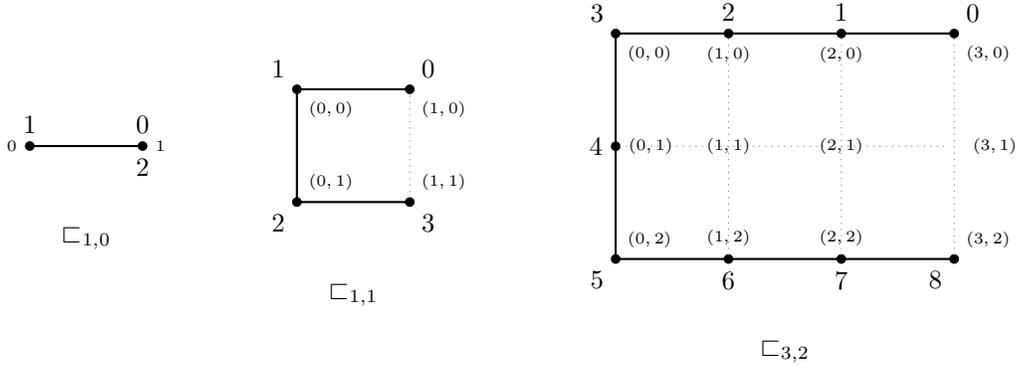
\begin{figure}[H]
\centering
\begin{minipage}{0.25\textwidth}
\centering
\begin{tikzpicture}[colorstyle/.style={circle, draw=black, fill=black, thick, inner sep=0pt, minimum size=1 mm, outer sep=0pt},scale=1.5]
    \node (1) at (0,0) [colorstyle, label = above: {$1$}, label = left: {\tiny $0$}] {};
    \node (2) at (1,0) [colorstyle, label = below: {$2$}, label = above: {$0$}, label = right: {\tiny $1$}] {};
    \draw [thick] (1) -- (2);
    \node at (0.5,-1) [anchor = south]{$\gbox_{1,0}$};
\end{tikzpicture}
\end{minipage}
\begin{minipage}{0.2\textwidth}
\centering
\begin{tikzpicture}[colorstyle/.style={circle, draw=black, fill=black, thick, inner sep=0pt, minimum size=1 mm, outer sep=0pt},scale=1.5]
    \node (1) at (0,1) [colorstyle, label = above left: {$1$}, label = below right: {\tiny $(0,0)$}] {};
    \node (2) at (1,1) [colorstyle, label = above right: {$0$}, label = below right: {\tiny $(1,0)$}] {};
    \node (3) at (0,0) [colorstyle, label = below left: {$2$}, label = above right: {\tiny $(0,1)$}] {};
    \node (4) at (1,0) [colorstyle, label = below right: {$3$}, label = above right: {\tiny $(1,1)$}] {};
    \draw [thick] (4) -- (3);
    \draw [thick] (3) -- (1);
    \draw [thick] (1) -- (2);
    \draw [gray,thin,dotted] (4) -- (2);
    \node at (0.5,-1) [anchor = south]{$\gbox_{1,1}$};
\end{tikzpicture}
\end{minipage}
\begin{minipage}{0.5\textwidth}
\centering
\begin{tikzpicture}[colorstyle/.style={circle, draw=black, fill=black, thick, inner sep=0pt, minimum size=1 mm, outer sep=0pt},scale=1.5]
    \node (0) at (3,0) [colorstyle, label = below left: {$8$}, label = above right: {\tiny $(3,2)$}] {};
    \node (1) at (2,0) [colorstyle, label = below: {$7$}, label = above: {\tiny $(2,2)$}] {};
    \node (2) at (1,0) [colorstyle, label = below: {$6$}, label = above: {\tiny $(1,2)$}] {};
    \node (3) at (0,0) [colorstyle, label = below left: {$5$}, label = above right: {\tiny $(0,2)$}] {};
    \node (4) at (0,1) [colorstyle, label = left: {$4$}, label = right: {\tiny $(0,1)$}] {};
    \node (5) at (0,2) [colorstyle, label = above left: {$3$}, label = below right: {\tiny $(0,0)$}] {};
    \node (6) at (1,2) [colorstyle, label = above: {$2$}, label = below: {\tiny $(1,0)$}] {};
    \node (7) at (2,2) [colorstyle, label = above: {$1$}, label = below: {\tiny $(2,0)$}] {};
    \node (8) at (3,2) [colorstyle, label = above right: {$0$}, label = below right: {\tiny $(3,0)$}] {};
    \node (9) at (3,1) [label = right: {\tiny $(3,1)$}] {};
    \node (10) at (2,1) {\tiny $(2,1)$};
    \node (11) at (1,1) {\tiny $(1,1)$};
    \draw [thick] (0) -- (1);
    \draw [thick] (1) -- (2);
    \draw [thick] (2) -- (3);
    \draw [thick] (3) -- (4);
    \draw [thick] (4) -- (5);
    \draw [thick] (5) -- (6);
    \draw [thick] (6) -- (7);
    \draw [thick] (7) -- (8);
    \draw [gray,thin,dotted] (0) -- (8);
    \draw [gray,thin,dotted] (1) -- (7);
    \draw [gray,thin,dotted] (2) -- (6);
    \draw [gray,thin,dotted] (4) -- (9);
    \node at (1.5,-1) [anchor = south]{$\gbox_{3,2}$};
\end{tikzpicture}
\end{minipage}
\end{figure}

\begin{definition}
    The \emph{(star) neighbourhood} $N_{x}$ of a vertex $x$ in a graph $X$ is a subgraph of $X$ given as follows:
    \begin{align*}
        \V{\left(N_{x}\right)} &= \left\{x' \in \V{X} \ \middle\vert \ x' \sim x \text{ or } x' = x \right\} \\
        \E{\left(N_{x}\right)} &= \left\{ e \in \E{X} \ \middle\vert \ x \in e \right\}
    \end{align*}
\end{definition}

\begin{lemma}\label{RLP-nbhd}
    Let $p \from Y \to X$ be a graph map.
    \begin{enumerate}
        \item The map $p$ has RLP against $\gbox_{1,0} \from I_{1} \to I_{2}$ if and only if the restriction $p|_{N_{y}} \from N_{y} \to N_{p\left(y\right)}$ is injective for all $y \in Y$.
        \item The map $p$ has RLP against $0 \from I_{0} \to I_{1}$ if and only if the restriction $p|_{N_{y}} \from N_{y} \to N_{p\left(y\right)}$ is surjective for all $y \in Y$.
    \end{enumerate}
\end{lemma}
\begin{proof}
\begin{enumerate}
    \item $\left(\Rightarrow\right)$ Let $y \in Y$ and $x = p\left(y\right) \in X$. Suppose there exist $y', y'' \in N_{y}$ such that $p\left(y'\right) = p\left(y''\right) = x' \in N_{x}$.
    We can define a graph map $\lambda \from I_{2} \to Y$ where
    \[
        \lambda\left(0\right) = y', \qquad \lambda\left(1\right) = y, \qquad \lambda\left(2\right) = y''
    \]
    giving us the following lifting problem:
    \[
    \begin{tikzcd}
        I_{2} \arrow[r,"\lambda"] \arrow[d,"\gbox_{1,0}"'] & Y \arrow[d,"p"] \\
        I_{1} \arrow[r,"e"] & X
    \end{tikzcd}
    \]
    where $e \from I_{1} \to X$ is given by $e\left(0\right) = x$ and $e\left(1\right) = x'$.
    By the hypothesis, the above lifting problem has a solution, which implies that
    \[
        y' = \lambda\left(0\right) = \lambda\left(2\right) = y''
    \]
    Thus, $p|_{N_{y}}$ is injective.

    $\left(\Leftarrow\right)$ Consider a lifting problem between $\gbox_{1,0} \from I_{2} \to I_{1}$ and $p \from Y \to X$ as follows:
    \[
    \begin{tikzcd}
        I_{2} \arrow[r,"u"] \arrow[d,"\gbox_{1,0}"'] & Y \arrow[d,"p"] \\
        I_{1} \arrow[r,"v"] & X
    \end{tikzcd}
    \]
    Then, we have two vertices $u\left(0\right), u\left(2\right) \in N_{u\left(1\right)}$ such that $p \circ u \left(0\right) = p \circ u \left(2\right) = v\left(1\right) \in N_{p \circ u\left(1\right)} = N_{v\left(0\right)}$. By the hypothesis, $p|_{N_{u\left(1\right)}}$ is injective, and hence, we must have $u\left(0\right) = u\left(2\right)$. Thus, the above lifting problem has a solution.
    
    \item $\left(\Rightarrow\right)$ Let $y \in Y$ and $x = p\left(y\right) \in X$. Given any $x' \in N_{x}$, let $e \from I_{1} \to X$ be the graph map given by $e\left(0\right) = x$ and $e\left(1\right) = x'$, and consider the following lifting problem:
    \[
    \begin{tikzcd}
        I_{0} \arrow[r,"y"] \arrow[d,"0"'] & Y \arrow[d,"p"] \\
        I_{1} \arrow[r,"e"] & X
    \end{tikzcd}
    \]
    By hypothesis, there exists a solution to the lifting problem, say $\tilde{e} \from I_{1} \to Y$.
    Then, letting $y' = \tilde{e}\left(1\right)$ we have $y' \in N_{y}$ and also
    \[
        p\left(y'\right) = \left(p \circ \tilde{e}\right)\left(1\right) = e\left(1\right) = x'
    \] 
    Thus, $p|_{N_{y}}$ is surjective.
    
    $\left(\Leftarrow\right)$ Consider a lifting problem between $0 \from I_{0} \to I_{1}$ and $p \from Y \to X$ as follows:
    \[
    \begin{tikzcd}
        I_{0} \arrow[r,"u"] \arrow[d,"0"'] & Y \arrow[d,"p"] \\
        I_{1} \arrow[r,"v"] & X
    \end{tikzcd}
    \]
    Then, we have a vertex $v\left(1\right) \in N_{v\left(0\right)} = N_{p \circ u\left(0\right)}$. By the hypothesis, $p|_{N_{u\left(0\right)}}$ is surjective. Hence, there exists some vertex $y' \in N_{u\left(0\right)}$ such that $p\left(y'\right) = v\left(1\right)$. Thus, the above lifting problem has a solution $\tilde{v} \from I_{1} \to Y$ given by $\tilde{v}\left(0\right) = u\left(0\right)$ and $\tilde{v}\left(1\right) = y'$.
    \qedhere
\end{enumerate}
\end{proof}

\begin{lemma} \label{lem:equiv-defs-local-iso}
    Let $p \from Y \to X$ be a graph map.
    The following are equivalent:
    \begin{enumerate}
        \item $p$ has the unique RLP against $0 \from I_{0} \to I_{n}$ for all $n \in \mathbb{N}$.
        \item $p$ has the unique RLP against $0 \from I_{0} \to I_{1}$.
        \item $p$ has RLP against $0 \from I_{0} \to I_{1}$ and $\gbox_{1,0} \from I_{2} \to I_{1}$.
        \item For every vertex $y \in Y$, the restriction $p|_{N_{y}} \from N_{y} \to N_{p\left(y\right)}$ is an isomorphism.
    \end{enumerate}
\end{lemma}
\begin{proof} ~
    \begin{itemize}
    \item[$\left(1\right) \Rightarrow \left(2\right)$] is immediate.
        
    \item[$\left(2\right) \Rightarrow \left(3\right)$] $p$ has RLP against $0 \from I_{0} \to I_{1}$ by assumption. So, it suffices to show that it has RLP against $\gbox_{1,0} \from I_{2} \to I_{1}$. To this end, consider a lifting problem as follows:
    \[
    \begin{tikzcd}
        I_{2} \arrow[r,"u"] \arrow[d,"\gbox_{1,0}"'] & Y \arrow[d,"p"] \\
        I_{1} \arrow[r,"v"] & X
    \end{tikzcd}
    \]
    Note that $\gbox_{1,0}\left(1\right) = 0$ and $\gbox_{1,0}\left(0\right) = 1 = \gbox_{1,0}\left(2\right)$.
    Thus, the above lifting problem admits a solution $h \from I_{1} \to Y$ if and only if we have $u\left(0\right) = u\left(2\right)$.
    
    To see that $u\left(0\right) = u\left(2\right)$, we consider the following diagram:
    \[
    \begin{tikzcd}
       I_{0} \arrow[r,"1"] \arrow[d,"0"'] & I_{2} \arrow[r,"u"] \arrow[d,"\gbox_{1,0}" description] & Y \arrow[d,"p"] \\
        I_{1} \arrow[r,equals] & I_{1} \arrow[r,"v"] & X
    \end{tikzcd}
    \]
    This defines a lifting problem between $0 \from I_{0} \to I_{1}$ and $p \from Y \to X$. There exist two distinct solutions to the lifting problem given by the left part of the diagram, namely $e_{2} \from I_{1} \to I_{2}$ and $\overline{e_{1}} \from I_{1} \to I_{2}$. It follows that $u \circ e_{2}$ and $u \circ \overline{e_{1}}$ are both solutions to the lifting problem between $0 \from I_{0} \to I_{1}$ and $p \from Y \to X$.
    But, since $p$ was assumed to have the unique RLP against $0 \from I_{0} \to I_{1}$, they must be equal. So, we have:
    \[
        u\left(0\right) = \left(u \circ \overline{e_{1}}\right)\left(1\right) = \left(u \circ e_{2}\right)\left(1\right) = u\left(2\right)
    \]
    as required.
        
    \item[$\left(3\right) \Rightarrow \left(4\right)$] This follows from \cref{RLP-nbhd}.
    
    \item[$\left(4\right) \Rightarrow \left(1\right)$] Consider the following lifting problem:
    \[
    \begin{tikzcd}
        I_{0} \arrow[r,"y"] \arrow[d,"0"'] & Y \arrow[d,"p"] \\
        I_{n} \arrow[r,"\gamma"] & X
    \end{tikzcd}
    \]
    We want to construct the unique lift $\tilde{\gamma} \from I_{n} \to Y$.
    We construct it inductively.
    Note that we must have $\tilde{\gamma}\left(0\right) = y$, giving us the base case.
    Suppose we have constructed $\tilde{\gamma}\left(i-1\right) \in Y$ for some $i \in \left\{1,\ldots,n\right\}$.
    Let $y_{i-1} = \tilde{\gamma}\left(i-1\right)$ and $x_{i-1} = \gamma\left(i-1\right)$.
    Then, since $p|_{N_{y_{i-1}}} \from N_{y_{i-1}} \to N_{x_{i-1}}$ is an isomorphism, there is a unique vertex $y_{i} \in N_{y_{i-1}}$ such that $p\left(y_{i}\right) = \gamma\left(i\right) \in N_{x_{i-1}}$.
    Then, we must define $\tilde{\gamma}\left(i\right) = y_{i}$.
    Thus, the unique lift $\tilde{\gamma} \from I_{n} \to Y$ is given by the following recursive formula:
    \[
        \tilde{\gamma}\left(i\right) = \begin{cases}
            y & \text{if } i = 0 \\
            \left(p|_{N_{\tilde{\gamma}\left(i-1\right)}}\right)^{-1}\left(\gamma\left(i\right)\right) & \text{if } i = 1, \ldots, n
        \end{cases}
        \qedhere
    \]
    \end{itemize}
\end{proof}

\begin{definition}
    A map $p \from Y \to X$ is a \emph{local isomorphism} if it satisfies any of the equivalent conditions in \cref{lem:equiv-defs-local-iso}.
\end{definition}

We emphasize, again, that our definition of a local isomorphism captures what other authors, e.g., \cite{barcelo-greene-jarrah-welker:vanishing,hardeman:lifting} refer to as a covering graph.
Some authors also require surjectivity in their definition of covering graphs to avoid dealing with empty fibers, but this additional requirement does not play any serious role in their treatments.
For us, a covering graph is a stronger notion, presented in \cref{def:covering}.

\begin{lemma}[Path lifting property, cf.~{\cite[Thm.~3.0.9]{hardeman:lifting}}] \label{lem:path-lifting-property}
    Let $p \from Y \to X$ be a local isomorphism.
    Given any path $\gamma \from x \rightsquigarrow x'$ in $X$ and any vertex $y \in p^{-1}\left(x\right)$, there exists a unique path $\tilde{\gamma}$ in $Y$ starting at $y$ such that $p \circ \tilde{\gamma} = \gamma$.
\end{lemma}
\begin{proof}
    This follows from the fact that local isomorphisms have the unique RLP against $0 \from I_{0} \to I_{n}$ for all $n \in \mathbb{N}$.
\end{proof}

\begin{lemma} \label{lem:equiv-defs-covering}
    Let $p \from Y \to X$ be a graph map.
    The following are equivalent:
    \begin{enumerate}
        \item[(1)] The induced map $p_{*} \from \ghom{I_{n}}{Y} \to \ghom{I_{n}}{X}$ is a local isomorphism for all $n \in \mathbb{N}$.
        \item[(1')] $p$ has the unique RLP against $0 \gtimes \text{id} \from I_{0} \gtimes I_{n} \to I_{1} \gtimes I_{n}$ for all $n \in \mathbb{N}$.
        \item[(2)] The induced map $p_{*} \from \ghom{I_{n}}{Y} \to \ghom{I_{n}}{X}$ is a local isomorphism for $n = 0, 1$.
        \item[(2')] $p$ has the unique RLP against $0 \gtimes \text{id} \from I_{0} \gtimes I_{n} \to I_{1} \gtimes I_{n}$ for $n = 0, 1$.
        \item[(3)] $p$ has the unique RLP against $0 \from I_{0} \to I_{1}$ and $\gbox_{1,1} \from I_{3} \to I_{1} \gtimes I_{1}$.
        \item[(4)] $p$ has RLP against $0 \from I_{0} \to I_{1}$, $\gbox_{1,0} \from I_{2} \to I_{1}$, and $\gbox_{1,1} \from I_{3} \to I_{1} \gtimes I_{1}$.
    \end{enumerate}
\end{lemma}
\begin{proof}
    By \cref{prop:graph-closed-monoidal}, we have a bijective correspondence between the following lifting problems:
    \[
    \begin{tikzcd}
        I_{0} \gtimes I_{n} \arrow[r] \arrow[d,"0 \gtimes \text{id}"'] & Y \arrow[d,"p"]  \\
        I_{1} \gtimes I_{n} \arrow[r] & X
    \end{tikzcd}
    \quad \leftrightsquigarrow \quad
    \begin{tikzcd}
        I_{0} \arrow[r] \arrow[d,"0"'] & \ghom{I_{n}}{Y} \arrow[d,"p_{*}"] \\
        I_{1} \arrow[r] & \ghom{I_{n}}{X}
    \end{tikzcd}
    \]
    Furthermore, either lifting problem admits a (unique) solution if and only if its counterpart admits a (unique) solution.
    Thus, we have the equivalences $\left(1\right) \Leftrightarrow \left(1'\right)$ and $\left(2\right) \Leftrightarrow \left(2'\right)$.
    
    \begin{itemize}
        \item[$\left(1'\right) \Rightarrow \left(2'\right)$]  is immediate.
        
        \item[$\left(2'\right) \Rightarrow \left(3\right)$] Note that for $n=0$, the map $0 \gtimes \text{id} \from I_{0} \gtimes I_{n} \to I_{1} \gtimes I_{n}$ equals the map $0 \from I_{0} \to I_{1}$. So, it suffices to show that $p$ has the unique RLP against $\gbox_{1,1} \from I_{3} \to I_{1} \gtimes I_{1}$. To this end, consider a lifting problem as follows:
        \[
        \begin{tikzcd}
            I_{3} \arrow[r,"u"] \arrow[d,"\gbox_{1,1}"'] & Y \arrow[d,"p"] \\
            I_{1} \gtimes I_{1} \arrow[r,"v"] & X
        \end{tikzcd}
        \]
        Note that this lifting problem has a unique solution if and only if either $u\left(0\right) \sim u\left(3\right)$ is an edge in $Y$ or if $u\left(0\right) = u\left(3\right)$.

        To see that either $u\left(0\right) \sim u\left(3\right)$ is an edge in $Y$ or $u\left(0\right) = u\left(3\right)$, consider the following diagram:
        \[
        \begin{tikzcd}
            I_{0} \gtimes I_{1} \cong I_{1} \arrow[r,"e_{2}"] \arrow[d,"0 \gtimes \text{id}"'] & I_{3} \arrow[r,"u"] \arrow[d,"\gbox_{1,1}" description] & Y \arrow[d,"p"] \\
            I_{1} \gtimes I_{1} \arrow[r,equal] & I_{1} \gtimes I_{1} \arrow[r,"v"] & X
        \end{tikzcd}
        \]
        This defines a lifting problem between $0 \gtimes \text{id} \from I_{0} \gtimes I_{1} \to I_{1} \gtimes I_{1}$ and $p$. By the assumption, this lifting problem has a unique solution. So, there exists a map $h \from I_{1} \gtimes I_{1} \to Y$ that makes the above diagram commute.
        Thus, either $h\left(1,1\right) \sim h\left(1,0\right)$ is an edge in $Y$ or we have $h\left(1,1\right) = h\left(1,0\right)$. We will now show that $h\left(1,0\right) = u\left(0\right)$ and that $h\left(1,1\right) = u\left(3\right)$.
        
        We know that $h\left(0,0\right) = u\left(1\right)$ and that $p \circ h = v$.
        Thus, both $u \circ \overline{e_{1}} \from I_{1} \to Y$ and $h \circ \left(\text{id} \gtimes 0\right) \from I_{1} \gtimes I_{0} \cong I_{1} \to Y$ are solutions to the following lifting problem:
        \[
        \begin{tikzcd}[column sep = large]
            I_{0} \arrow[r,"u\left(1\right)"] \arrow[d,"0"'] & Y \arrow[d,"p"]  \\
            I_{1} \arrow[r,"v \circ \left(\text{id} \gtimes 0\right)"] & X
        \end{tikzcd}
        \]
        Since $p$ has the unique RLP against $0 \from I_{0} \to I_{1}$ by the assumption, the two solutions must coincide and in particular, we have:
        \[
            h\left(1,0\right) = \left(h \circ \left(\text{id} \gtimes 0\right)\right)\left(1\right) = \left(u \circ \overline{e_{1}}\right)\left(1\right) = u\left(0\right)
        \]
        A similar argument shows that $h\left(1,1\right) = u\left(3\right)$.
        
        \item[$\left(3\right) \Rightarrow \left(1'\right)$] Consider the following lifting problem:
        \[
        \begin{tikzcd}
            I_{0} \gtimes I_{n} \arrow[r,"u"] \arrow[d,"0 \gtimes \text{id}"'] & Y \arrow[d,"p"] \\
            I_{1} \gtimes I_{n} \arrow[r,"v"] & X
        \end{tikzcd}
        \]
        We want to construct a lift $h \from I_{1} \gtimes I_{n} \to Y$ and show that it is unique.
        
        Note that any solution $h \from I_{1} \gtimes I_{n} \to Y$ must necessarily satisfy the condition $h\left(0,i\right) = u\left(0,i\right)$ for $i = 0,\ldots,n$.
        
        For each $i = 0, \ldots, n$, consider the following diagram:
        \[
        \begin{tikzcd}
            I_{0} \gtimes I_{0} \arrow[r,"\text{id} \gtimes i"] \arrow[d,"0 \gtimes \text{id}"'] & I_{0} \gtimes I_{n} \arrow[r,"u"] \arrow[d,"0 \gtimes \text{id}" description] & Y \arrow[d,"p"] \\
            I_{1} \gtimes I_{0} \arrow[r,"\text{id} \gtimes i"] & I_{1} \gtimes I_{n} \arrow[r,"v"] & X
        \end{tikzcd}
        \]
        Since $p$ has the unique RLP against $0 \gtimes \text{id} \from I_{0} \gtimes I_{0} \to I_{1} \gtimes I_{0}$ by hypothesis, there exists a unique lift $h_{i} \from I_{1} \cong I_{1} \gtimes I_{0} \to Y$.
        
        Note that if a solution $h \from I_{1} \gtimes I_{n} \to Y$ to the original lifting problem exists, then the composite $h \circ \left(\text{id} \gtimes i\right)$ is also a solution to the second lifting problem.
        By uniqueness of solutions, we must have $h \circ \left(\text{id} \gtimes i\right) = h_{i}$.
        In other words, we must necessarily define $h\left(1,i\right) = h_{i}\left(1\right)$ for $i = 0, \ldots, n$.
        
        It only remains to verify that for each $i = 1, \ldots, n$, either $h\left(1,i-1\right) \sim h\left(1,i\right)$ is an edge in $Y$ or we have $h\left(1,i-1\right) = h\left(1,i\right)$.
        
        To this end, let $\tilde{h}_{i} \from I_{3} \to Y$ be the map given by the concatenation $\tilde{h}_{i} = \overline{h_{i-1}} \ast u_{i} \ast h_{i}$, where $u_{i} \from I_{1} \to Y$ is the composite $u_{i} = u \circ \left(\text{id} \gtimes e_{i}\right) \from I_{1} \cong I_{0} \gtimes I_{1} \to Y$, for $i = 1, \ldots, n$.
        Observe that:
        \[
            \tilde{h}_{i}\left(0\right) = h_{i-1}\left(1\right) = h\left(i-1,1\right), \qquad \tilde{h}_{i}\left(3\right) = h_{i}\left(1\right) = h\left(i,1\right)
        \]
        We can consider the following lifting problem for $i = 1, \ldots, n$:
        \[
        \begin{tikzcd}[column sep = large]
            I_{3} \arrow[r,"\tilde{h}_{i}"] \arrow[d,"\gbox_{1,1}"'] & Y \arrow[d,"p"] \\
            I_{1} \gtimes I_{1} \arrow[r,"v \circ \left(\text{id} \gtimes e_{i}\right)"] & X
        \end{tikzcd}
        \]
        By hypothesis, $p$ has the unique RLP against $\gbox_{1,1} \from I_{3} \to I_{1} \gtimes I_{1}$, and hence either $\tilde{h}_{i}\left(0\right) \sim \tilde{h}_{i}\left(3\right)$ is an edge in $Y$ or we have $\tilde{h}_{i}\left(0\right) = \tilde{h}_{i}\left(3\right)$, as required.
        
        \item[$\left(3\right) \Leftrightarrow \left(4\right)$] This follows from \cref{lem:equiv-defs-local-iso} and the observation that any solution to a lifting problem against $\gbox_{1,1} \from I_{3} \to I_{1} \gtimes I_{1}$ is necessarily unique.
        \qedhere
    \end{itemize}
\end{proof}

\begin{definition} \label{def:covering}
    A map $p \from Y \to X$ is a \emph{covering map} if it satisfies any of the equivalent conditions in \cref{lem:equiv-defs-covering}.
    In such a case, the graph $X$ is called the \emph{base graph} and the graph $Y$ is called the \emph{covering graph}.
\end{definition}

\begin{remark}
    Every covering map is a local isomorphism.
    A local isomorphism is a covering map if and only if it also has RLP against $\gbox_{1,1} \from I_{3} \to I_{1} \gtimes I_{1}$.
    In particular, any local isomorphism $p \from Y \to X$ where the base graph $X$ contains no 3-cycles or 4-cycles is a covering map.
\end{remark}

\begin{example}
    The graph map $p \from I_{\infty} \to C_{n};\ i \mapsto i\ (\text{mod } n)$ is a local isomorphism for all $n \geq 3$, but a covering map only for $n \geq 5$.
    In particular, this shows how our definition of a covering graph differs from the one given in \cite{barcelo-kramer-laubenbacher-weaver,hardeman:lifting}.
\end{example}

\begin{example}
    The graph maps $C_{3n} \to C_{3};\ i \mapsto i\ (\text{mod } 3)$ and $C_{4n} \to C_{4};\ i \mapsto i\ (\text{mod } 4)$ are both local isomorphisms for every $n \geq 1$, but covering maps only for $n = 1$.
\end{example}

We now prove the homotopy lifting property of covering graphs.
This property is known to fail for local isomorphisms if the base graph contains a 3- or 4-cycle.
In fact, our definition of a covering graph was greatly influenced by closely examining Hardeman's proof of the analogous property for local isomorphisms under the additional assumption that the base graph does not contain any 3- or 4-cycles \cite[Thm.~3.0.10]{hardeman:lifting}.

\begin{lemma}[Homotopy lifting property] \label{lem:homotopy-lifting-property}
    Let $p \from Y \to X$ be a covering map.
    Given any two paths $\gamma, \sigma \from x \rightsquigarrow x'$ in $X$, a vertex $y \in p^{-1}\left(x\right)$, and two paths $\tilde{\gamma}, \tilde{\sigma}$ in $Y$ starting at $y$ such that $p \circ \tilde{\gamma} = \gamma$ and $p \circ \tilde{\sigma} = \sigma$, if we have $\left[\gamma\right] = \left[\sigma\right]$ in $\pi_{0}P_{\mathbb{N}}X\left(x,x'\right)$, then the paths $\tilde{\gamma}$ and $\tilde{\sigma}$ have the same endpoint $y' \in p^{-1}\left(x'\right)$ and we have $\left[\tilde{\gamma}\right] = \left[\tilde{\sigma}\right]$ in $\pi_{0}P_{\mathbb{N}}Y\left(y,y'\right)$.
\end{lemma}
\begin{proof}
    Since we have $\left[\gamma\right] = \left[\sigma\right]$ in $\pi_{0}P_{\mathbb{N}}X\left(x,x'\right)$, there must exist a path-homotopy $H \from I_{m} \gtimes I_{n} \to X$ such that
    \[
        H\left(-,0\right) = \gamma \circ s, \qquad
        H\left(-,n\right) = \sigma \circ t
    \]
    and
    \[
        H\left(0,-\right) = c_{x}, \qquad
        H\left(m,-\right) = c_{x'}
    \]
    where $s$ and $t$ are some shrinking maps.
    
    Consider the following lifting problem:
    \[
    \begin{tikzcd}
        I_{0} \gtimes I_{n} \arrow[r,"c_{y}"] \arrow[d,"0 \gtimes \text{id}"'] & Y \arrow[d,"p"] \\
        I_{m} \gtimes I_{n} \arrow[r,"H"] & X
    \end{tikzcd}
    \]
    
    Since $p \from Y \to X$ is a covering map, the induced map $p_{*} \from \ghom{I_{n}}{Y} \to \ghom{I_{n}}{X}$ is a local isomorphism for all $n \in \mathbb{N}$.
    Hence, the induced map $p_{*} \from \ghom{I_{n}}{Y} \to \ghom{I_{n}}{X}$ has the unique RLP against $0 \from I_{0} \to I_{m}$ for all $m, n \in \mathbb{N}$.
    Equivalently, $p \from Y \to X$ has the unique RLP against $0 \gtimes \text{id} \from I_{0} \gtimes I_{n} \to I_{m} \gtimes I_{n}$ for all $m, n \in \mathbb{N}$. Thus, the lifting problem described earlier has a unique solution. That is, we have a unique map $\tilde{H} \from I_{m} \gtimes I_{n} \to Y$ such that:
    \[
        p \circ \tilde{H} = H, \text{ and} \qquad \tilde{H}\left(0,-\right) = c_{y}.
    \]
    Also note that we have:
    \[
        \tilde{H}\left(-,0\right) = \tilde{\gamma} \circ s, \qquad \tilde{H}\left(-,n\right) = \tilde{\sigma} \circ t.
    \]
    
    Observe that the path $\tilde{H}\left(m,-\right) \from \left(\tilde{\gamma} \circ s\right) \left(n\right) \rightsquigarrow \left(\tilde{\sigma} \circ t\right) \left(n\right)$ satisfies $p \circ \tilde{H}\left(m,-\right) = c_{x'}$. Thus, it must also be a constant path at some $y' \in p^{-1}\left(x'\right)$. It follows that both $\tilde{\gamma}$ and $\tilde{\sigma}$ must have their endpoints at the same vertex $y' \in p^{-1}\left(x'\right)$ and that the map $\tilde{H}$ defines a path-homotopy $\tilde{\gamma} \circ s \Rightarrow \tilde{\sigma} \circ t$.
\end{proof}

\begin{proposition}
    For all $y_{0} \in Y$, a covering map $p \from Y \to X$ induces an isomorphism $A_n(Y, y_0) \to A_n(X, f(y_0)$ for $n \geq 2$ and a monomorphism for $n=1$.
\end{proposition}

\begin{proof}
    A covering map is a fibration of graphs in the sense of \cite[Thm.~5.9]{carranza-kapulkin:cubical-graphs} with discrete fiber.
    Using the long exact sequence of the fibration \cite[Thm.~5.13]{carranza-kapulkin:cubical-graphs}, we obtain the desired isomorphism.
\end{proof}

\subsection{The category of coverings over a fixed base graph}

Our next goal is to define and study properties of the category of coverings of a fixed graph.

\begin{definition} \label{def:map-coverings}
    Let $p \from Y \to X$ and $p' \from Y' \to X$ be two covering maps over the same base $X$.
    A \emph{morphism $f \from p \to p'$ of coverings over $X$} is a graph map $f \from Y \to Y'$ such that the following diagram commutes:
    \[
    \begin{tikzcd}[column sep = small]
        Y \arrow[dr,"p"'] \arrow[rr,"f"] && Y' \arrow[dl,"p'"] \\
        & X &
    \end{tikzcd}
    \]
    For any graph $X$, coverings over $X$ and morphisms between them form a category, denoted $\Cov\left(X\right)$.
    It is the full subcategory of the slice $\Graph \slice X$ on covering maps.
  
    Given a graph map $g \from X' \to X$, we have a functor $g^{*} \from \Cov\left(X\right) \to \Cov\left(X'\right)$.
    It maps a covering $p \from Y \to X$ over $X$ to the pullback $g^{*}p \from g^{*}Y \to X'$, which is a covering over $X'$.
    It is straightforward to verify that $g^{*}$ is a functor.
    Furthermore, the assignment $\left(g \from X' \to X\right) \mapsto \left(g^{*} \from \Cov\left(X\right) \to \Cov\left(X'\right)\right)$ is also functorial, and we have a well-defined functor
    \[
        \Cov \from \Graph^{\op} \to \Cat
    \]
\end{definition}

The next lemma is a well-known categorical fact about cancellation of morphisms with certain lifting property, which we chose to include here for completeness.

\begin{lemma}
    Let $\cat{C}$ be any category.
    Consider the following commutative triangle in $\cat{C}$:
    \[
    \begin{tikzcd}[column sep = small]
        X \arrow[dr,"p"'] \arrow[rr,"q"] && Y \arrow[dl,"r"] \\
        & Z &
    \end{tikzcd}
    \]
    If $p$ and $r$ have the unique RLP against some fixed morphism $i \from A \to B$, then so does $q$.
\end{lemma}
\begin{proof}
    Consider a lifting problem between $i$ and $q$ as follows:
    \[
    \begin{tikzcd}
        A \arrow[r,"u"] \arrow[d,"i"'] & X \arrow[d,"q"] \arrow[dr,"p"] \\
        B \arrow[r,"v"] & Y \arrow[r,"r"] & Z
    \end{tikzcd}
    \]
    Since $p$ has unique $RLP$ against $i$, there exists a unique morphism $h \from B \to X$ such that $h \circ i = u$ and $p \circ h = r \circ v$.

    Observe that both $q \circ h$ and $v$ are solutions to the following lifting problem between $i$ and $r$:
    \[
    \begin{tikzcd}
        A \arrow[r,"q \circ u"] \arrow[d,"i"'] & Y \arrow[d,"r"] \\
        B \arrow[r,"r \circ v"] & Z
    \end{tikzcd}
    \]
    But since $r$ has the unique RLP against $i$, we must have $q \circ h = v$.

    Since we have both $h \circ i = u$ and $q \circ h = v$, it follows that $h \from B \to X$ is a solution to the original lifting problem between $i$ and $q$.
    Furthermore, it is unique since any solution to the lifting problem between $i$ and $q$ is also a solution to the corresponding lifting problem between $i$ and $p$, which admits a unique solution.
\end{proof}

Recalling that covering maps are defined to be those maps that have the unique RLP against $0 \from I_{0} \to I_{1}$ and $\gbox_{1,1} \from I_{3} \to I_{1} \gtimes I_{1}$, we have the following result:

\begin{proposition}
    Let $f \from p \to p'$ be a morphism in $\Cov\left(X\right)$ between two coverings $p \from Y \to X$ and $p' \from Y' \to X$.
    Then the underlying graph map $f \from Y \to Y'$ is itself a covering map.
    \qed
\end{proposition}

\begin{definition}\leavevmode
\begin{enumerate}
    \item Given a local isomorphism $p \from Y \to X$ and a path $\gamma \from x \rightsquigarrow x'$ in $X$, the \emph{unwinding of $\gamma$} is a set-function
    \[
        \mathsf{unw}_{\gamma} \from p^{-1}\left(x\right) \to p^{-1}\left(x'\right)
    \]
    given by
    \[
        y \mapsto \text{ the end-point of the unique lift } \tilde{\gamma} \text{ of } \gamma \text{ starting at } y
    \]
    where $\tilde{\gamma}$ is the unique path in $Y$ that starts at $y$ and satisfies $p \circ \tilde{\gamma} = \gamma$ (cf.~\cref{lem:path-lifting-property}).

    If $p$ is a covering map, by \cref{lem:homotopy-lifting-property}, the function $\mathsf{unw}_{\gamma} \from p^{-1}\left(x\right) \to p^{-1}\left(x'\right)$ only depends on the path-homotopy class of $\gamma$ and so, we may denote it by $\mathsf{unw}_{\left[\gamma\right]} \from p^{-1}\left(x\right) \to p^{-1}\left(x'\right)$.
    
    \item Given any covering map $p \from Y \to X$, let
    \[
        \mathsf{Fib}_{X}\left(p\right)  \from \Pi_{1}{X} \to \Set
    \]
    be the functor that maps a vertex $x \in X$ to the fiber $p^{-1}\left(x\right)$ over $x$, and a morphism $\left[\gamma\right] \from x \to x'$ in $X$ to the set-function $\mathsf{unw}_{\left[\gamma\right]} \from p^{-1}\left(x\right) \to p^{-1}\left(x'\right)$.

    \item Given any graph $X$, let
    \[
        \mathsf{Fib}_{X} \from \Cov\left(X\right) \to \Set^{\Pi_{1}{X}}
    \]
    be the functor that maps an object $p \in \Cov\left(X\right)$ to the functor $\mathsf{Fib}_{X}\left(p\right) \from \Pi_{1}X \to \Set$, and a morphism $f \from p \to p'$ in $\Cov\left(X\right)$ to the induced natural transformation $\mathsf{Fib}_{X}\left(f\right) \from \mathsf{Fib}_{X}\left(p\right) \Rightarrow \mathsf{Fib}_{X}\left(p'\right)$ whose components are given by the restrictions $f|_{p^{-1}\left(x\right)} \from p^{-1}\left(x\right) \to \left(p'\right)^{-1}\left(x\right)$ for $x \in X$.

    \item Given any morphism $g \from X' \to X$ in $\Graph$, the following naturality square commutes up to a natural isomorphism of functors:
    \[
    \begin{tikzcd}
        \Cov\left(X\right) \arrow[r,"\mathsf{Fib}_{X}"] \arrow[d,"g^{*}"'] & \Set^{\Pi_{1}X} \arrow[d,"\left(-\right) \circ g_{*}"] \\
        \Cov\left(X'\right) \arrow[r,"\mathsf{Fib}_{X'}"] & \Set^{\Pi_{1}X'} 
    \end{tikzcd}
    \]
    Let $\mathsf{Fib} \from \Cov \Rightarrow \Set^{\Pi_{1}\left(-\right)}$ be the pseudo-natural transformation whose components are given by $\mathsf{Fib}_{X} \from \Cov\left(X\right) \to \Set^{\Pi_{1}X}$.
\end{enumerate}
\end{definition}

The next definition mimics the standard topological construction of a covering space associated to a functor, cf.~\cite[pp.~63--64]{hatcher}.

\begin{definition}\label{def:total-graph}
Given any graph $X$, the \emph{total graph} $\mathsf{Tot}_{X}F$ of a functor $F \from \Pi_{1}X \to \Set$ is defined as follows:
\begin{align*}
    \V{\left(\mathsf{Tot}_{X}F\right)} &= \coprod_{x \in X}{Fx} \\
    \E{\left(\mathsf{Tot}_{X}F\right)} &=
    \left\{
        y \sim y'
    \ \middle\vert \   
        \begin{aligned}
        x  \sim x' \text{ is an edge in X,} & \qquad \text{where } y \in Fx \text{ and } y' \in Fx' \text{, and} \\
        \left(F\left[e\right]\right)\left(y\right) = y', &\qquad \text{where } e \from I_{1} \to X \text{ is given by } e\left(0\right) = x \text{ and } e\left(1\right) = x' 
        \end{aligned}
    \right\}
\end{align*}
It comes equipped with a graph map $p \from \mathsf{Tot}_{X}F \to X$ that maps $y \in \mathsf{Tot}_{X}F$ to $x \in X$ if $y \in Fx$.
\end{definition}

\begin{proposition}
    Given any graph $X$ and functor $F \from \Pi_{1}X \to \Set$, the graph map $p \from \mathsf{Tot}_{X}F \to X$ is a covering map.
\end{proposition}
\begin{proof}
    We need to verify that $p \from \mathsf{Tot}_{X}F \to X$ has the unique RLP against $0 \from I_{0} \to I_{1}$ and against $\gbox_{1,1} \from I_{3} \to I_{1} \gtimes I_{1}$.
    
    We first consider the following lifting problem against $0 \from I_{0} \to I_{1}$:
    \[
    \begin{tikzcd}
        I_{0} \arrow[r,"y"] \arrow[d,"0"'] &
        \mathsf{Tot}_{X}F \arrow[d,"p"] \\
        I_{1} \arrow[r,"e"] &
        X
    \end{tikzcd}
    \]
    where $e\left(0\right) = p\left(y\right) = x$ and $e\left(1\right) = x'$ for some $x, x' \in X$ such that either $x \sim x'$ is an edge in $X$ or $x = x'$.
    
    Note that any solution $\tilde{e} \from I_{1} \to Y$ to the above lifting problem must satisfy $\tilde{e}\left(0\right) = y$ and $\tilde{e}\left(1\right) = F\left(\left[e\right]\right)\left(y\right)$.
    Since these two conditions define a well-defined graph map, we have a unique solution to the above lifting problem.
    
    Next, we consider the lifting problem:
    \[
    \begin{tikzcd}
        I_{3} \arrow[r,"u"] \arrow[d,"\gbox_{1,1}"'] & \mathsf{Tot}_{X}F \arrow[d,"p"] \\
        I_{1} \gtimes I_{1} \arrow[r,"v"] & X
    \end{tikzcd}
    \]
    We can visualize the maps $u \from I_{3} \to \mathsf{Tot}_{X}F$ and $v \from I_{1} \gtimes I_{1} \to X$ as follows:
    
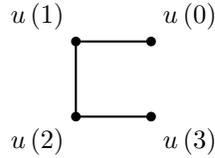
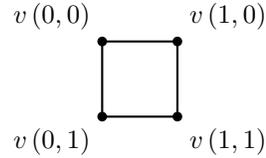
\begin{figure}[H]
    \centering
    \begin{minipage}{0.45\textwidth}
    \centering
    \begin{tikzpicture}[colorstyle/.style={circle, draw=black, fill=black, thick, inner sep=0pt, minimum size=1 mm, outer sep=0pt}, scale=1]
        \node (1) at (0,1) [colorstyle, label = above left: {$u\left(1\right)$}] {};
        \node (2) at (1,1) [colorstyle, label = above right: {$u\left(0\right)$}] {};
        \node (3) at (0,0) [colorstyle, label = below left: {$u\left(2\right)$}] {};
        \node (4) at (1,0) [colorstyle, label = below right: {$u\left(3\right)$}] {};
        \draw [thick] (1) -- (2);
        \draw [thick] (1) -- (3);
        \draw [thick] (3) -- (4);
    \end{tikzpicture}
    \end{minipage}
    \begin{minipage}{0.45\textwidth}
    \centering
    \begin{tikzpicture}[colorstyle/.style={circle, draw=black, fill=black, thick, inner sep=0pt, minimum size=1 mm, outer sep=0pt}, scale=1]
        \node (1) at (0,1) [colorstyle, label = above left: {$v\left(0,0\right)$}] {};
        \node (2) at (1,1) [colorstyle, label = above right: {$v\left(1,0\right)$}] {};
        \node (3) at (0,0) [colorstyle, label = below left: {$v\left(0,1\right)$}] {};
        \node (4) at (1,0) [colorstyle, label = below right: {$v\left(1,1\right)$}] {};
        \draw [thick] (1) -- (2) ;
        \draw [thick] (1) -- (3);
        \draw [thick] (3) -- (4);
        \draw [thick] (2) -- (4);
    \end{tikzpicture}
    \end{minipage}
    \end{figure}
    This lifting problem admits a unique solution if and only if either $u\left(0\right) \sim u\left(3\right)$ is an edge in $\mathsf{Tot}_{X}F$ or $u\left(0\right) = u\left(3\right)$.
    
    Note that we either have an edge $v\left(1,0\right) \sim v\left(1,1\right)$ in $X$ or have $v\left(1,0\right) = v\left(1,1\right)$.
    In either case, it suffices to show that $F\left(\left[e\right]\right)\left(u\left(0\right)\right) = u\left(3\right)$ where $e \from I_{1} \to X$ is given by $e\left(0\right) = v\left(1,0\right)$ and $e\left(1\right) = v\left(1,1\right)$.
    
    Consider the path $\gamma \from I_{3} \to X$ given by the composite
    \[
    \gamma =
    \left(
    \begin{tikzcd}
        I_{3} \arrow[r,"\gbox_{1,1}"] & I_{1} \gtimes I_{1} \arrow[r,"v"] & X
    \end{tikzcd}
    \right)
    \]
    Note that $\left[\gamma\right] = \left[e\right]$.
    
    Thus, we have:
    \begin{align*}
        F\left[e\right]\left(u\left(0\right)\right) &= \left(F\left[\gamma\right]\right)\left(u\left(0\right)\right) \\
        &= F\left[e_{3}^{*}\left(\gamma\right)\right] \circ F\left[e_{2}^{*}\left(\gamma\right)\right] \circ F\left[e_{1}^{*}\left(\gamma\right)\right] \left(u\left(0\right)\right) \\
        &= F\left[e_{3}^{*}\left(\gamma\right)\right] \circ F\left[e_{2}^{*}\left(\gamma\right)\right] \left(u\left(1\right)\right) \\
        &= F\left[e_{3}^{*}\left(\gamma\right)\right] \left(u\left(2\right)\right) \\
        &= u\left(3\right)
    \end{align*}
    as required.
\end{proof}

\begin{definition}
\begin{enumerate}
    \item Given any graph $X$, let
    \[
        \mathsf{Tot}_{X} \from \Set^{\Pi_{1}X} \to \mathsf{Cov}\left(X\right)
    \]
    be the functor that maps a functor $F \from \Pi_{1}X \to \Set$ to the covering map $\mathsf{Tot}_{X}F \to X$, and a natural transformation $\alpha \from F \Rightarrow G$ between two functors $F, G \from \Pi_{1}X \to \Set$ to the induced morphism $\mathsf{Tot}_{X}\left(\eta\right) \from \mathsf{Tot}_{X}\left(F\right) \to \mathsf{Tot}_{X}\left(G\right)$ of coverings over $X$ given by the disjoint union $\coprod_{x \in X}{\alpha_{x}}$.

    \item Given any morphism $g \from X' \to X$ in $\Graph$, the following naturality square commutes up to a natural isomorphism:
    \[
    \begin{tikzcd}
        \Set^{\Pi_{1}{X}} \arrow[r,"\mathsf{Tot}_{X}"] \arrow[d,"\left(-\right) \circ g_{*}"'] & \mathsf{Cov}\left(X\right) \arrow[d,"g^{*}"] \\
        \Set^{\Pi_{1}{X'}} \arrow[r,"\mathsf{Tot}_{X'}"] & \mathsf{Cov}\left(X'\right)
    \end{tikzcd}
    \]
    Let $\mathsf{Tot} \from \Set^{\Pi_{1}{\left(-\right)}} \Rightarrow \mathsf{Cov}$ be the pseudo-natural transformation whose components are given by $\mathsf{Tot}_{X} \from \Set^{\Pi_{1}X} \to \mathsf{Cov}\left(X\right)$.
\end{enumerate}
\end{definition}

\begin{theorem} \label{thm:covers-as-functors}
    For each graph $X$, the functors
    \[
        \mathsf{Fib}_{X} \from \mathsf{Cov}\left(X\right) \to \Set^{\Pi_{1}{X}}
    \]
    and
    \[
        \mathsf{Tot}_{X} \from \Set^{\Pi_{1}{X}} \to \mathsf{Cov}\left(X\right)
    \]
    define an equivalence of categories.
\end{theorem}

\begin{proof}
    Straightforward after unfolding the definitions.
\end{proof}

In case when $(X, x_0)$ is a pointed connected graph, by \cref{A1-v-Pi1-connected}, we may replace the functor category $\Set^{\Pi_{1}{X}}$ by a smaller model $\Set^{A_{1}(X,x_0)}$, i.e., the category of $A_1(X, x_0)$-sets.
  
\newcommand{\Fib}[1]{\mathsf{Fib}_{#1}}
\newcommand{\Tot}[1]{\mathsf{Tot}_{#1}}

\begin{corollary} \label{cor:equiv-A1-set-cov}
    For a pointed connected graph $(X, x_0)$, we have an equivalence of categories
    \[ \Cov(X) \equiv \Set^{A_{1}(X, x_0)} \]
    given by postcomposing the functors $\Fib{X}$ and $\Tot{X}$ with the equivalence of \cref{A1-v-Pi1-connected}. \qed
\end{corollary}

This allows us to recover the usual Galois correspondence between subgroups of the fundamental group of a pointed connected space and the connected coverings thereof.

\begin{corollary} \label{cor:Galois-correspondence}
    For a pointed connected graph $(X, x_0)$, the functors $\Fib{X}$ and $\Tot{X}$ restrict to give an equivalence of categories between the opposite of the poset of subgroups of $A_1(X, x_0)$ and the category of connected covering graphs of $(X, x_0)$, i.e.,
     \[\left\{ \begin{array}{c} \textnormal{subgroups of } A_1(X, x_0) \\ \textnormal{ordered by } \subseteq \end{array} \right\}^\op 
     \equiv 
     \left\{ \begin{array}{c} \textnormal{connected coverings } \\ \textnormal{of } (X, x_0)  \end{array} \right\}\text{.}\]
     A connected cover $p \from (Y,y_{0}) \to (X,x_{0})$ corresponds to the subgroup $p_{*}(A_{1}(Y,y_{0})) \leq A_{1}(X,x_{0})$.
     Moreover, the subgroup $p_{*}(A_{1}(Y,y_{0}))$ is normal exactly when the action of $A_{1}(X,x_{0})$ on $p^{-1} (x_0)$ is transitive.
\end{corollary}

\begin{proof}
    By restricting the equivalence of \cref{cor:equiv-A1-set-cov} to connected objects in both categories (i.e., objects $X$ such $\mathrm{Hom}(X, -)$ preserves coproducts), we obtain an equivalence between the categories of transitive $A_1(X,x_0)$-sets and of connected covering graphs of $(X,x_0)$.
    The former is however equivalent to that of quotients of $A_1(X,x_0)$ by its subgroups, and hence equivalent to the opposite of the poset of subgroups of $A_1(X, x_0)$.

    The verification that under this correspondence normal subgroups correspond to connected coverings with a transitive action of the fundamental group follows verbatim the proof given in the classical setting, e.g., \cite[Prop.~1.39.(a)]{hatcher}.
\end{proof}

An analogue of a slightly weaker version of this result was proven in \cite[\S3]{chih-scull:covers}.

\subsection{Universal covers}

We next turn our attention to the construction of the universal cover of a pointed graph $(X, x_0)$, which is the initial object in the category of pointed covering graphs of $(X,x_0)$.
To define it, we need to upgrade the definition of coverings and their maps (\cref{def:covering,def:map-coverings}) to the pointed setting.

\begin{definition} \leavevmode
\begin{enumerate}
    \item Let $\left(X,x_{0}\right)$ be a pointed graph.
    A pointed graph map $p \from \left(Y,y_{0}\right) \to \left(X,x_{0}\right)$ where the underlying map $p \from Y \to X$ is a covering is called a \emph{pointed covering}.

    \item Given two pointed coverings $p \from \left(Y,y_{0}\right) \to \left(X,x_{0}\right)$ and $p' \from \left(Y',y_{0}'\right) \to \left(X,x_{0}\right)$ over the same pointed base $\left(X,x_{0}\right)$, a \emph{morphism $f \from p \to p'$ of pointed coverings over $\left(X,x_{0}\right)$} is a pointed map $f \from \left(Y,y_{0}\right) \to \left(Y',y_{0}'\right)$ such that the underlying map $f \from Y \to Y'$ is a morphism of coverings over $X$.

    \item For any pointed graph $\left(X,x_{0}\right)$, pointed coverings over $X$ and morphisms between them form a category, denoted $\mathsf{Cov}\left(X,x_{0}\right)$.
    It is the full subcategory of $\Graph_{*} \slice X$ on pointed covering maps.

    \item An initial object in the category $\mathsf{Cov}\left(X,x_{0}\right)$ is called a \emph{universal cover}.
\end{enumerate}
\end{definition}

Recall that a graph $X$ is \emph{simply connected} if it is path-connected and if the fundamental group $A_{1}\left(X,x_{0}\right)$ is trivial for any vertex $x_{0} \in X$.

\begin{theorem} \leavevmode \label{univ-cov-exists}
\begin{enumerate}
    \item Every pointed graph $\left(X,x_{0}\right)$ admits a universal cover.
    \item A pointed covering $p \from \left(Y,y_{0}\right) \to \left(X,x_{0}\right)$ is universal if and only if $Y$ is simply connected.
\end{enumerate}
\end{theorem}
\begin{proof}
    We can suppose the base graph $X$ is path-connected without any loss in generality.
    Indeed, if $X$ is not path-connected, then for pointed cover $p \from (Y,y_{0}) \to (X,x_{0})$ to be initial in the category $\mathsf{Cov}\left(X,x_{0}\right)$, the fiber $p^{-1}(x)$ over any vertex $x$ not in the path-component of $x_{0}$ must be empty.
    Thus any statement about a universal cover over $(X,x_{0})$ reduces to a statement about a universal cover over the path-component of $x_{0}$ in $X$.
    
    Assuming $X$ is path-connected, for any local isomorphism $p \from Y \to X$, the graph $Y$ must also be path-connected.
    Thus, the category $\mathsf{Cov}\left(X,x_{0}\right)$ coincides with the category of pointed, connected coverings of $(X,x_{0})$.
    By the Galois correspondence of \cref{cor:Galois-correspondence}, this category is equivalent to the poset of subgroups of $A_{1}(X,x_{0})$ ordered by reverse inclusion.
    This poset has an initial object, namely the trivial subgroup $\left\{[c_{x_{0}}]\right\}$.
    It follows that the category $\mathsf{Cov}\left(X,x_{0}\right)$ also has an initial object.
    Furthermore, a pointed covering $p \from (Y,y_{0}) \to (X,x_{0})$ is universal if and only if the subgroup $p_{*}(A_{1}(Y,y_{0}))$ is trivial.
    Since $p$ is a covering map, the group homomorphism $p_{*} \from A_{1}(Y,y_{0}) \to A_{1}(X,x_{0})$ is injective.
    Thus, a pointed covering $p \from (Y,y_{0}) \to (X,x_{0})$ is universal if and only $A_{1}(Y,y_{0})$ is trivial, i.e. if and only if $Y$ is simply connected.
\end{proof}

\begin{example}\leavevmode
\begin{enumerate}
    \item Every simply connected graph is its own universal cover via the identity map.
    Thus, the cycle graphs $C_{3}$ and $C_{4}$, which are in fact contractible, are their own universal covers.
    Other examples of simply connected graphs include trees (i.e. connected graphs without cycles) and connected, chordal graphs (i.e. connected graphs in which every cycle having four or more vertices has a \emph{chord} -- an edge that is not part of the cycle, but connects two vertices of the cycle).
    \item Recall that the map $p \from I_{\infty} \to C_{n}; i \mapsto i\ (\text{mod } n)$ is a covering map for $n \geq 5$.
    Since $I_{\infty}$ is simply connected, this map is the universal cover of $C_{n}$ for $n \geq 5$.
    For $n = 3, 4$, this map is not a covering map, so there is no question of it being a universal cover.
    We note here that other treatments of this subject, which define universal covers as local isomorphisms with simply connected domains, suffer from the issue of non-uniqueness: for example, both the identity map on $C_{n}$ as well as the map $p \from I_{\infty} \to C_{n}; i \mapsto i\ (\text{mod } n)$ are local isomorphisms with simply connected domains, for $n = 3$ and $4$.
\end{enumerate}
\end{example}

We now proceed to give an explicit description of the universal cover over any pointed graph $(X,x_{0})$, essentially following the construction of \cref{def:total-graph} for the representable functor $\Pi_{1}X(x_{0},-) \from \Pi_{1}X \to \Set$, which corresponds to the action of $A_{1}(X,x_{0})$ on itself by left multiplication.

\begin{definition}
    Given any pointed graph $\left(X,x_{0}\right)$, we construct the graph $\tilde{X}_{x_{0}}$ as follows:
    \begin{align*}
    \V{\left(\tilde{X}_{x_{0}}\right)} &=
    \left\{
    \text{path-homotopy class } \left[\gamma\right]
    \ \middle\vert \
    \gamma \text{ is a path in } X \text{ starting at } x_{0} \right\} \\
    \E{\left(\tilde{X}_{x_{0}}\right)} &=
    \left\{
    \left[\gamma\right] \sim \left[\gamma \ast e\right]
    \ \middle\vert \
    \begin{aligned}
        & \gamma \text{ is a path in } X \text{ starting at } x_{0} \text{, and} \\
        & e \text{ is a path of length 1 in } X, \\
        & \text{starting at the end-point of } \gamma
    \end{aligned}
    \right\}
    \end{align*}
    It has a distinguished vertex given by the path-homotopy class of the constant path at $x_{0}$.
    It also comes equipped with a pointed graph map $p \from \left(\tilde{X}_{x_{0}},\left[\mathrm{c}_{x_{0}}\right]\right) \to \left(X,x_{0}\right)$ that maps $\left[\gamma\right] \in \tilde{X}_{x_{0}}$ to the end-point of $\gamma$ in $X$.
\end{definition}

\begin{proposition}
    Given any pointed graph $\left(X,x_{0}\right)$, the map $p \from \left(\tilde{X}_{x_{0}},\left[\mathrm{c}_{x_{0}}\right]\right) \to \left(X,x_{0}\right)$ is a universal cover.
    \qed
\end{proposition}

\begin{remark}
    Observe that the fiber of $p \from \tilde{X}_{x_{0}} \to X$ over the base point $x_{0} \in X$ is precisely the fundamental group $A_{1}\left(X,x_{0}\right)$.
\end{remark}

\begin{definition}
    Given a pointed graph $\left(X,x_{0}\right)$, let $\mathsf{Fib}_{x_{0}} \from \mathsf{Cov}\left(X\right) \to \Set$ denote the functor that maps a covering map $p \from Y \to X$ to the fiber $p^{-1}\left(x_{0}\right)$ of $p$ over $x_{0}$ and a morphism of coverings $f \from p \to p'$ over $X$ to the restriction $f|_{p^{-1}\left(x_{0}\right)} \from p^{-1}\left(x_{0}\right) \to \left(p'\right)^{-1}\left(x_{0}\right)$.
\end{definition}

\begin{proposition}
Let $\left(X,x_{0}\right)$ be a pointed graph.
\begin{enumerate}
    \item The functor $\mathsf{Fib}_{x_{0}}$ is representable, and it is represented by the universal cover $\tilde{X}_{x_{0}} \to X$. In particular, given any cover $Y \to X$ over $X$, the hom-set $\mathsf{Cov}\left(X\right)\left(\tilde{X}_{x_{0}},Y\right)$ is in bijection with the fiber $\mathsf{Fib}_{x_{0}}\left(Y\right)$ of $Y$ over the base-point $x_{0}$.
    
    \item The fundamental group $A_{1}\left(X,x_{0}\right)$ is isomorphic to the automorphism group $\mathsf{Aut}_{\Set^{\mathsf{Cov}\left(X\right)}}\left(\mathsf{Fib}_{x_{0}}\right)$ of the functor $\mathsf{Fib}_{x_{0}} \from \mathsf{Cov}\left(X\right) \to \Set$ in the functor category $\Set^{\mathsf{Cov}\left(X\right)}$.
    \end{enumerate}
\end{proposition}

\begin{proof}
We prove each part in turn.
\begin{enumerate}
    \item Given any $y_{0} \in \mathsf{Fib}_{x}\left(Y\right)$, we have a unique morphism of pointed coverings $\left(\tilde{X}_{x_{0}},\left[c_{x_{0}}\right]\right) \to \left(Y,y_{0}\right)$. On the other hand, given any morphism of coverings $f \from \tilde{X}_{x_{0}} \to Y$, we have a unique element $y_{0} = f\left(\left[c_{x_{0}}\right]\right)$ in $\mathsf{Fib}_{x_{0}}\left(Y\right)$.

    \item By the Yoneda lemma, we have the following bijection:
    \[
        \Set^{\mathsf{Cov}\left(X\right)}\left(\mathsf{Fib}_{x_{0}},\mathsf{Fib}_{x_{0}}\right) \cong \mathsf{Cov}\left(X\right)\left(\tilde{X}_{x_{0}},\tilde{X}_{x_{0}}\right) 
    \]
    Moreover, $\mathsf{Cov}\left(X\right)\left(\tilde{X}_{x_{0}},\tilde{X}_{x_{0}}\right)$ is in bijection with $\mathsf{Fib}_{x_{0}}\left(\tilde{X}_{x_{0}}\right)$, which is precisely the fundamental group $A_{1}\left(X,x_{0}\right)$.

    Since every morphism of coverings $\tilde{X}_{x_{0}} \to \tilde{X}_{x_{0}}$ must be invertible, we have:
    \[
        \mathsf{Aut}_{\mathsf{Cov}\left(X\right)}\left(\tilde{X}_{x_{0}}\right) = \mathsf{Cov}\left(X\right)\left(\tilde{X}_{x_{0}},\tilde{X}_{x_{0}}\right)
    \]
    and similarly, we have:
    \[
        \mathsf{Aut}_{\Set^{\mathsf{Cov}\left(X\right)}}\left(\mathsf{Fib}_{x_{0}}\right) = \Set^{\mathsf{Cov}\left(X\right)}\left(\mathsf{Fib}_{x_{0}},\mathsf{Fib}_{x_{0}}\right)
    \]
    Furthermore, we can check that all the bijections preserve the group structure and are, in fact, group isomorphisms. \qedhere
\end{enumerate}   
\end{proof}

\begin{example}
    For $n \geq 5$, the automorphism group of the universal cover $p \from I_{\infty} \to C_{n}$ is $\mathbb{Z}$, with an automorphism $i \mapsto i + nk$ for each $k \in \mathbb{Z}$.
    Thus, we have:
    \[
        \pi_{1}\left(C_{n}, \ast \right) \cong
        \begin{cases}
            0 & \text{for } n = 3, 4 \\
            \mathbb{Z} & \text{for } n \geq 5
        \end{cases}
    \]
\end{example}

\section{Seifert--Van Kampen theorem for graphs} \label{sec:kampen}

This section establishes an analogue of the familiar Seifert--van Kampen theorem from algebraic topology, a version of which was previously proven in A-homotopy theory in \cite{barcelo-kramer-laubenbacher-weaver}.
Our statement is a strengthening of the result found therein, obtained by refining the condition on the pushout square to be preserved.

Our preliminary statement of the Seifert--van Kampen theorem (\cref{thm:van-kampen}) for the fundamental groupoid contains the technical heart of the proof.
We subsequently give two more statements: \cref{thm:van-kampen-refined}, which is a refined version more convenient to use in practice, and \cref{thm:van-kampen-groups}, which specializes the theorem to the fundamental group.
We conclude this section with an example application of our refined Seifert--van Kampen \cref{thm:van-kampen-groups} to which the statement of \cite{barcelo-kramer-laubenbacher-weaver} would not have applied (\cref{ex:S2-simply-connected}).

We begin with a preliminary definition required to establish our technical assumption for the statement of the Seifert--van Kampen theorem.

\begin{definition} \leavevmode
\begin{enumerate}
    \item For any $m, n \in \mathbb{N}$, let the \emph{boundary} $\partial_{m,n} \from I_{2m+2n} \to I_{m} \gtimes I_{n}$ be the graph map given by:
    \[
        \partial_{m,n}\left(i\right) = 
        \begin{cases}
            \left(m-i,0\right) & \text{if } 0 \leq i \leq m \\
            \left(0,i-m\right) & \text{if } m \leq i \leq m+n \\
            \left(i-m-n,n\right) & \text{if } m+n \leq i \leq 2m+n \\
            \left(m,2n+2n-i\right) & \text{if } 2m+n \leq i \leq 2m+2n
        \end{cases}
    \]
\begin{figure}[h]
\centering
\begin{minipage}{0.45\textwidth}
\centering
\begin{tikzpicture}[colorstyle/.style={circle, draw=black, fill=black, thick, inner sep=0pt, minimum size=1 mm, outer sep=0pt},scale=1.5]
    \node (1) at (0,2) [colorstyle, label = above left: {$1$}, label = below right : {\tiny $(0,0)$}] {};
    \node (2) at (2,2) [colorstyle, label = above left: {$0$}, label = below right: {$4$}, label = below left : {\tiny $(1,0)$}] {};
    \node (3) at (0,0) [colorstyle, label = below left: {$2$}, label = above right : {\tiny $(0,1)$}] {};
    \node (4) at (2,0) [colorstyle, label = below right: {$3$}, label = above left : {\tiny $(1,1)$}] {};
    \draw [thick] (4) -- (3);
    \draw [thick] (3) -- (1);
    \draw [thick] (1) -- (2);
    \draw [thick] (2) -- (4);
    \node at (0.5,-1) [anchor = south]{$\partial_{1,1}$};
\end{tikzpicture}
\end{minipage}
\begin{minipage}{0.45\textwidth}
\centering
\begin{tikzpicture}[colorstyle/.style={circle, draw=black, fill=black, thick, inner sep=0pt, minimum size=1 mm, outer sep=0pt},scale=1.5]
    \node (1) at (2,2) [colorstyle, label = above left: {$0$}, label = below right: {$8$}, label = below left: {\tiny $(2,0)$}] {};
    \node (2) at (1,2) [colorstyle, label = above: {$1$}, label = below: {\tiny $(1,0)$}] {};
    \node (3) at (0,2) [colorstyle, label = above left: {$2$}, label = below right: {\tiny $(0,0)$}] {};
    \node (4) at (0,1) [colorstyle, label = left: {$3$}, label = right: {\tiny $(0,1)$}] {};
    \node (5) at (0,0) [colorstyle, label = below left: {$4$}, label = above right: {\tiny $(0,2)$}] {};
    \node (6) at (1,0) [colorstyle, label = below: {$5$}, label = above: {\tiny $(1,2)$}] {};
    \node (7) at (2,0) [colorstyle, label = below right: {$6$}, label = above left: {\tiny $(2,2)$}] {};
    \node (8) at (2,1) [colorstyle, label = right: {$7$}, label = left: {\tiny $(2,1)$}] {};
    \node at (1,1) {\tiny $(1,1)$};
    \draw [thick] (1) -- (2);
    \draw [thick] (2) -- (3);
    \draw [thick] (3) -- (4);
    \draw [thick] (4) -- (5);
    \draw [thick] (5) -- (6);
    \draw [thick] (6) -- (7);
    \draw [thick] (7) -- (8);
    \draw [thick] (8) -- (1);
    \draw [gray,thin,dotted] (2) -- (6);
    \draw [gray,thin,dotted] (4) -- (8);
    \node at (1,-1) [anchor = south]{$\partial_{2,2}$};
\end{tikzpicture}
\end{minipage}
\end{figure}
    
    \item Let $X$ be a graph and $h \from I_{1} \gtimes I_{1} \to X$ be any map. A \emph{net of $h$} is a pair $\left(H,s\right)$ consisting of a map $H \from I_{m} \gtimes I_{n} \to X$ for some $m, n \in \mathbb{N}$, together with a shrinking map $s \from I_{2m+2n} \to I_{4}$, such that the following diagram commutes:
    \[
    \begin{tikzcd}
        I_{2m+2n} \arrow[r,"\partial_{m,n}"] \arrow[d,"s"'] &
        I_{m} \gtimes I_{n} \arrow[r,"H"] &
        X \arrow[d,equal] \\
        I_{4} \arrow[r,"\partial_{1,1}"] &
        I_{1} \gtimes I_{1} \arrow[r,"h"] &
        X
    \end{tikzcd}
    \]

    \item Let $X$ be a graph and $H \from I_{m} \gtimes I_{n} \to X$ be any map. The \emph{$(i,j)$-cell of $H$}, denoted by $H_{i,j}$, is given by the following composite:
    \[
        H_{i,j} =
        \left(
        \begin{tikzcd}
            I_{1} \gtimes I_{1} \arrow[rr,"e_{i} \gtimes e_{j}"] &&
            I_{m} \gtimes I_{n} \arrow[r,"H"] &
            X
        \end{tikzcd}
        \right)
        \qquad
        \text{for } i = 1, \ldots, m \text{ and } j = 1, \ldots, n
    \]
\end{enumerate}
\end{definition}

The name `boundary map' used above might seem a little awkward in the context of chain complexes where a boundary map would go from a higher- to a lower-dimensional object.
Our choice of the name is inspired by simplicial/cubical homotopy theory, where one refers to the inclusion of the boundary of a simplex/cube as the `boundary map.'

\begin{theorem} \label{thm:van-kampen}
    Consider a pushout square in $\Graph$ as follows:
    \[
    \begin{tikzcd}
        X_{0} \arrow[r] \arrow[d] \arrow[dr,phantom,"\ulcorner" description, very near end] &
        X_{1} \arrow[d] \\
        X_{2} \arrow[r] &
        X
    \end{tikzcd}
    \]
    If every map $h \from I_{1} \gtimes I_{1} \to X$ satisfies the following net resolution condition:
    \begin{equation} \tag{N} \label{vK-cond} 
    \text{$h$ admits a net $\left(H,s\right)$ such that each cell $H_{i,j}$ of $H$ factors through $X_{1}$ or $X_{2}$}
    \end{equation}
    then the pushout square is preserved by the functor $\Pi_{1} \from \Graph \to \Gpd$. That is, we have the following pushout square in $\Gpd$:
    \[
    \begin{tikzcd}
        \Pi_{1}X_{0} \arrow[r] \arrow[d] \arrow[dr,phantom,"\ulcorner" description, very near end] &
        \Pi_{1}X_{1} \arrow[d] \\
        \Pi_{1}X_{2} \arrow[r] &
        \Pi_{1}X
    \end{tikzcd}
    \]
\end{theorem}

To prove this theorem, we will first need to characterize functors of the form $F \from \Pi_{1}X \to \cat{G}$ for any given graph $X$ and groupoid $\cat{G}$.
By definition of the fundamental groupoid (see \cref{def:fund-gpd}), a functor $F \from \Pi_{1}X \to \cat{G}$ consists of the following data: a set-function $F \from \V{X} \to \ob{\cat{G}}$ on objects, and for each pair of vertices $x,x' \in X$, a set-function $F_{x,x'} \from \pi_{0}P_{\mathbb{N}}X\left(x,x'\right) \to \cat{G}(Fx,Fx')$, subject to functoriality.
By adjointness (see \cref{prop:pi0-discrete-graph-adjunction}), the set-functions
\[
    F_{x,x'} \from \pi_{0}P_{\mathbb{N}}X\left(x,x'\right) \to \cat{G}(Fx,Fx')
\]
correspond to graph maps
\[
    F_{x,x'} \from P_{\mathbb{N}}X\left(x,x'\right) \to \cat{G}(Fx,Fx')_{\mathsf{discrete}}
\]
which in turn correspond to (see \cref{def:path-graphs}) graph maps
\[
    F_{x,x'}^{(n)} \from P_{n}X\left(x,x'\right) \to \cat{G}(Fx,Fx')_{\mathsf{discrete}}, \qquad n \in \mathbb{N}.
\]
subject to compatibility with reparametrization by shrinking maps.
Finally, since each path $\gamma \from x \rightsquigarrow x'$ in $X$ of finite length $n$ can be expressed as a concatenation of $n$ paths of length $1$ i.e. $\gamma = e_{1}^{*}(\gamma) \ast \cdots \ast e_{n}^{\ast}(\gamma)$, by functoriality, it suffices to describe the graph maps
\[
    F_{x,x'}^{(1)} \from P_{1}X\left(x,x'\right) \to \cat{G}(Fx,Fx')_{\mathsf{discrete}}.
\]
Thus, given a functor $F \from \Pi_{1}X \to \cat{G}$, we have the data of a set-function $F \from \V{X} \to \ob{\cat{G}}$ on objects, and for each pair of vertices $x,x' \in X$, the graph maps $F_{x,x'}^{(1)} \from P_{1}X\left(x,x'\right) \to \cat{G}(Fx,Fx')_{\mathsf{discrete}}$.
The following lemma tells us the conditions under which, we can build a functor $F \from \Pi_{1}X \to \cat{G}$ out of any given data of this form.

\begin{lemma} \label{lem:functors-out-of-Pi1}
    Let $X$ be a graph and $\cat{G}$ be a groupoid. Further suppose we are given the following data:
    \begin{itemize}
        \item a function $F \from \V{X} \to \ob{\cat{G}}$
        \item graph maps $F_{x,x'}^{(1)} \from P_{1}X\left(x,x'\right) \to \cat{G}\left(Fx,Fx'\right)_{\mathsf{discrete}}$, one for each pair of vertices $x,x' \in X$
    \end{itemize}
    subject to the following conditions:
    \begin{itemize}
        \item $F_{x,x}^{(1)}\left(c_{x}\right) = \text{id}_{Fx}$ for each $x \in X$
        \item $F_{x',x}^{(1)}\left(\overline{e}\right) = F_{x,x'}^{(1)}\left(e\right)^{-1}$ for each path $e \from x \rightsquigarrow x'$ of length $1$
    \end{itemize}
    Given any path $\gamma \from x \rightsquigarrow x'$ of length $n \in \mathbb{N}$, let $F_{x,x'}^{(n)}\left(\gamma\right) \in \cat{G}\left(Fx,Fx'\right)$ be given by:
    \[
        F_{x,x'}^{(n)}\left(\gamma\right) = F_{x_{n-1},x_{i}}^{(1)}\left(e_{n}^{*}\left(\gamma\right)\right) \circ \cdots \circ F_{x_{0},x_{1}}^{(1)}\left(e_{1}^{*}\left(\gamma\right)\right)
    \]
    where $x_{i} = \gamma\left(i\right)$ for $i = 0, \ldots, n$. Then, we have the following:
    \begin{enumerate}
        \item Given any path $\gamma \from x \rightsquigarrow x'$ of length $m$ and any path $\sigma \from x' \rightsquigarrow x''$ of length $n$, we have:
        \[
            F_{x,x''}^{(m+n)}\left(\gamma \ast \sigma\right) = F_{x',x''}^{(n)}\left(\sigma\right) \circ F_{x,x'}^{(m)}\left(\gamma\right)
        \]

        \item Given any path $\gamma \from x \rightsquigarrow x'$ of length $n$, we have:
        \[
            F_{x',x}^{(n)}\left(\overline{\gamma}\right) = F_{x,x'}^{(n)}\left(\gamma\right)^{-1}
        \]

        \item Given any path $\gamma \from x \rightsquigarrow x'$ of length $n$ and any shrinking map $s \from I_{m} \to I_{n}$, we have:
        \[
            F_{x,x'}^{(m)}\left(\gamma \circ s\right) = F_{x,x'}^{(n)}\left(\gamma\right)
        \]
        
        \item Given any map $H \from I_{m} \gtimes I_{n} \to X$ where each cell $H_{i,j}$ satisfies the condition:
        \[
            F^{(4)}_{x_{i,j},x_{i,j}}\left(H_{i,j} \circ \partial_{1,1}\right) = \text{id}_{Fx_{i,j}}
        \]
        where $x_{i,j} = \left(H_{i,j} \circ \partial_{1,1}\right)\left(0\right)$, we have:
        \[
            F^{(2m+2n)}_{x,x}\left(H \circ \partial_{m,n}\right) = \text{id}_{Fx},
        \]
        where $x = \left(H \circ \partial_{m,n}\right)\left(0\right)$.
        
        \item If, for each map $h \from I_{1} \gtimes I_{1} \to X$, we have
        \[
            F_{x,x}^{(4)}\left(h \circ \partial_{1,1} \right) = \text{id}_{Fx}
        \]
        where $x = \left(h \circ \partial_{1,1}\right)\left(0\right)$, then
        \[
            F_{x,x'}^{(n)} \from P_{n}X\left(x,x'\right) \to \cat{G}\left(Fx,Fx'\right)_{\mathsf{discrete}}
        \]
        is a well-defined graph map for all $n \in \mathbb{N}$ and all $x, x' \in X$.

        \item If, for each map $h \from I_{1} \gtimes I_{1} \to X$, we have
        \[
            F_{x,x}^{(4)}\left(h \circ \partial_{1,1} \right) = \text{id}_{Fx}
        \]
        where $x = \left(h \circ \partial_{1,1}\right)\left(0\right)$, then we have a well-defined functor $F \from \Pi_{1}X \to \cat{G}$ that maps $x \in X$ to $Fx \in \ob{\cat{G}}$ and a morphism $\left[\gamma\right] \from x \to x'$, where $\gamma \from x \rightsquigarrow x'$ is a path in $X$ of length $n$ to $F_{x,x'}^{(n)}\left(\gamma\right)$.
    \end{enumerate}
\end{lemma}
\begin{proof}
\begin{enumerate}
    \item This follows from the observation that:
    \[
        e_{i}^{*}\left(\gamma \ast \sigma\right) =
        \begin{cases}
            e_{i}^{*}\left(\gamma\right) & \text{for } i = 1, \ldots, m \\
            e_{i-m}^{*}\left(\sigma\right) & \text{for } i = m+1, \ldots, m+n
        \end{cases}
    \]

    \item This follows from the observation that:
    \[
        e_{i}^{*}\left(\overline{\gamma}\right) = \overline{e_{n+1-i}^{*}\left(\gamma\right)}
    \]

    \item This follows from the observation that:
    \[
        e_{i}^{*}\left(\gamma \circ s\right) =
        \begin{cases}
            c_{\left(\gamma \circ s\right)\left(i\right)} & \text{if } s\left(i\right) = s\left(i-1\right) \\
            e_{s\left(i\right)}^{*}\left(\gamma\right) & \text{if } s\left(i\right) = s\left(i-1\right) + 1
        \end{cases}
    \]

    \item Given any map $H \from I_{m} \gtimes I_{n} \to X$, we have a diagram in $\cat{G}$ as follows:
    \[
    \begin{tikzcd}
        FH\left(0,0\right) \arrow[r] \arrow[d] & FH\left(1,0\right) \arrow[r] \arrow[d] & \cdots \arrow[r] & FH\left(m-1,0\right) \arrow[r] \arrow[d] & FH\left(m,0\right) \arrow[d] \\
        FH\left(0,1\right) \arrow[r] \arrow[d] & FH\left(1,1\right) \arrow[r] \arrow[d] & \cdots \arrow[r] & FH\left(m-1,1\right) \arrow[r] \arrow[d] & FH\left(m,1\right) \arrow[d] \\
        \vdots \arrow[d] & \vdots \arrow[d]  & \ddots & \vdots \arrow[d] & \vdots \arrow[d]  \\
        FH\left(0,n-1\right) \arrow[r] \arrow[d] & FH\left(1,n-1\right) \arrow[r] \arrow[d] & \cdots \arrow[r] & FH\left(m-1,n-1\right) \arrow[r] \arrow[d] & FH\left(m,n-1\right) \arrow[d] \\
        FH\left(0,n\right) \arrow[r] & FH\left(1,n\right) \arrow[r] & \cdots \arrow[r] & FH\left(m-1,n\right) \arrow[r] & FH\left(m,n\right)
    \end{tikzcd}
    \]
    where every arrow is invertible. If any particular cell $H_{i,j} \from I_{1} \gtimes I_{1} \to X$ satisfies the condition
    \[
        F^{(4)}_{x_{i,j},x_{i,j}}\left(H_{i,j} \circ \partial_{1,1}\right) = \text{id}_{Fx_{i,j}}
    \]
    where $x_{i,j} = \left(H_{i,j} \circ \partial_{1,1}\right)\left(0\right)$,
    then it means that the corresponding square in the above diagram is commutative. But if every square in the above diagram commutes, then the outer boundary of the diagram also commutes. Hence, the condition
    \[
        F^{(2m+2n)}_{x,x}\left(H \circ \partial_{m,n}\right) = \text{id}_{Fx},
    \]
    where $x = \left(H \circ \partial_{m,n}\right)\left(0\right)$, is satisfied.
    
    \item In order to check that
    \[
        F_{x,x'}^{(n)} \from P_{n}X\left(x,x'\right) \to \cat{G}\left(Fx,Fx'\right)_{\mathsf{discrete}}
    \]
    is a well-defined graph map, we need to verify that
    \[
        F_{x,x'}^{(n)}\left(\gamma\right) = F_{x,x'}^{(n)}\left(\sigma\right)
    \]
    for any two paths $\gamma, \sigma \from x \rightsquigarrow x'$ of length $n$ that are adjacent in the graph $P_{n}X\left(x,x'\right)$.

    Let $H \from I_{n} \gtimes I_{1} \to X$ be a map given by $H\left(-,0\right) = \gamma$ and $H\left(-,1\right) = \sigma$.

    Since every cell $H_{i,j} \from I_{1} \gtimes I_{1} \to X$ satisfies the condition
    \[
        F^{(4)}_{x_{i,j},x_{i,j}}\left(H_{i,j} \circ \partial_{1,1}\right) = \text{id}_{Fx_{i,j}}
    \]
    where $x_{i,j} = \left(H_{i,j} \circ \partial_{1,1}\right)\left(0\right)$, we have:
    \[
        F_{x',x'}^{(2n+2)}\left(H \circ \partial_{n,1}\right) = \text{id}_{Fx'}
    \]
    Noting that
    \[
        H \circ \partial_{n,1} = \overline{\gamma} \ast c_{x} \ast \sigma \ast c_{x'}
    \]
    we are done.
    
    \item We have well-defined graph maps
    \[
        F_{x,x'}^{(n)} \from P_{n}X\left(x,x'\right) \to \cat{G}\left(Fx,Fx'\right)_{\mathsf{discrete}}
    \]
    that satisfy
    \[
        F_{x,x'}^{(m)}\left(\gamma \circ s\right) = F_{x,x'}^{(n)}\left(\gamma\right)
    \]
    for every path $\gamma \from x \rightsquigarrow x'$ of length $n$ and every shrinking map $s \from I_{m} \to I_{n}$.

    Thus, we have well-defined graph maps
    \[
        F_{x,x'} \from P_{\mathbb{N}}X\left(x,x'\right) \to \cat{G}\left(Fx,Fx'\right)_{\mathsf{discrete}}
    \]
    for each $x, x' \in X$. Equivalently, we have well-defined set-functions:
    \[
        F_{x,x'} \from \pi_{0}P_{\mathbb{N}}X\left(x,x'\right) \to \cat{G}\left(Fx,Fx'\right)
    \]
    Furthermore, we have $F_{x,x}\left[c_{x}\right] = \text{id}_{Fx}$ for each $x \in X$ and $F_{x,x''}\left[\gamma \ast \sigma\right] = F_{x',x''}\left[\sigma\right] \circ F_{x,x'}\left[\gamma\right]$ for each pair of paths $\gamma \from x \rightsquigarrow x'$ and $\sigma \from x' \rightsquigarrow x''$.

    Thus, we have a well-defined functor $F \from \Pi_{1}X \to \cat{G}$.
\end{enumerate}
\end{proof}

We can now prove \cref{thm:van-kampen}.

\begin{proof}[Proof of \cref{thm:van-kampen}]
    Let $\cat{G}$ be any groupoid, and let $F_{1} \from \Pi_{1}X_{1} \to \cat{G}$ and $F_{2} \from \Pi_{1}X_{2} \to \cat{G}$ be functors such that the outer square in the following diagram commutes:
    \[
    \begin{tikzcd}
        \Pi_{1}X_{0} \arrow[r] \arrow[d] &
        \Pi_{1}X_{1} \arrow[d] \arrow[ddr, bend left = 50,"F_{1}"] & \\
        \Pi_{1}X_{2} \arrow[r] \arrow[drr, bend right = 40,"F_{2}"] &
        \Pi_{1}X \arrow[dr, dotted,"F" description] & \\
        & & \cat{G}
    \end{tikzcd}
    \]
    We want to show that there exists a unique functor $F \from \Pi_{1}X \to \cat{G}$ such that the above diagram commutes.

    By \cref{lem:functors-out-of-Pi1}, it suffices to provide the following data:
    \begin{itemize}
        \item[$\bullet$] a function $F \from \V{X} \to \ob{\cat{G}}$
        \item[$\bullet$] graph maps $F_{x,x'}^{(1)} \from P_{1}X\left(x,x'\right) \to \cat{G}\left(Fx,Fx'\right)_{\mathsf{discrete}}$, one for each pair of vertices $x,x' \in X$
    \end{itemize}
    subject to the following conditions:
    \begin{itemize}
        \item[$\bullet$] $F_{x,x}^{(1)}\left(\mathrm{c}_{x}\right) = \mathrm{id}_{Fx}$ for each $x \in X$
        \item[$\bullet$] $F_{x',x}^{(1)}\left(\overline{e}\right) = F_{x,x'}^{(1)}\left(e\right)^{-1}$ for each path $e \from x \rightsquigarrow x'$ of length $1$
        \item[$\bullet$] for each map $h \from I_{1} \gtimes I_{1} \to X$, we have
        \[
            F_{x,x}^{(4)}\left(h \circ \partial_{1,1} \right) = \mathrm{id}_{Fx}
        \]
        where $x = \left(h \circ \partial_{1,1}\right)\left(0\right)$
    \end{itemize}
    
    The function $F \from \V{X} \to \ob{\cat{G}}$ is uniquely determined by $F_{1}$ and $F_{2}$ since we have the following pushout square in $\Set$:
    \[
    \begin{tikzcd}
        \V{(X_{0})} \arrow[r] \arrow[d] \arrow[dr,phantom,"\ulcorner" description, very near end] &
        \V{(X_{1})} \arrow[d] \\
        \V{(X_{2})} \arrow[r] &
        \V{X}
    \end{tikzcd}
    \]

    Similarly, since every path of length $1$ in $X$ factors through either $X_{1}$ or $X_{2}$, the functions $F_{x,x'}^{(1)} \from P_{1}X\left(x,x'\right) \to \cat{G}\left(Fx,Fx'\right)$ are also uniquely determined by $F_{1}$ and $F_{2}$.

    The first two conditions can also be easily verified.
    It only remains to verify the third condition.

    By the hypothesis, every map $h \from I_{1} \gtimes I_{1} \to X$ admits a net $\left(H,s\right)$ such that each cell $H_{i,j}$ of $H$ factors through either $X_{1}$ or $X_{2}$.
    It follows that each cell $H_{i,j}$ satisfies the condition
    \[
        F^{(4)}_{x_{i,j},x_{i,j}}\left(H_{i,j} \circ \partial_{1,1}\right) = \mathrm{id}_{Fx_{i,j}}
    \]
    where $x_{i,j} = \left(H_{i,j} \circ \partial_{1,1}\right)\left(0\right)$.
    By part (4) of \cref{lem:functors-out-of-Pi1}, it follows that we have:
    \[
        F^{(2m+2n)}_{x,x}\left(H \circ \partial_{m,n}\right) = \mathrm{id}_{Fx},
    \]
    where $x = \left(H \circ \partial_{m,n}\right)\left(0\right)$.
    But since $\left(H,s\right)$ is a net of $h$, it follows that:
    \[
        H \circ \partial_{m,n} = h \circ \partial_{1,1} \circ s 
    \]
    Thus, we have $F^{(2m+2n)}_{x,x}\left(h \circ \partial_{1,1} \circ s\right) = \mathrm{id}_{Fx}$.
    By part (3) of \cref{lem:functors-out-of-Pi1}, it follows that $F_{x,x}^{(4)}\left(h \circ \partial_{1,1} \right) = \mathrm{id}_{Fx}$, as required.
\end{proof}

In practice, one does not need to check the net resolution condition (\ref{vK-cond}) of \cref{thm:van-kampen} for all maps $I_{1} \gtimes I_{1} \to X$, since for many of them it is trivially satisfied.
In \cref{thm:van-kampen-refined}, we give a refined version of \cref{thm:van-kampen} that only requires checking the possibly problematic situation.

The table below gives all possible shapes that the image of a map $h \from I_{1} \gtimes I_{1} \to X$ can take. For convenience, we will name the vertices as follows:
\begin{figure}[H]
\centering
\begin{tikzpicture}[colorstyle/.style={circle, draw=black, fill=black, thick, inner sep=0pt, minimum size=1 mm, outer sep=0pt},scale=1.5]
    \node (1) at (0,1) [colorstyle, label = above left: {$a = h\left(0,0\right)$}] {};
    \node (2) at (1,1) [colorstyle, label = above right: {$b = h\left(1,0\right)$}] {};
    \node (3) at (0,0) [colorstyle, label = below left: {$c = h\left(0,1\right)$}] {};
    \node (4) at (1,0) [colorstyle, label = below right: {$d = h\left(1,1\right)$}] {};
    \draw [thick] (4) -- (3);
    \draw [thick] (3) -- (1);
    \draw [thick] (1) -- (2);
    \draw [thick] (2) -- (4);
\end{tikzpicture}
\end{figure}

Then, up to symmetry of $I_{1} \gtimes I_{1}$, a map $h \from I_{1} \gtimes I_{1} \to X$ can take the following distinct shapes:

\begin{longtable}{| c | c | c |} 
    \hline
    Image & Diagram & Conditions on $a, b, c, d$ \\
    \hline
    \hline
    \text{single vertex} &
    \begin{tikzpicture}[colorstyle/.style={circle, draw=black, fill=black, thick, inner sep=0pt, minimum size=1 mm, outer sep=0pt},scale=1.5]
        \node (0) at (0,0) [colorstyle, label = above: {~}, label = below: {~}] {};
    \end{tikzpicture}
    & $a = b = c = d$ \\
    \hline
    \hline
    \text{single edge} &
    \begin{tikzpicture}[colorstyle/.style={circle, draw=black, fill=black, thick, inner sep=0pt, minimum size=1 mm, outer sep=0pt},scale=1.5]
        \node (0) at (0,0) [colorstyle, label = above: {~}, label = below: {~}] {};
        \node (1) at (1,0) [colorstyle, label = above: {~}, label = below: {~}] {};
        \draw [thick] (0) -- (1);
    \end{tikzpicture}
    & $a = c$ and $b = d$ \\
    \hline
    \hline
    \text{single edge} &
    \begin{tikzpicture}[colorstyle/.style={circle, draw=black, fill=black, thick, inner sep=0pt, minimum size=1 mm, outer sep=0pt},scale=1.5]
        \node (0) at (0,0) [colorstyle, label = above: {~}, label = below: {~}] {};
        \node (1) at (1,0) [colorstyle, label = above: {~}, label = below: {~}] {};
        \draw [thick] (0) -- (1);
    \end{tikzpicture}
    & $a = d$ and $b = c$ \\
    \hline
    \hline
    \text{single edge} &
    \begin{tikzpicture}[colorstyle/.style={circle, draw=black, fill=black, thick, inner sep=0pt, minimum size=1 mm, outer sep=0pt},scale=1.5]
        \node (0) at (0,0) [colorstyle, label = above: {~}, label = below: {~}] {};
        \node (1) at (1,0) [colorstyle, label = above: {~}, label = below: {~}] {};
        \draw [thick] (0) -- (1);
    \end{tikzpicture}
    & $a = b = c$ \\
    \hline
    \hline
    \text{pair of edges} &
    \begin{tikzpicture}[colorstyle/.style={circle, draw=black, fill=black, thick, inner sep=0pt, minimum size=1 mm, outer sep=0pt},scale=1.5]
        \node (0) at (0.25,0) [colorstyle, label = above: {~}, label = below: {~}] {};
        \node (1) at (1,0) [colorstyle, label = above: {~}, label = below: {~}] {};
        \node (2) at (1.75,0) [colorstyle, label = above: {~}, label = below: {~}] {};
        \draw [thick] (0) -- (1);
        \draw [thick] (1) -- (2);
    \end{tikzpicture}
    & $a = d$ \\ \hline \hline
    \text{a 3-cycle} &
    \begin{tikzpicture}[colorstyle/.style={circle, draw=black, fill=black, thick, inner sep=0pt, minimum size=1 mm, outer sep=0pt},scale=1.5]
        \node (0) at (0.5,0.866) [colorstyle, label = above: {~}, label = below: {~}] {};
        \node (1) at (1,0) [colorstyle, label = above: {~}, label = below: {~}] {};
        \node (2) at (0,0) [colorstyle, label = above: {~}, label = below: {~}] {};
        \draw [thick] (0) -- (1);
        \draw [thick] (1) -- (2);
        \draw [thick] (2) -- (0);
    \end{tikzpicture}
    & $a = b$ \\ \hline \hline
    \text{a 4-cycle} &
    \begin{tikzpicture}[colorstyle/.style={circle, draw=black, fill=black, thick, inner sep=0pt, minimum size=1 mm, outer sep=0pt},scale=1.5]
        \node (0) at (0,1) [colorstyle, label = above: {~}, label = below: {~}] {};
        \node (1) at (1,1) [colorstyle, label = above: {~}, label = below: {~}] {};
        \node (2) at (1,0) [colorstyle, label = above: {~}, label = below: {~}] {};
        \node (3) at (0,0) [colorstyle, label = above: {~}, label = below: {~}] {};
        \draw [thick] (0) -- (1);
        \draw [thick] (1) -- (2);
        \draw [thick] (2) -- (3);
        \draw [thick] (3) -- (0);
    \end{tikzpicture}
    & $\varnothing$ \\ \hline
\end{longtable}

Note that whenever the map $h \from I_{1} \gtimes I_{1} \to X$ itself factors through one of $X_{1}$ or $X_{2}$, it trivially satisfies the net resolution condition (\ref{vK-cond}).
This applies to all maps $h \from I_{1} \gtimes I_{1} \to X$ whose image in $X$ is either a single vertex or a single edge.

Let us also consider the case when the image of a map $h \from I_{1} \gtimes I_{1} \to X$ is a pair of edges. If $h$ does not itself factor through either $X_{1}$ or $X_{2}$, then we must have that one of the edges factors through $X_{1}$ while the other factors through $X_{2}$. We can visualize such a map as follows:
\begin{figure}[H]
\centering
\begin{tikzpicture}[colorstyle/.style={circle, draw=black, fill=black, thick, inner sep=0pt, minimum size=1 mm, outer sep=0pt},scale=1.5]
    \node (1) at (0,1) [colorstyle, label = above left: {$a$}] {};
    \node (2) at (1,1) [colorstyle, label = above right: {$b$}] {};
    \node (3) at (0,0) [colorstyle, label = below left: {$c$}] {};
    \node (4) at (1,0) [colorstyle, label = below right: {$a$}] {};
    \draw [thick] (4) -- (3);
    \draw [thick] (3) -- (1);
    \draw [thick] (1) -- (2);
    \draw [thick] (2) -- (4);
    \node at (0.5,-0.8) [anchor = south]{a pair of edges in $X$, with $a \sim b$ in $X_{1}$ and $a \sim c$ in $X_{2}$};
\end{tikzpicture}
\end{figure}

Even so, $h$ satisfies the condition (\ref{vK-cond}), with a net $H \from I_{2} \gtimes I_{2} \to X$ that can be visualized as follows:
\begin{figure}[H]
\centering
\begin{tikzpicture}[colorstyle/.style={circle, draw=black, fill=black, thick, inner sep=0pt, minimum size=1 mm, outer sep=0pt},scale=1.5]
    \node (1) at (0,2) [colorstyle, label = above left: {$a$}] {};
    \node (2) at (1,2) [colorstyle, label = above: {$b$}] {};
    \node (3) at (2,2) [colorstyle, label = above right: {$b$}] {};
    \node (4) at (0,1) [colorstyle, label = left: {$a$}] {};
    \node (5) at (1,1) [colorstyle, label = above right: {$a$}] {};
    \node (6) at (2,1) [colorstyle, label = right: {$a$}] {};
    \node (7) at (0,0) [colorstyle, label = below left: {$c$}] {};
    \node (8) at (1,0) [colorstyle, label = below: {$c$}] {};
    \node (9) at (2,0) [colorstyle, label = below right: {$a$}] {};
    \draw [thick] (1) -- (3);
    \draw [thick] (4) -- (6);
    \draw [thick] (7) -- (9);
    \draw [thick] (1) -- (7);
    \draw [thick] (2) -- (8);
    \draw [thick] (3) -- (9);
    \node at (1,-0.8) [anchor = south]{a net for $h$};
\end{tikzpicture}
\end{figure}

Thus, it suffices to check the net resolution condition (\ref{vK-cond}) for those maps $h \from I_{1} \gtimes I_{1} \to X$ whose image is a 3- or 4-cycle in $X$.

\begin{theorem}[Seifert--van Kampen Theorem for Fundamental Groupoid]\label{thm:van-kampen-refined}
    Consider a pushout square in $\Graph$ as follows:
    \[
    \begin{tikzcd}
        X_{0} \arrow[r] \arrow[d] \arrow[dr,phantom,"\ulcorner" description, very near end] &
        X_{1} \arrow[d] \\
        X_{2} \arrow[r] &
        X
    \end{tikzcd}
    \]
    If every map $h \from I_{1} \gtimes I_{1} \to X$ whose image in $X$ is a 3- or 4-cycle satisfies the following net resolution condition:
    \begin{equation} \tag{N} 
    \text{$h$ admits a net $\left(H,s\right)$ such that each cell $H_{i,j}$ of $H$ factors through $X_{1}$ or $X_{2}$}
    \end{equation}
    then the pushout square is preserved by the functor $\Pi_{1} \from \Graph \to \Gpd$. \qed
\end{theorem}

\begin{theorem}[Seifert--van Kampen Theorem for Fundamental Group]\label{thm:van-kampen-groups}
    Consider a pushout square of pointed connected graphs as follows:
    \[
    \begin{tikzcd}
        X_{0} \arrow[r] \arrow[d] \arrow[dr,phantom,"\ulcorner" description, very near end] &
        X_{1} \arrow[d] \\
        X_{2} \arrow[r] &
        X
    \end{tikzcd}
    \]
    If every map $h \from I_{1} \gtimes I_{1} \to X$ whose image in $X$ is a 3- or 4-cycle satisfies the following net resolution condition:
    \begin{equation} \tag{N}
    \text{$h$ admits a net $\left(H,s\right)$ such that each cell $H_{i,j}$ of $H$ factors through $X_{1}$ or $X_{2}$}
    \end{equation}
    then we have the following pushout square in $\Grp$:
    \[
    \pushQED{\qed}
    \begin{tikzcd}
        A_{1}\left(X_{0},x_0\right) \arrow[r] \arrow[d] \arrow[dr,phantom,"\ulcorner" description, very near end] &
        A_{1}\left(X_{1},x_1\right) \arrow[d] \\
        A_{1}\left(X_{2},x_2\right) \arrow[r] &
        A_{1}\left(X,x\right)
    \end{tikzcd}
    \qedhere
    \popQED
    \]
\end{theorem}

\begin{example}
    For any $m \in \mathbb{N}$ and $n \geq 5$, the fundamental group of the graph $\bigvee_{i=1}^{m}{C_{n}}$ is isomorphic to the free group of rank $m$. We can prove this by induction on $m$ and by applying \cref{thm:van-kampen-groups} to the following pushout diagram, since $\bigvee_{i=1}^{m}{C_{n}}$ has no 3- or 4-cycles:
    \[
    \begin{tikzcd}
        I_{0} \arrow[r,"0"] \arrow[d,"0"'] \arrow[dr,phantom,"\ulcorner" description, very near end] & \bigvee_{i=1}^{m-1}{C_{n}} \arrow[d] \\
        C_{n} \arrow[r] & \bigvee_{i=1}^{m}{C_{n}}
    \end{tikzcd}
    \]    
\end{example}

\begin{example} \label{ex:S2-simply-connected}
    Note that the boundary map $\partial_{m,n} \from I_{2m+2n} \to I_{m} \gtimes I_{n}$ satisfies $\partial_{m,n}\left(0\right) = \partial_{m,n}\left(2m+2n\right)$ and hence descends to a \emph{boundary inclusion} $\partial_{m,n} \from C_{2m+2n} \hookrightarrow I_{m} \gtimes I_{n}$.
    Let $Y_{m,n}$ denote the graph given by the following pushout:
    \[
    \begin{tikzcd}
        C_{2m+2n} \arrow[r,"\partial_{m,n}",hook] \arrow[d,"\partial_{m,n}"',hook] \arrow[dr,phantom,"\ulcorner" description, very near end] & I_{m} \gtimes I_{n} \arrow[d] \\
        I_{m} \gtimes I_{n} \arrow[r] & Y_{m,n}
    \end{tikzcd}
    \]
    For $m, n \geq 3$, the graph $Y_{m,n}$ has precisely four 4-cycles that are not contained entirely in either copy of $I_{m} \gtimes I_{n}$.
    However, each of these 4-cycles satisfies the net resolution condition (\ref{vK-cond}) via a nontrivial net.
    The figure below showcases a nontrivial net for one of these 4-cycles in $Y_{3,3}$.
    \begin{figure}[H]
    \centering
    \begin{tikzpicture}[colorstyle/.style={circle, draw=black, fill=black, thick, inner sep=0pt, minimum size=1 mm, outer sep=0pt},scale=1.5]
    \node (0) at (0,0) [colorstyle, label = below left: {$6$}] {};
    \node (1) at (1,0) [colorstyle, label = below: {$7$}] {};
    \node (2) at (2,0) [colorstyle, label = below: {$8$}] {};
    \node (3) at (3,0) [colorstyle, label = below right: {$9$}] {};
    \node (4) at (0,1) [colorstyle, label = left: {$5$}] {};
    \node (5) at (1,1) [colorstyle, label = above right: {$c$}] {};
    \node (6) at (2,1) [colorstyle, label = above right: {$d$}] {};
    \node (7) at (3,1) [colorstyle, label = right: {$10$}] {};
    \node (8) at (0,2) [colorstyle, label = left: {$4$}] {};
    \node (9) at (1,2) [colorstyle, label = above right: {$b$}] {};
    \node (10) at (2,2) [colorstyle, label = above right: {$a$}] {};
    \node (11) at (3,2) [colorstyle, label = right: {$11$}] {};
    \node (12) at (0,3) [colorstyle, label = above left: {$3$}] {};
    \node (13) at (1,3) [colorstyle, label = above: {$2$}] {};
    \node (14) at (2,3) [colorstyle, label = above: {$1$}] {};
    \node (15) at (3,3) [colorstyle, label = above right: {$0$}] {};

    \draw [thick] (0) -- (3);
    \draw [thick] (4) -- (7);
    \draw [thick] (8) -- (11);
    \draw [thick] (12) -- (15);
    \draw [thick] (0) -- (12);
    \draw [thick] (1) -- (13);
    \draw [thick] (2) -- (14);
    \draw [thick] (3) -- (15);
    \draw [very thick,red] (10) -- (11);
    \draw [very thick,red] (10) -- (14);

    \node (0') at (5,0) [colorstyle, label = below left: {$6$}] {};
    \node (1') at (6,0) [colorstyle, label = below: {$7$}] {};
    \node (2') at (7,0) [colorstyle, label = below: {$8$}] {};
    \node (3') at (8,0) [colorstyle, label = below right: {$9$}] {};
    \node (4') at (5,1) [colorstyle, label = left: {$5$}] {};
    \node (5') at (6,1) [colorstyle, label = above right: {$c'$}] {};
    \node (6') at (7,1) [colorstyle, label = above right: {$d'$}] {};
    \node (7') at (8,1) [colorstyle, label = right: {$10$}] {};
    \node (8') at (5,2) [colorstyle, label = left: {$4$}] {};
    \node (9') at (6,2) [colorstyle, label = above right: {$b'$}] {};
    \node (10') at (7,2) [colorstyle, label = above right: {$a'$}] {};
    \node (11') at (8,2) [colorstyle, label = right: {$11$}] {};
    \node (12') at (5,3) [colorstyle, label = above left: {$3$}] {};
    \node (13') at (6,3) [colorstyle, label = above: {$2$}] {};
    \node (14') at (7,3) [colorstyle, label = above: {$1$}] {};
    \node (15') at (8,3) [colorstyle, label = above right: {$0$}] {};

    \draw [thick] (0') -- (3');
    \draw [thick] (4') -- (7');
    \draw [thick] (8') -- (11');
    \draw [thick] (12') -- (15');
    \draw [thick] (0') -- (12');
    \draw [thick] (1') -- (13');
    \draw [thick] (2') -- (14');
    \draw [thick] (3') -- (15');
    \draw [very thick,red] (10') -- (11');
    \draw [very thick,red] (10') -- (14');
    \node at (4,-1) [anchor = south]{The graph $Y_{3,3}$ - each pair of identically labelled vertices is identified.};
    \end{tikzpicture}
    \end{figure}
    \begin{figure}[H]
    \centering
    \begin{tikzpicture}[colorstyle/.style={circle, draw=black, fill=black, thick, inner sep=0pt, minimum size=1 mm, outer sep=0pt},scale=1.5]
    \node (1) at (0,2) [colorstyle, label = above left: {$1$}] {};
    \node (2) at (1,2) [colorstyle, label = above: {$1$}] {};
    \node (3) at (2,2) [colorstyle, label = above right: {$a$}] {};
    \node (4) at (0,1) [colorstyle, label = left: {$1$}] {};
    \node (5) at (1,1) [colorstyle, label = above right: {$0$}] {};
    \node (6) at (2,1) [colorstyle, label = right: {$11$}] {};
    \node (7) at (0,0) [colorstyle, label = below left: {$a'$}] {};
    \node (8) at (1,0) [colorstyle, label = below: {$11$}] {};
    \node (9) at (2,0) [colorstyle, label = below right: {$11$}] {};
    \draw [thick] (1) -- (3);
    \draw [thick] (4) -- (6);
    \draw [thick] (7) -- (9);
    \draw [thick] (1) -- (7);
    \draw [thick] (2) -- (8);
    \draw [thick] (3) -- (9);
    \node at (1,-1) [anchor = south]{a net for the highlighted 4-cycle in $Y_{3,3}$};
    \end{tikzpicture}
    \end{figure}

    Thus, we can apply \cref{thm:van-kampen-groups} to conclude that $Y_{m,n}$ is simply connected for $m, n \geq 3$.
\end{example}

\section{Application: graphs with prescribed fundamental group} \label{sec:application}

In this final section, we present a construction that applies the Seifert--van Kampen \cref{thm:van-kampen-groups} to produce an example of a graph with a prescribed fundamental group.
After preliminary definitions of cones and disks, we give the construction in \cref{def:graph-with-A1} and prove its correctness in \cref{thm:prescribed-A1-correct}.

We note that the problem can also be approached topologically by taking a sufficiently fine quadrangulation of a manifold with a known fundamental group, cf.~\cite[p.~123]{barcelo-kramer-laubenbacher-weaver}, although, unlike ours, this process is not algorithmic.
For example, it is unclear what quadrangulation would be considered sufficiently fine.

\begin{definition}
The \emph{cone $\mathsf{Cone}_{n}\left(X\right)$} of height $n \in \mathbb{N}$ on a graph $X$ is given by the following pushout:
\[
\begin{tikzcd}
    X \gtimes I_{0} \arrow[r,"\text{id} \gtimes 0"] \arrow[d,"!"'] \arrow[dr,phantom,"\ulcorner" description, very near end] & X \gtimes I_{n} \arrow[d] \\
    I_{0} \arrow[r] & \mathsf{Cone}_{n}\left(X\right)
\end{tikzcd}
\]
Let $\ast$ denote the identification class of $\left(x,0\right)$ in $\mathsf{Cone}_{n}\left(X\right)$ for all $x \in X$.
\end{definition}

\begin{proposition}
    Given any graph $X$ and $n \in \mathbb{N}$, the cone $\mathsf{Cone}_{n}\left(X\right)$ is contractible.
\end{proposition}
\begin{proof}
    Consider the graph map $\tilde{H} \from X \gtimes I_{n} \gtimes I_{n} \to \mathsf{Cone}_{n}\left(X\right)$ given by the formula $\tilde{H}\left(x,i,j\right) = \left[x,\min{\left(i,n-j\right)}\right]$.
    Since $\tilde{H}\left(x,0,j\right) = \ast$ for all $x \in X$ and $j \in I_{n}$, it descends to a map  $H \from \mathsf{Cone}_{n}\left(X\right) \gtimes I_{n} \to \mathsf{Cone}_{n}\left(X\right)$.
    Note that $H\left(-,-,0\right) = \text{id}_{\mathsf{Cone}_{n}\left(X\right)}$ and $H\left(-,-,n\right) = c_{\ast}$.
    Thus, the graph $\mathsf{Cone}_{n}\left(X\right)$ is contractible.
\end{proof}

\begin{definition}
    For $m, n \in \mathbb{N}$, the \emph{$(m,n)-$disk} $D_{m,n}$, with $m$ spokes and of radius $n$, is given by $\mathsf{Cone}_{n}\left(C_{m}\right)$.
    The \emph{boundary of the $\left(m,n\right)$-disk} is the map $\partial_{m,n} \from C_{m} \to D_{m,n}$ given by $ \partial_{m,n}\left(i\right) = \left(i,n\right)$.
\end{definition}

\begin{figure}[H]
\centering
\begin{minipage}{0.45\textwidth}
\centering
\begin{tikzpicture}[colorstyle/.style={circle, draw=black, fill=black, thick, inner sep=0pt, minimum size=1.2 mm, outer sep=0pt},scale=1.5]
    \draw [thick] (0,0) circle (1.2cm);
    \node (1) at (0,2.4/2) [colorstyle] {};
    \node (2) at (2.28254/2, 0.741642/2) [colorstyle] {};
    \node (3) at (1.41068/2, -1.94164/2) [colorstyle] {};
    \node (4) at (-1.41068/2, -1.94164/2) [colorstyle] {};
    \node (5) at (-2.28254/2, 0.741642/2) [colorstyle] {};
    
    \node at (0,-2) [anchor = south]{$C_{5}$};
\end{tikzpicture}
\end{minipage}
\begin{minipage}{0.45\textwidth}
\centering
\begin{tikzpicture}[colorstyle/.style={circle, draw=black, fill=black, thick, inner sep=0pt, minimum size=1.5 mm, outer sep=0pt},scale=1.5]
    \node (0) at (0,0) [colorstyle] {};
    
    \draw [thick] (0,0) circle (0.4cm);
    \node at (0,0.8/2) [colorstyle] {};
    \node at (0.760846/2, 0.247214/2) [colorstyle] {};
    \node at (0.470228/2, -0.647214/2) [colorstyle] {};
    \node at (-0.470228/2, -0.647214/2) [colorstyle] {};
    \node at (-0.760846/2, 0.247214/2) [colorstyle] {};
    
    \draw  [thick] (0,0) circle (0.8cm);
    \node at (0,1.6/2) [colorstyle] {};
    \node at (1.52169/2, 0.494428/2) [colorstyle] {};
    \node at (0.940456/2, -1.29443/2) [colorstyle] {};
    \node at (-0.940456/2, -1.29443/2) [colorstyle] {};
    \node at (-1.52169/2, 0.494428/2) [colorstyle] {};
    
    \draw [thick] (0,0) circle (1.2cm);
    \node (1) at (0,2.4/2) [colorstyle] {};
    \node (2) at (2.28254/2, 0.741642/2) [colorstyle] {};
    \node (3) at (1.41068/2, -1.94164/2) [colorstyle] {};
    \node (4) at (-1.41068/2, -1.94164/2) [colorstyle] {};
    \node (5) at (-2.28254/2, 0.741642/2) [colorstyle] {};

    \draw [thick] (0) -- (1);
    \draw [thick] (0) -- (2);
    \draw [thick] (0) -- (3);
    \draw [thick] (0) -- (4);
    \draw [thick] (0) -- (5);
    
    \node at (0,-2) [anchor = south]{$D_{5,3}$};
\end{tikzpicture}
\end{minipage}
\end{figure}

\begin{construction} \label{def:graph-with-A1}
    Let $F_{S}$ be the free group generated by a set $S$.
    Given any word $r = s_{1}^{d_{1}} \cdots s_{k}^{d_{k}} \in F_{S}$, with $s_{i} \neq s_{i+1}$ for $1 \leq i \leq k-1$, we can define its \emph{degree} as follows:
    \[
        \mathsf{deg}\left(r\right) = \left\lvert d_{1} \right\rvert + \cdots + \left\lvert d_{k} \right\rvert.
    \]
    We can then define a map
    \[
        \omega_{r} = C_{5 \cdot \mathsf{deg}\left(r\right)} \longrightarrow \bigvee_{s \in S}{C_{5}}
    \]
    corresponding the word $r$ as follows: first, wrap $d_{1}$ times around the $C_{5}$ corresponding to $s_{1} \in S$ (clockwise if $d_{1} > 0$ and counter-clockwise if $d_{1} < 0$), then $d_{2}$ times around the $C_{5}$ corresponding to $s_{2} \in S$, and so on.
    
    Given any group $G$, with a presentation
    \[
        G = \left\langle S \ \middle\vert \ R\right\rangle
    \]
    let the graph $X_{S, R}$ be given by following pushout:
    \[
    \begin{tikzcd}[row sep = large, column sep = large]
        \coprod_{r \in R}{C_{5 \cdot \mathsf{deg}\left(r\right)}} \arrow[r,"{(\omega_{r})_{r \in R}}"] \arrow[d,"\coprod_{r \in R}{\partial_{5 \cdot \mathsf{deg}\left(r\right),3}}"'] \arrow[dr, phantom, "\ulcorner" description, very near end] & \bigvee_{s \in S}{C_{5}} \arrow[d] \\
        \coprod_{r \in R}{D_{5 \cdot \mathsf{deg}\left(r\right),3}} \arrow[r] & X_{S, R}
    \end{tikzcd}
    \]
\end{construction}

\begin{theorem} \label{thm:prescribed-A1-correct}
    Given any group $G$, with a presentation
    \[
        G = \left\langle S \ \middle\vert \ R\right\rangle = \left\langle s_{1}, \ldots, s_{m} \ \middle\vert \ r_{1}, \ldots, r_{n}\right\rangle,
    \]
    we have:
    \[
        A_{1}\left(X_{S, R}, x\right) \cong G
    \]
    for any $x \in X_{S, R}$.
\end{theorem}
\begin{proof}
    The key observation is that every 3-cycle or 4-cycle in $X_{S, R}$ necessarily factors through one of the disks $D_{5 \cdot \mathsf{deg}\left(r_{j}\right),3}$.
    Thus, we can apply \cref{thm:van-kampen-groups} to get the required result.
\end{proof}

\begin{example}
    Taking $G = \mathbb{Z}/2\mathbb{Z}$, with the presentation $\left\langle a \ \middle\vert \ a^{2} \right\rangle$, the graph $X_{\{a\},\{a^{2}\}}$ is given by the following pushout:
    \[
    \begin{tikzcd}
        C_{10} \arrow[r,"\omega_{a^{2}}"] \arrow[d,"\partial_{10,3}"'] \arrow[dr, phantom, "\ulcorner" description, very near end] & C_{5} \arrow[d] \\
        D_{10,3} \arrow[r] & X_{\{a\},\{a^{2}\}}
    \end{tikzcd}
    \]
    We can visualize this graph as follows, with each pair of identically labelled vertices identified:
\begin{figure}[H]
\centering
\begin{tikzpicture}[colorstyle/.style={circle, draw=black, fill=black, thick, inner sep=0pt, minimum size=1.2 mm, outer sep=0pt},scale=1.5]
    \node (0) at (0,0) [colorstyle] {};
    
    \draw [thick] (0,0) circle (0.4cm);
    \node at (0,0.8/2) [colorstyle] {};
    \node at (0,-0.8/2) [colorstyle] {};
    \node at (0.760846/2, 0.247214/2) [colorstyle] {};
    \node at (0.760846/2, -0.247214/2) [colorstyle] {};
    \node at (0.470228/2, -0.647214/2) [colorstyle] {};
    \node at (0.470228/2, 0.647214/2) [colorstyle] {};
    \node at (-0.470228/2, -0.647214/2) [colorstyle] {};
    \node at (-0.470228/2, 0.647214/2) [colorstyle] {};
    \node at (-0.760846/2, 0.247214/2) [colorstyle] {};
    \node at (-0.760846/2, -0.247214/2) [colorstyle] {};
    
    \draw  [thick] (0,0) circle (0.8cm);
    \node at (0,1.6/2) [colorstyle] {};
    \node at (1.52169/2, 0.494428/2) [colorstyle] {};
    \node at (0.940456/2, -1.29443/2) [colorstyle] {};
    \node at (-0.940456/2, -1.29443/2) [colorstyle] {};
    \node at (-1.52169/2, 0.494428/2) [colorstyle] {};
    \node at (0,-1.6/2) [colorstyle] {};
    \node at (1.52169/2, -0.494428/2) [colorstyle] {};
    \node at (0.940456/2, 1.29443/2) [colorstyle] {};
    \node at (-0.940456/2, 1.29443/2) [colorstyle] {};
    \node at (-1.52169/2, -0.494428/2) [colorstyle] {};
    
    \draw [thick] (0,0) circle (1.2cm);
    \node (1) at (0,2.4/2) [colorstyle, label = above: {$0$}] {};
    \node (2) at (2.28254/2, 0.741642/2) [colorstyle, label = above right: {$2$}] {};
    \node (3) at (1.41068/2, -1.94164/2) [colorstyle, label = below right: {$4$}] {};
    \node (4) at (-1.41068/2, -1.94164/2) [colorstyle, label = below left: {$1$}] {};
    \node (5) at (-2.28254/2, 0.741642/2) [colorstyle, label = above left: {$3$}] {};
    \node (6) at (0,-2.4/2) [colorstyle, label = below: {$0$}] {};
    \node (7) at (2.28254/2, -0.741642/2) [colorstyle, label = below right: {$3$}] {};
    \node (8) at (1.41068/2, 1.94164/2) [colorstyle, label = above right: {$1$}] {};
    \node (9) at (-1.41068/2, 1.94164/2) [colorstyle, label = above left: {$4$}] {};
    \node (10) at (-2.28254/2, -0.741642/2) [colorstyle, label = below left: {$2$}] {};

    \draw [thick] (1) -- (6);
    \draw [thick] (2) -- (10);
    \draw [thick] (3) -- (9);
    \draw [thick] (4) -- (8);
    \draw [thick] (5) -- (7);
    
    \node at (0,-2) [anchor = south]{The graph $X_{\{a\},\{a^{2}\}}$ with fundamental group $\mathbb{Z}/2\mathbb{Z}$};
\end{tikzpicture}
\end{figure}
\end{example}

We conclude the paper with three open questions:
\begin{enumerate}
    \item Can the construction of a graph $X_G$ be modified to a construction of Eilenberg--Mac Lane graphs, i.e., graphs with prescribed homotopy group in a fixed degree and other homotopy groups trivial?
    \item Can such Eilenberg--Mac Lane graphs be used to define graph cohomology in a way that is analogous to the classical theory (i.e., via Brown's representability theorem)?
    \item Is the category of graphs localized at the class of $1$-equivalences (i.e., graph maps inducing isomorphisms on $A_0$ and $A_1$) equivalent, via the fundamental groupoid functor to the category of groupoids? This is known for topological spaces.
\end{enumerate}

\bibliographystyle{amsalphaurlmod}
\bibliography{disc-fund-gpd}

\end{document}